\documentclass[11pt, reqno]{amsart}
\usepackage{a4,amsmath,amssymb,amscd,verbatim,bbm,graphicx,enumerate, blkarray}
\usepackage{amsfonts,amsthm}
\usepackage[driverfallback=dvipdfm, colorlinks=true, citecolor=blue]{hyperref}
\usepackage{graphicx}
\usepackage{bmpsize}
\usepackage{MnSymbol}
\usepackage{color}

\usepackage[left=3.20cm,top=3.5cm,right=3.20cm,bottom=2.5cm]{geometry}

\newtheorem{theorem}{Theorem}
\newtheorem{proposition}[theorem]{Proposition}
\newtheorem{corollary}[theorem]{Corollary}
\newtheorem{lemma}[theorem]{Lemma}

\theoremstyle{definition}

\newtheorem{definition}[theorem]{Definition}

\newtheorem{remark}[theorem]{Remark}

\newtheorem{example}[theorem]{Example}

\newtheorem{convention}[theorem]{Convention}

\numberwithin{equation}{section}
\numberwithin{theorem}{section}

\newcommand{\G}{\Gamma}
\renewcommand{\d}{\delta}

\renewcommand{\k}{\kappa}

\renewcommand{\l}{\lambda}

\renewcommand{\o}{\omega}

\renewcommand{\r}{\rho}

\newcommand{\p}{\psi}
\renewcommand{\P}{\Psi}

\DeclareMathOperator{\Fix}{Fix}
\DeclareMathOperator{\tr}{tr}
\DeclareMathOperator{\st}{st}
\newcommand{\al}{\alpha}

\newcommand{\gam}{\gamma}
\newcommand{\Del}{\Delta}

\newcommand{\ep}{\epsilon}

\newcommand{\thet}{\theta}
\newcommand{\Lam}{\Lambda}
\newcommand{\lam}{\lambda}
\newcommand{\Sig}{\Sigma}
\newcommand{\sig}{\sigma}

\newcommand{\om}{\omega}

\newcommand{\Cc}{{\mathcal C}}

\newcommand{\N}{{\mathbb N}}
\newcommand{\R}{{\mathbb R}}

\newcommand{\Z}{{\mathbb Z}}
\newcommand{\Q}{{\mathbb Q}}

\newcommand{\calA}{\mathcal{A}}

\newcommand{\calC}{\mathcal{C}}

\newcommand{\calG}{\mathcal{G}}

\newcommand{\calI}{\mathcal{I}}

\newcommand{\calL}{\mathcal{L}}
\newcommand{\calM}{\mathcal{M}}

\newcommand{\calP}{\mathcal{P}}

\newcommand{\calV}{\mathcal{V}}

\newcommand{\calX}{\mathcal{X}}

\newcommand{\calZ}{\mathcal{Z}}

\newcommand{\bbH}{\mathbb{H}}

\newcommand{\bbP}{\mathbb{P}}

\newcommand{\bbW}{\mathbb{W}}

\newcommand{\bbZ}{\mathbb{Z}}

\newcommand{\frakC}{\mathfrak{C}}
\newcommand{\frakG}{\mathfrak{G}}
\newcommand{\frakX}{\mathfrak{X}}

\newcommand{\frakh}{\mathfrak{h}}

\newcommand{\frakw}{\mathfrak{w}}

\def\({\left(}
\def\){\right)}

\def\l\{{\left\{}
\def\r\}{\right\}}
\def\wt{\widetilde}
\def\wh{\widehat}

\DeclareMathOperator{\diam}{diam}
\newcommand{\bs}{\backslash}
\newcommand{\Ra}{\Rightarrow}
\newcommand{\ra}{\rightarrow}
\newcommand{\loopra}{\looparrowright}
\newcommand{\lra}{\longrightarrow}
\newcommand{\trieq}{\trianglelefteq}
\newcommand{\nil}{\varnothing}
\newcommand{\ssm}{\smallsetminus}
\newcommand{\what}{\widehat}
\newcommand{\ov}[1]{\overline{#1}}
\newcommand{\inn}[2]{\langle #1,#2\rangle}

\newcommand{\wtilde}{\widetilde}

\def\wbar{\overline}
\def\id{{\rm id}}

\def\Cay{{\rm Cay}}
\def\Leb{{\rm Leb}}
\def\ev{{\rm ev\,}}
\def\1{{\bf 1}}

\def\gr{{\rm gr}}
\def\conj{{\bf conj}}

\def\ac{{\rm ac}}
\def\sing{{\rm sing}}

\def\H{{\mathbb H}}

\def\bS{{\bf S}}

\def\CAT{{\rm CAT}}

\def\Dil{{\rm Dil}}

\newcommand\numberthis{\addtocounter{equation}{1}\tag{\theequation}}

   \newcommand{\SC}[1]{
     {\color{red} \bf{(SC: #1)}}
   }
  \newcommand{\RT}[1]{
    {\color{blue} \bf{(RT: #1)}}
  }

\date{\today}

\title[Rigidity and statistics for group actions on $\text{CAT}(0)$ cube complexes]{Rigidity phenomena and the statistical properties of group actions on $\text{CAT}(0)$ cube complexes}
\author{\small{Stephen Cantrell and Eduardo Reyes}}

\begin{document}
\maketitle
\begin{abstract}
We compare the marked length spectra of some pairs of proper and cocompact cubical actions of a non-virtually cyclic group on $\CAT(0)$ cube complexes. The cubulations are required to be virtually co-special, have the same sets of convex-cocompact subgroups, and admit a contracting element.  
There are many groups for which these conditions are always fulfilled for any pair of cubulations, including non-elementary cubulable hyperbolic groups, many cubulable relatively hyperbolic groups, and many right-angled Artin and Coxeter groups. 

For these pairs of cubulations, we study the Manhattan curve associated to their combinatorial metrics. We prove that this curve is analytic and convex, and a straight line if and only if the marked length spectra are homothetic. The same result holds if we consider invariant combinatorial metrics in which the lengths of the edges are not necessarily one. In addition, for their standard combinatorial metrics, we prove a large deviation theorem with shrinking intervals for their marked length spectra. We deduce the same result for pairs of word metrics on hyperbolic groups. 

The main tool is the construction of a finite-state automaton that simultaneously encodes the marked length spectra of both cubulations in a coherent way, in analogy with results about (bi)combable functions on hyperbolic groups by Calegari-Fujiwara \cite{calegari-fujiwara}. The existence of this automaton allows us to apply the machinery of thermodynamic formalism for suspension flows over subshifts of finite type, from which we deduce our results. 
\end{abstract}
\maketitle


\section{Introduction}

In this work we study rigidity phenomena and the statistical properties of group actions on $\CAT(0)$ cube complexes and the methods we use exploit the interplay between geometric group theory and dynamics. Group actions on $\CAT(0)$ cube complexes are nowadays a central object of study. Since the influential work of Sageev \cite{sageev}, we have known that many groups admit proper and cocompact cubical actions on $\CAT(0)$ cube complexes (in the sequel, referred to as \emph{cubulations}), such as small cancellation
groups \cite{martin-steenbock,wise.small}, many 3-manifold groups \cite{bergeron-wise,hagen-przytycki,przytycki-wise.1,przytycki-wise.2,tidmore}, Coxeter groups \cite{niblo-reeves}, many Artin groups \cite{charney-davis,godelle-paris}, random groups at low density \cite{odrzygozdz,ollivier-wise}, 1-relator groups with torsion \cite{lauer-wise,stucky,wise}, hyperbolic free-by-cyclic groups \cite{hagen-wise.1,hagen-wise.2}, and so on. In particular, the fundamental groups of (compact) \emph{special cube complexes} introduced by Haglund and Wise \cite{haglund-wise.special} form a very rich class of convex-cocompact subgroups of right-angled Artin groups, and they played a key role in the resolution of the Virtual Haken and Virtual Fibering Conjectures \cite{agol,wise}. 

In general, when non-empty, the space of geometric actions of a given group on $\CAT(0)$ cube complexes is quite large. For example, each filling multicurve on a closed hyperbolic surface is dual to a cubulation of its fundamental group \cite{sageev}. 
Similarly, cubulations for fundamental groups of cusped hyperbolic 3-manifolds can be obtained from their vast sets of (relatively) quasiconvex surface subgroups \cite{bergeron-wise,cooper-futer,kahn-markovic}. 
Moreover, cubulations can be used to define deformation spaces, such as the classical Culler-Vogtmann outer space \cite{culler-vogtmann} that encodes geometric actions of free groups on trees. This perspective has been extended to right-angled Artin groups, for which outer spaces have been constructed using cubulations with some particular special cube complexes as quotients \cite{BCV,CSV}.

Under some reasonable irreducibility assumptions, actions on $\CAT(0)$ cube complexes are \emph{marked length-spectrum rigid} \cite{beyrer-fioravanti.1,beyrer-fioravanti.2}. More precisely, let $\calX$ be a cubulation of a group $\G$ and let $\conj=\conj(\G)$ denote the set of conjugacy classes of $\G$. The \emph{(stable) translation length} of this action is the function $\ell_\calX:\conj \ra \R$ given by
\[\ell_\calX[g]=\lim_{n\to \infty}{\frac{d_\calX(g^n x,x)}{n}},\]
where $d_\calX$ denotes the combinatorial metric on the 1-skeleton of $\calX$ and the limit above is independent of the representative $g\in [g]$ and the vertex $x\in \calX$.

For two cubulations $\calX,\calX_\ast$ of $\G$ on $\CAT(0)$, marked length-spectrum rigidity states that the equality of translation length functions $\ell_\calX=\ell_{\calX_\ast}$ implies the existence of a $\G$-equivariant cubical isometry from $\calX$ onto $\calX_\ast$. Since in general the translation length functions $\ell_\calX$ and $\ell_{\calX_\ast}$ will not coincide, it is natural to ask about the behavior of $\ell_{\calX_\ast}[g]$ when $\ell_\calX[g]$ is large. The goal of this paper is to address this question for ``compatible'' pairs of virtually co-special cubulations i.e., those having quotients with a special cube complex as a finite cover. Such compatibility is described in Definition \ref{def.classX}, and is guaranteed for any group in the following class. 
\begin{definition}\label{def.frakG}
    Let $\frakG$ be the class of non-virtually cyclic groups $\G$ satisfying the following:
    \begin{enumerate}
        \item $\G$ admits a proper, cocompact and virtually co-special action on a $\CAT(0)$ cube complex.
        \item The class of convex-cocompact subgroups of $\G$ is the same with respect to any proper and cocompact action on a $\CAT(0)$ cube complex. 
        \item Some (equivalently, any) proper and cocompact action of $\G$ on a $\CAT(0)$ cube complex has a contracting element.
    \end{enumerate}
\end{definition}

By Agol's theorem \cite[Thm.~1.1]{agol} and the characterization of convex-cocompact subgroups in terms of quasiconvexity \cite[Prop.~7.2]{haglund-wise.special}, we see that every cubulable non-elementary hyperbolic group belongs to $\frakG$. Moreover, the class $\frakG$ is closed under relative hyperbolicity, in the sense that a cubulable relatively hyperbolic group belongs to $\frakG$ as long as its peripheral subgroups belong to $\frakG$. In particular, any $C'(1/6)$ small cancellation quotient of the free product of finitely many groups in $\frakG$ belongs to $\frakG$ \cite{martin-steenbock}. However, this class is much larger since it contains some infinite families of right-angled Artin and Coxeter groups, most of them not being relatively hyperbolic with respect to any collection of proper subgroups. For instance, any right-angled Artin group with finite outer automorphism group belongs to $\frakG$.  See Proposition \ref{prop.groupsfornicecubulation} for the precise statement.

\subsection{Manhattan curves}  
Let $\calX,\calX_\ast$ be two cubulations of the group $\G$. We endow these cubulations with $\G$-invariant orthotope structures $\frakw,\frakw_\ast$ respectively, consisting of (non-necessarily integer) positive lengths assigned to the hyperplanes which are invariant under the action of $\G$. This induces isometric actions of $\G$ on the cuboid complexes $\calX^\frakw=(\calX,\frakw)$ and $\calX_\ast^{\frakw_\ast}=(\calX_\ast,\frakw_\ast)$, see Section \ref{subsec.prelimcub} for further details. The \emph{Manhattan curve} for the pair $(\calX^\frakw, \calX^{\frakw_\ast}_\ast)$ is the boundary of the convex set
\[
\mathcal{C}_{\calX^{\frakw_\ast}_\ast/\calX^\frakw}:=\left\{ (a,b) \in \R^2 : \sum_{[g] \in \conj(\G)} e^{-a\ell^{\frakw}_\calX[g] - b\ell^{\frakw_\ast}_{\calX_\ast}[g]} < \infty \right\},
\]
where $\ell_\calX^{\frakw}$ and $\ell^{\frakw_\ast}_{\calX_\ast}$ are the respective translation length functions of the actions of $\G$ on the 1-skeleta of $\calX^\frakw$ and $\calX_\ast^{\frakw_\ast}$. We can parameterize this curve as $s \mapsto \theta_{\calX^{\frakw_\ast}_\ast/{\calX^\frakw}}(s)$, where for $s \in \R$, $\theta_{\calX^{\frakw_\ast}_\ast/{\calX^\frakw}}(s)$ is the abscissa of convergence of the series
\[
t \mapsto \sum_{[g] \in \conj(\G)} e^{-t\ell^\frakw_\calX[g] - s\ell^{\frakw_\ast}_{\calX_\ast}[g]}.
\]
By abuse of notation, we also call the parametrization $\theta_{\calX^{\frakw_\ast}_\ast/\calX^\frakw}$ the Manhattan curve of $(\calX^\frakw,\calX^{\frakw_\ast}_\ast)$. 

Manhattan curves are useful tools for studying pairs of actions and are related to rigidity results, as well as recovering asymptotic invariants. They were introduced by Burger \cite{burger} for pairs of convex-cocompact representations of a group on rank 1 symmetric spaces, and later Sharp \cite{sharp.3} proved that they are analytic for pairs of cocompact Fuchsian representations. Sharp also extended these results for pairs of points in the outer space of a free group \cite{sharp.1,sharp.2}. Recently, Manhattan curves have been studied for pairs of cusped Fuchsian representations \cite{kao.1,kao.2}, pairs of cusped quasi-Fuchsian representations \cite{BCK}, comparing quasi-Fuchsian representations with negatively curved metrics on surfaces \cite{kao.3}, 
pairs of cusped Hitchin representations \cite{BCKM}, and
pairs of geometric actions on hyperbolic groups \cite{cantrell.tanaka.2,cantrell.tanaka.1}.

Our first main theorem fits into the aforementioned results, and to the authors' knowledge, is the first to address the analyticity of Manhattan curves outside the scope of relatively hyperbolic groups or representation theory.

\begin{theorem}\label{thm.manhattanfactsG}
Let $\G$ be a group in the class $\frakG$ and let it act properly and cocompactly on the cuboid complexes $\calX^\frakw=(\calX,\frakw)$ and $\calX_\ast^{\frakw_\ast}=(\calX_\ast,\frakw_\ast)$. Then the Manhattan curve $\thet_{\calX^{\frakw_\ast}_\ast/\calX^\frakw}:\R \ra \R$
is convex, decreasing, and analytic. In addition, the following limit exists and equals $-\thet'_{\calX^{\frakw_\ast}_\ast/\calX^\frakw}(0)$:
\[\tau(\calX^{\frakw_\ast}_\ast/\calX^\frakw):= \lim_{T\to\infty} \frac{1}{\#\{[g]\in \conj \colon \ell^\frakw_\calX[g]< T\}} \sum_{\ell^\frakw_\calX[g]< T} \frac{\ell^{\frakw_\ast}_{\calX_\ast}[g]}{T}.
    \]
Moreover, we always have
\[\tau(\calX^{\frakw_\ast}_\ast/\calX^\frakw)\geq v_{\calX^\frakw}/v_{\calX^{\frakw_\ast}_\ast},\]
for $v_{\calX^\frakw},v_{\calX^{\frakw_\ast}_\ast}$ the corresponding exponential growth rates, and the following are equivalent:
\begin{enumerate}
    \item $\theta_{\calX^{\frakw_\ast}_\ast/\calX^\frakw}$ is a straight line;
    \item there exists $\Lam >0$ such that $\ell^\frakw_\calX[g] = \Lam  \ell^{\frakw_\ast}_{\calX_\ast}[g]$ for all $[g] \in \conj(\G)$; and,
    \item $\tau(\calX^{\frakw_\ast}_\ast/\calX^\frakw) = v_{\calX^\frakw}/v_{\calX^{\frakw_\ast}_\ast}$.
\end{enumerate}
\end{theorem}

\begin{remark}
    In the result above, the group $\G$ does not have to belong to $\frakG$ as long as the triple $(\G,\calX,\calX_\ast)$ belongs to the class $\frakX$ in Definition \ref{def.classX}. In particular, the action on $\calX_\ast$ does not have to be proper. See Theorem \ref{thm.manhattanfactsX} for the more general statement.
\end{remark}

The main tool in the proof of Theorems \ref{thm.manhattanfactsG} and \ref{thm.manhattanfactsX} is the construction of a finite-state automaton that simultaneously encodes translation lengths for the actions on both $\calX$ and $\calX_\ast$. This is done in Proposition \ref{prop.languageL_X}, where for a convenient finite index subgroup $\ov\G<\G$ such that the quotient $\ov\calX=\ov\G \bs \calX$ is a special cube complex, we construct a finite state automaton $\calA_\calX$ with corresponding directed graph $\calG_\calX$ satisfying the following:
\begin{itemize}
    \item The labeling map assigns to any edge of $\calG_\calX$  an oriented hyperplane of $\ov\calX$. 
    \item There is a base vertex $o \in \calX$ such that each admissible path in $\calG_\calX$ corresponds to a geodesic path in $\calX$ starting at $o$. The hyperplanes dual to the edges of this geodesic project in $\ov\calX$ to the labels of the corresponding path in $\calG_\calX$. 
    \item The assignment from the set of admissible paths in $\calG_\calX$ into $\calX^0$ that sends each path to the endpoint of its geodesic constructed as above is surjective and uniformly finite-to-one. 
    \item Loops in $\calG_\calX$ correspond to geodesics in $\calX$ with endpoints related by an element of $\ov\G$, so they induce conjugacy classes in $\ov \G$. The length of any such geodesic is the $\ell_\calX$-translation length of the conjugacy class.
    \item There exist functionals $r_\calX^{\frakw}$ and $\psi^{\frakw_\ast}_{\calX_\ast}$ on the edges of $\calG_\calX$ such that their sums along the edges of a loop in $\calG_\calX$ equal the $\ell^{\frakw}_{\calX}$-translation length and $\ell^{\frakw_\ast}_{\calX_\ast}$-translation length of the induced conjugacy class, respectively. 
\end{itemize}

We use the automaton $\calA_\calX$ and the functional $r_\calX^\frakw$ to define a suspension flow over a subshift of finite type, whose periodic orbits induce conjugacy classes in $\conj(\ov\G)$ with periods corresponding to $\ell^\frakw_\calX$-translation lengths. The functional $\psi^{\frakw_\ast}_{\calX_\ast}$ induces a potential $\Phi$ on this suspension, whose Birkhoff sums correspond to $\ell^{\frakw_\ast}_{\calX_\ast}$-translation lengths. The existence of a contracting element allows us to describe the Manhattan curve $\theta_{\calX^{\frakw_\ast}_\ast/\calX^\frakw}$ in terms of pressure functions associated to $\Phi$ (see Proposition \ref{prop.manattanpressure}), and the theorems then follow by standard results in thermodynamic formalism.

\subsection{Large deviations}
For a pair $\calX,\calX_\ast$ of cubulations of a group $\G\in \frakG$, we also study large deviations for their translation lengths. That is, we estimate the number of conjugacy classes $[g]$ for which
\begin{equation}\label{eq.quotient}
\left| \frac{\ell_{\calX_\ast}[g]}{\ell_\calX[g]} - \eta \right| < \epsilon \ \text{ for some $\eta \in \R$ and a small $\epsilon >0$}.
\end{equation}
A more difficult question is to study the set of conjugacy classes $[g]$ such that
\begin{equation} \label{eq.difference}
|\ell_{\calX_\ast}[g] - \eta \ell_\calX[g]| < \epsilon \ \text{ for some $\eta \in \R$ and a small $\epsilon >0$}.    
\end{equation}
Indeed, this is more delicate than the corresponding quotient question above as (when $\ell_{\calX}[g]$ is bounded away from $0$) \eqref{eq.difference} implies \eqref{eq.quotient} but not vice versa.

Despite this latter question being significantly harder, there are previous works that tackle this problem in other settings. For example, let $\Sigma$ be a closed surface with negative Euler characteristic and fundamental group $\G$, and suppose that $\mathfrak{g}$ and $\mathfrak{g}_\ast$ are two hyperbolic metrics on $\Sigma$. These metrics induce isometric actions of $\G$ on $\widetilde\Sigma$ with translation length functions $\ell_{\mathfrak{g}}$ and $\ell_{\mathfrak{g}_\ast}$. A result of Schwartz and Sharp \cite{schwartz.sharp} states that there is an interval $(\alpha, \beta) \subset \R$ and constants $C, \lam >0$ such that any $\eta \in (\alpha, \beta)$ satisfies
\[
\#\left\{ [g] \in \conj(\G): \ell_{\mathfrak{g}}[g] < T : |\ell_{\mathfrak{g}_\ast}[g] - \eta \ell_{\mathfrak{g}}[g] | <\epsilon \right\} \sim \frac{Ce^{\lam T}}{T^{3/2}}
\]
as $T\to\infty$ for any fixed $\epsilon >0$ (here `$\sim$' represents that the quotients of the two quantities converge to $1$ as $T\to\infty$). Similar results are known to hold for surfaces of variable negative curvature by Dal'bo \cite{dalbo}, for Hitchin representations by Dai and Martone \cite{dai.martone}, for Green metrics by Cantrell \cite{cantrell.new} and for some pairs of points in outer space by Sharp \cite{sharp.1}. These asymptotics are often referred to as correlation results and to prove them thermodynamic formalism is usually employed. 
To apply thermodynamic formalism one needs to know that the length spectra of the two considered metrics are not \emph{rationally related}. That is, if $\ell_1$, $\ell_2$ are the length spectra that we want to compare then we would need to know that there do not exist non-zero $a,b \in \R$ with $a\ell_1[g] + b\ell_2[g] \in \Z$ for all $[g] \in \conj(\G).$
This property is vital as it implies bounds on the operator norm of families of transfer operators which are then used in the proof of the correlation asymptotic.

On the other hand, the length spectra of a pair of cubical actions on $\CAT(0)$ cube complexes are always rationally related. Indeed, after possibly performing one cubical barycentric subdivision, every cubical isometry of a $\CAT(0)$ cube complex either fixes a vertex or preserves a bi-infinite geodesic on which the isometry acts by non-trivial translations \cite{haglund}. In particular, the translation length function associated to any of these actions has image belonging to $\frac{1}{2}\Z$. 

However, by means of the automaton $\calA_\calX$, we are still able to estimate the number of conjugacy classes satisfying \eqref{eq.difference}. As a consequence of Theorem \ref{thm.ldtsi} in the setting of subshifts of finite type, we can prove the following theorem. 

\begin{theorem}\label{thm.cubulation.Gversion}
Let $\G$ be a group in the class $\frakG$ and let it act properly and cocompactly on the $\CAT(0)$ cube complexes $\calX$ and $\calX_\ast$. Let $\conj'\subset \conj$ be the set of non-torsion conjugacy classes and consider the dilations 
\[
\Dil(\calX_\ast,\calX) = \sup_{[g] \in \conj'} \frac{\ell_{\calX_\ast}[g]}{\ell_{\calX}[g]} \ \text{ and } \
\Dil(\calX,\calX_\ast)^{-1} = \inf_{[g] \in \conj'} \frac{\ell_{\calX_\ast}[g]}{\ell_{\calX}[g].}
\] 
Then there exists an analytic function $$\calI:[\Dil(\calX,\calX_\ast)^{-1}, \Dil(\calX_\ast,\calX)]\ra \R$$ and $C>0$ such that for any $\eta \in (\Dil(\calX,\calX_\ast)^{-1}, \Dil(\calX_\ast,\calX))$ we have
    \begin{equation}\label{eq.LDcubu}
   0 <  \limsup_{T\to\infty} \frac{1}{T} \log \left(\#\left\{ [g] \in \conj: \ell_\calX[g] < T, | \ell_{\calX_\ast}[g]  - \eta \ell_{\calX}[g] | < \frac{C}{T} \right\} \right)= \calI(\eta) \le v_\calX.
    \end{equation} 
    Furthermore, we have equality in the above inequality if and only if $\eta=\tau(\calX_\ast/\calX)$.
\end{theorem}

\begin{remark}
As in Theorem \ref{thm.manhattanfactsG}, the conclusion above still holds for triples $(\G,\calX,\calX_\ast)$ in the class $\frakX$, see Theorem \ref{thm.cubulation}. However, for our arguments (particularly Theorem \ref{thm.ldtsi}) it is crucial that the translation length functions belong to a lattice in $\R$. We still expect Theorem \ref{thm.cubulation} to hold for arbitrary cuboid complexes $\calX^\frakw$ and $\calX^{\frakw_\ast}_\ast$, but we will not pursue this in this work.    
\end{remark}

As an application of Theorem \ref{thm.cubulation.Gversion} we deduce large deviations with shrinking for the intersection of curves on hyperbolic surfaces. Let $\Sig$ be a closed orientable surface of negative Euler characteristic and fundamental group $\G$. If $\al,\beta$ are immersed closed oriented curves in $\Sig$, then the \emph{(geometric) intersection number} is the minimal number $i_\Sig(\al,\beta)$ of intersections of closed curves in the free homotopy classes of $\al$ and $\beta$. The function $i_\Sig$ can be extended by bilinearity to \emph{weighted multicurves}, which are finite sums of the form $\sum_j{\lam_j\al_j}$ with $\al_j$ immersed oriented closed curves in $\Sig$ and a set $(\lam_j\geq 0)_i$  of \emph{weights}. Any non-trivial element in $\conj(\G)$ is represented by a unique free homotopy class of immersed oriented closed curves, so we can talk of the intersection number between a weighted multicurve in $\Sigma$ and a conjugacy class in $\conj(\G)$. For more details about the intersection number, see \cite{farb-margalit}. 

A generating set $S$ for $\G$ is \emph{simple} if there exists a point $p\in \Sig$ such that elements of $S\subset \G=\pi_1(\Sig,p)$ can be represented by simple loops that
are pairwise non-homotopic and disjoint except at the base point $p$. For example, the generating set for the standard presentation
$$\G=\left<a_1,b_1,\dots,a_g,b_g\colon[a_1,b_1]\cdots [a_g,b_g]\right>$$
is simple. In \cite[Thm.~1.2]{erlandsson}, Erlandsson proved that the translation length function $\ell_S$ of the word metric of a simple generating set $S$ can be recovered by pairing in the intersection number against a carefully chosen weighted multicurve $\al_S$ with weights in $\frac{1}{2}\bbZ$. Up to scaling, this translation length function can also be recovered by looking at the $\CAT(0)$ cube complex dual to the multicurve $\al_S$. Therefore, Theorem \ref{thm.cubulation} applies and we obtain the following. 

\begin{corollary}\label{coro.intersectionsurfaces}
    Let $\G$ be the fundamental group of the closed orientable hyperbolic surface $\Sigma$ and consider a simple generating set $S$ of exponential growth rate $v_S$. Let $\al$ be a non-trivial weighted multicurve on $\Sigma$ with integer weights, and define
    \[a_{\inf}:=\inf_{[g] \in \conj'} \frac{i_\Sig(\al,[g])}{\ell_{S}[g]} \ \text{ and } \
a_{\sup} := \sup_{[g] \in \conj'} \frac{i_\Sig(\al,[g])}{\ell_{S}[g]}.
\] Then there exists an analytic convex function $\calI:[a_{\inf},a_{\sup}] \ra \R$ and $C>0$ such that for any $\eta \in (a_{\inf},a_{\sup})$ we have
        \[
   0 <  \limsup_{T\to\infty} \frac{1}{T} \log \left(\#\left\{ [g] \in \conj: \ell_S[g] < T, | i_\Sigma(\al,[g])  - \eta \ell_{S}[g] | < \frac{C}{T} \right\} \right)= \calI(\eta) \le v_S.
    \]
\end{corollary}

\subsection{Word metrics on hyperbolic groups}
Since the proof of Theorem \ref{thm.cubulation.Gversion} relies on Theorem \ref{thm.ldtsi} (a purely dynamical statement), the existence of the automaton $\calA_\calX$ encoding both actions on $\calX$ and $\calX_\ast$, and the arithmeticity of the translation length functions, we can deduce large deviation with shrinking for any pair of group actions fulfilling similar conditions. That is the case of word metrics on hyperbolic groups, and in fact, this was the author's main motivation at the beginning of this project.

Let $\G$ be a non-elementary hyperbolic group and let $S,S_\ast\subset \G$ be finite generating sets with corresponding word metrics $d_S, d_{S_\ast}$. By Cannon's theorem \cite{cannon}, for a total order on $S$, the language of lexicographically first geodesics in $\G$ is regular, so it is parametrized by a finite-state automaton. As a consequence of \cite[Lem.~3.8]{calegari-fujiwara}, Calegari and Fujiwara are able to modify this automaton (without modifying the parameterized language), and find an integer functional on the edges of the graph of the automaton so that its sum over paths recovers the $S_\ast$-word length for the corresponding element in $\G$ (this was our main motivation to construct the automaton $\calA_\calX$ in Proposition \ref{prop.languageL_X}).

By studying the subshift of finite type associated to this automaton, Cantrell and Tanaka \cite{cantrell.tanaka.1,cantrell.tanaka.2} deduced analyticity of the Manhattan curve for $S,S_\ast$ as well as a large deviation principle. 
More precisely, let $\ell_S, \ell_{S_\ast}$ be the corresponding translation length functions and 
consider the dilations
\[
\Dil(S_\ast,S) = \sup_{[g] \in \conj'} \frac{\ell_{S_\ast}[g]}{\ell_{S}[g]} \ \text{ and } \
\Dil(S,S_\ast)^{-1} = \inf_{[g] \in \conj'} \frac{\ell_{S_\ast}[g]}{\ell_{S}[g].}
\] 
Then there exists a real analytic, concave function $\calI :[ \Dil(S,S_\ast)^{-1}, \Dil(S_\ast,S)] \to \R_{>0}$ such that for $\eta \in (\Dil(S,S_\ast)^{-1}, \Dil(S_\ast,S))$ we have
\[
\lim_{\epsilon \to 0^+} \limsup_{n\to\infty} \frac{1}{T} \log \left( \#\left\{[g] \in \conj(\G): \ell_S[g] <  T : \left| \frac{\ell_{S_\ast}[g]}{\ell_S[g]} - \eta  \right| < \epsilon \right\} \right)= \calI(\eta).
\]
By applying Theorem \ref{thm.ldtsi} to this subshift, we can improve this result and obtain a large deviation theorem with shrinking.

\begin{theorem}\label{thm.wm}
    Let $\G$ be a non-elementary hyperbolic group and consider two finite generating sets $S, S_\ast$ for $\G$ with exponential growth rates $v_S,v_{S_\ast}$. Then there exists $C>0$ such that for any $\eta \in (\Dil(S,S_\ast)^{-1}, \Dil(S_\ast,S))$ we have
    \[
   0 <  \limsup_{T\to\infty} \frac{1}{T} \log \left(\#\left\{ [g] \in \conj: \ell_S[g] < T, | \ell_{S_\ast}[g]  - \eta \ell_{S}[g] | < \frac{C}{T} \right\} \right)= \calI(\eta) \le v_S.
    \]
    Furthermore, we have equality in the above inequality if and only if
    \[
    \eta = \tau(S_\ast/S):=\lim_{T\to\infty} \frac{1}{\#\{[g]\in \conj \colon \ell_S[g]<T\}} \sum_{\ell_S[g] < T} \frac{\ell_{S_\ast}[g]}{T}.
    \]
\end{theorem}
This result implies that, after scaling a pair of word metrics by their exponential growth rates, there is always an exponentially growing set for which their translation lengths are close, i.e. for any $\epsilon > 0$
\begin{equation}\label{example}
0 < \limsup_{T\to\infty} \frac{1}{T} \log \#\left\{ [g] \in \conj: \ell_S[g] < T, | v_S\ell_S[g]  - v_{S_\ast} \ell_{S_\ast}[g] | < \epsilon \right\}  \le v_S.     
\end{equation}
This extends the recent work \cite[Thm.~4.1]{cantrell.reyes.2} where the authors proved a correlation result for pairs of word metrics under an additional rationality assumption on the exponential growth rates. This result also answers a question raised by the authors in \cite[Rmk.~4.3]{cantrell.reyes.2}. We deduce the following corollary.
\begin{corollary} \label{coro.approx}
Let $\G$ be a non-elementary hyperbolic group, and let $S$ and $S_\ast$ be two finite generating sets on $\G$. Then there exists $C>0$ such that for any $\eta \in [\Dil(S,S_\ast)^{-1}, \Dil(S_\ast,S)]$ we can find an infinite sequence $(g_n)_{n\geq 1} \subset \G$ such that 
\[
\left|\frac{\ell_{S_\ast}[g_n]}{\ell_{S}[g_n]} - \eta \right| \le \frac{C}{|g_n|_S^{2}}.
\]
If $\eta \in [\Dil(S,S_\ast)^{-1}, \Dil(S_\ast,S)]$ is rational then there exists $g \in \G$ such that
\[
\frac{\ell_{S_\ast}[g]}{\ell_{S}[g]} = \eta.
\]
\end{corollary}
In Section \ref{sec.example} we present an example for a pair of word metrics on a free group. In particular, we compute the limit supremum in \eqref{example} for this pair of word metrics.

\subsection*{Organization}
The organization of the paper is as follows. In Section \ref{sec.prelim} we cover preliminary material about Manhattan curves, $\CAT(0)$ cube complexes and cubulable groups, symbolic dynamics and suspension flows, and finite-state automata. 

In Section \ref{sec.ld} we prove Theorem \ref{thm.ldtsi}, a large deviation result with shrinking intervals for lattice potentials on mixing subshifts of finite type that are constant on 2 cylinders. We apply this theorem to pairs of word metrics on hyperbolic groups in Section in Section \ref{sec.ldwordmetric} and prove Theorem \ref{thm.wm}.

In Section \ref{sec.cubulation} we prove Proposition \ref{prop.groupsfornicecubulation} that describes large classes of groups included in $\frakG$. There we also prove Proposition \ref{prop.languageL_X}, in which we construct a finite-state automaton for pairs of compatible actions on $\CAT(0)$ cube complexes.  We use this automaton in Section \ref{sec.thmscub} to prove Theorems \ref{thm.manhattanfactsX} and \ref{thm.cubulation}, from which we deduce Theorems \ref{thm.manhattanfactsG} and \ref{thm.cubulation.Gversion}. 

Finally, in the appendix we prove a Proposition \ref{prop.appcriterioncvxcc}, a criterion for convex-cocompactness of subgroups of cubulable relatively hyperbolic groups, which may be of independent interest.


\subsection*{Acknowledgements}
Both authors are grateful to Richard Sharp for his comments and suggestions on a preliminary version of this paper. The second author would like to thank the Max Planck Institut f\"ur Mathematik for its hospitality and financial support.


\section{Preliminaries}\label{sec.prelim}


\subsection{Isometric group actions}

Let $\G$ be a finitely generated group acting by isometries on the metric space $(X,d_X)$ and let $x\in X$ be an arbitrary base point. The \emph{(stable) translation length} of this action is the function $\ell_X:\conj \ra \R$ given by
\[\ell_X[g]=\lim_{n\to \infty}{\frac{d_X(g^nx,x)}{n}} \ \text{ for }g\in [g] \text{ in }\conj.\] 

The \emph{exponential growth rate} of this action is the quantity 
\[v_X:=\limsup_{T \to \infty}\frac{\log \#\{g\in \G \colon d_X(gx,x)<T\}}{T}\in [0,+\infty].\] 
As for the translation lengths, the exponential growth rate is independent of the chosen base point $x$. If $X$ is geodesic (or more generally roughly geodesic) and the action of $\G$ on $X$ is proper and cocompact, then $v_X$ is finite. In some cases we can recover the exponential growth rate as the limit
\[v_X=v_X(\G)=\limsup_{T\to \infty}{\frac{1}{T}\log \# \frakC_X(T)},\]
where each $T>0$ we denote
\[\frakC_X(T)=\{[g]\in \conj \colon \ell_X[g]<T\}.\]
This happens for example when $\G$ is hyperbolic and the action on $X$ is proper and cocompact. 

Given two isometric actions of $\G$ on the metric spaces $X$ and $X_\ast$, the \emph{Manhattan curve} for the pair $(X, X_\ast)$ is the boundary of the convex set
\[
\mathcal{C}_{X_\ast/X}:=\left\{ (a,b) \in \R^2 : \sum_{[g] \in \conj} e^{-a\ell_X[g] - b\ell_{X_\ast}[g]} < \infty \right\},
\]
assuming it is non-empty. Equivalently, $\calC_{X_\ast/X}$ is the set of points $(s,\theta_{X_\ast/X}(s))$ where  $\theta_{X_\ast/{X}}(s)$ is the abscissa of convergence of the series
\[
t \mapsto \sum_{[g] \in \conj} e^{-t\ell_X[g] - s\ell_{X_\ast}[g]}.
\]
By abuse of notation, $\thet_{X_\ast/X}$ is also called the Manhattan curve for $(X,X_\ast)$.


\subsection{$\CAT(0)$ cube complexes}\label{subsec.prelimcub}

For bibliography about $\CAT(0)$ cube complexes and groups acting on them, we refer the reader to \cite{bridson-haefliger,sageev.book}. A \emph{non-positively curved (NPC) cube complex} is a metric polyhedral complex in which all polyhedra are unit-length Euclidean cubes, and satisfies Gromov’s link condition: the link of each vertex is a flag complex. If this complex is simply connected we say that it is a \emph{$\CAT(0)$ cube complex}. 

Let $\calX$ be an NPC cube complex. Consider the minimal equivalence relations on the set of edges (resp. oriented edges) of $\calX$ such that the edges $e=\{v,w\}$ and $e=\{v',w'\}$ (resp. oriented edges $v\xrightarrow{e} w$ and $v'\xrightarrow{e'} w'$) are in the same equivalence class if $v,w,v',w'$ span a square in $\calX$ (resp. $v,w,v'w'$ span a square in $\calX$ with $v,v'$ adjacent and $w,w'$ adjacent). Equivalence classes of these equivalence relations are called \emph{hyperplanes} (resp. \emph{oriented hyperplanes}), and we let $\bbH(\calX)$ denote the set of hyperplanes of $\calX$. If the equivalence class of an (oriented) edge $e$ is the (oriented) hyperplane $\frakh$, then we say that $e$ \emph{dual} to $\frakh$. A hyperplane is \emph{2-sided} if it corresponds to exactly two oriented hyperplanes, otherwise it is \emph{1-sided}. 

Suppose now that $\calX$ is a  $\CAT(0)$ cube complex, in which case all hyperplanes are 2-sided. A \emph{combinatorial path} in $\calX$ is a sequence $\gam=(\gam_0,\dots,\gam_n)$ of vertices in $\calX$ such that $\gam_i$ is adjacent to $\gam_{i+1}$ for $i=0,\dots,n-1$. In that case, we say that $\gam_0$ is the initial vertex of $\gam_0$ and $\gam_n$ is its final vertex. The length of $\gam=(\gam_0,\dots,\gam_n)$ is defined as $n$. This path is often seen as a continuous path by also considering the edges $e_{i+1}$ joining each $\gam_i$ with $\gam_{i+1}$. Such a path is \emph{geodesic} if no two distinct edges $e_i$ are dual to the same hyperplane. The combinatorial metric on $\calX$ is the graph metric $d_\calX$ on its 1-skeleton $\calX^1$ so that each edge has length 1. It follows that a combinatorial path is geodesic if an only if it is geodesic for the metric $d_\calX$.

A hyperplane $\frakh$ in the $\CAT(0)$ cube complex $\calX$  \emph{separates} two vertices of $\calX$ if some (any) combinatorial path connecting these vertices has an edge dual to $\frakh$. It follows that the combinatorial distance of any two vertices in $\calX$ equals the number of hyperplanes separating them. Also, a hyperplane $\frakh$ determines the equivalence relation of ``not being separated by $\frakh$'' on the set of vertices of $\calX$. This equivalence relation has exactly 2 equivalence classes $\{\frakh^-,\frakh^+\}$, which are the \emph{halfspaces} determined by $\frakh$. 
 A subcomplex of $\calX$ is \emph{convex} if its vertex set is the intersection of halfspaces. Equivalently, $Z\subset \calX$ is convex if any combinatorial geodesic joining points in $Z^0$ is contained in $Z^0$. 

An \emph{orthotope structure} on the $\CAT(0)$ cube complex $\calX$ is a function $$\frakw: \bbH(\calX) \ra \R_{>0},$$
and the pair $\calX^\frakw=(\calX,\frakw)$ is called a \emph{cuboid complex}. An orthotope structure induces a metric $d_\calX^\frakw$ on $\calX^1$ by declaring each edge $e$ to have length $\frakw(\frakh)$ for $\frakh$ the hyperplane dual to $e$. In this way, for any two vertices $x,y\in \calX^0$ we have
\[d_\calX^\frakw(x,y)=\sum_{\frakh\in \bbH(x|y)}{\frakw(\frakh)},\]
where $\bbH(x|y)\subset \bbH(\calX)$ is the collection of hyperplanes separating $x$ and $y$. Note that if $\frakw$ is the constant function equal to 1, then $d_\calX^\frakw$ is just the standard combinatorial metric $d_\calX$.

\begin{remark}\label{rmk.geodorthotope}
    It is clear that a geodesic in $\calX$ with respect to $d_\calX$ is also geodesic with respect to $d_\calX^\frakw$ for any orthotope structure $\frakw$. 
\end{remark}

Now let $\G$ be a group acting on the $\CAT(0)$ cube complex $\calX$. We always assume that the action is cubical, and hence by isometries on $\calX^1$ with the combinatorial metric $d_\calX$. Similarly, $\ell_\calX$ always denotes the stable translation length of $\G$ with respect to the action on $(\calX^1,d_\calX)$. If the action of $\G$ is proper and cocompact, we say that $\calX$ is a \emph{cubulation} of $\G$. 

The action of $\G$ on $\calX$ induces a natural action on $\bbH(\calX)$. If $\frakw$ is a $\G$-invariant orthotope structure on $\calX$ in the sense that $\frakw(\frakh)=\frakw(g\frakh)$ for all $g\in \G$ and $\frakh\in \bbH(\calX)$, then we say that $\G$ \emph{acts} on the cuboid complex $(\calX,\frakw)$. In that case the action of $\G$ on $(\calX^1,d_\calX^\frakw)$ is also by isometries. We let $\ell_\calX^\frakw$ denote the stable translation length of $\G$ for its action on $(\calX^1,d_\calX^\frakw)$.

By a \emph{hyperplane stabilizer} we mean a subgroup of $\G$ group elements $g$ such that $g \frakh=\frakh$ for some fixed hyperplane $\frakh\in \bbH(\calX)$. A hyperplane $\frakh$ is \emph{essential} for the action of $\G$ if for any vertex $x$ of $\calX$, the halfspaces $\frakh^\pm$ contain elements in the $\G$-orbit of $x$ arbitrarily far from $\frakh^{\mp}$. The action of $\G$ on $\calX$ is \emph{essential} if every hyperplane is essential.  
   
If $\calX$ is a cubulation of $\G$, a subgroup $H<\G$ is called \emph{convex-cocompact} with respect to $\calX$ if there exists a convex subcomplex $Z\subset \calX$ that is $H$-invariant and so that the action of $H$ on $Z$ is cocompact. Such a subcomplex $Z$ is called a \emph{convex core} for $H$. Note that the hyperplane stabilizer of any hyperplane $\frakh$ is convex-cocompact since it acts cocompactly on the (convex subcomplex spanned by the) set of vertices in edges dual to $\frakh$ \cite[Lem.~13.4]{haglund-wise.special}.


If $\calX$ is a $\CAT(0)$ cube complex and $\bbW\subset \bbH(\calX)$ is any collection of hyperplanes, in \cite[Sec.~2.3]{caprace-sageev} Caprace and Sageev introduced the \emph{restriction quotient}, which is a $\CAT(0)$ cube complex $\calX(\bbW)$ equipped with a surjective cellular map $\phi:\calX \ra \calX(\bbW)$ satisfying the following: an edge in $\calX$ is collapsed to a single vertex under $\phi$ if and only if it is dual to a hyperplane not in $\bbW$. The projection $\phi$ induces a natural bijection between $\bbW$ and $\bbH(\calX(\bbW))$, and hence any orthotope structure $\frakw$ on $\calX$ induces an orthotope structure $\phi_\ast(\frakw)$ on $\calX(\bbW)$.


\begin{remark}\label{rmk.projquot}
Note that (pre)images of convex subcomplexes under restriction quotients remain convex. In particular, images of geodesic paths remain geodesic, although some subpaths are allowed to collapse to points. Also note that if $\G$ acts on $\calX$ and $\bbW$ is $\G$-invariant, then there is a natural action of $\G$ on $\calX(\bbW)$. 
\end{remark}

A cubulation $\calX$ of $\G$ is \emph{co-special} if the quotient $\ov\calX=\G \bs \calX$ is a \emph{special cube complex} in the sense of Haglund-Wise \cite{haglund-wise.special}. Equivalently, $\calX$ is co-special if $\G$ injects into a right-angled Artin group $A_G$ inducing a $\G$-equivariant isometric embedding of $\calX$ into $R_G$ as a convex subcomplex, where $R_G$ is the universal cover of the \emph{Salvetti complex} $\ov{R}_G$ associated to the graph $G$. Among other properties of special cube complexes, hyperplanes are 2-sided and embedded, and they do not self-osculate \cite[Def.~3.2]{haglund-wise.special}. In particular, different oriented edges with the same initial vertex are dual to different oriented hyperplanes. The cubulation $\calX$ of $\G$ is \emph{virtually co-special} if there exists a finite-index subgroup $\ov\G <\G$ such that the action of $\ov\G$ on $\calX$ is co-special.


\subsection{Symbolic dynamics}\label{sec.sd}
In this section we introduce the preliminary material we need from symbolic dynamics. See \cite{parry.pollicott} for more details.
Let $A$ be a $k \times k$ matrix with entries $0$ or $1$. 
This matrix is said to be \emph{aperiodic} if there exists $N \ge 1$ such that all of the entries of $A^N$ are strictly positive. We say that $A$ is \emph{irreducible} if for any $i,j \in \{1, \ldots, k \}$ there exists $n \ge 1$ such that $A_{i,j}^n$ is strictly positive.

The (one-sided) \emph{subshift of finite type} $\Sigma_A$ associated to $A$ is the set of infinite sequences
\[
\Sigma_A = \left\{(x_n)_{n=0}^{\infty} : x_j \in \{1, \ldots, k\} \text{ and } A_{x_j, x_{j+1}} =1 \text{ for all } j \ge 0\right\}.
\]
These infinite sequences can be seen as infinite paths in a directed graph $\mathcal{G}_A$ with vertices labelled $1, \ldots, k$ and a directed edge from vertex $i$ to $j$ if and only if $A_{i,j} =1$. We will therefore refer to the numbers $1, \ldots, k$ as the \emph{states} of $\Sigma_A$. We equip $\Sigma_A$ with the \emph{shift map} $\sigma: \Sigma_A \to \Sigma_A$ defined by
\[
\sigma((x_n)_{n=0}^\infty) = (x_{n+1})_{n=0}^\infty
\]
to obtain a dynamical system $(\Sigma_A, \sigma)$. 

Consider a finite ordered string $x_0, \ldots, x_{m-1} \in \{1, \ldots, k \}$ where $A_{x_j, x_{j+1}} =1$ for each $j=0, \ldots, m-2$. The \emph{cylinder set} associated to this string is the subset of $\Sigma_A$ given by
\[
[x_0, \ldots, x_{m-1}] := \left\{ (y_n)_{n=0}^\infty \in \Sigma_A : y_j = x_j \text{ for } j=0, \ldots, m-1 \right\}.
\]
We endow $\Sigma_A$ with a topology by declaring the set of all cylinder sets to be an open basis.

The system $(\Sigma_A, \sigma)$ is said to be \emph{mixing} if for any two open sets $U, V \subset \Sigma_A$ there is $N \ge 1$ such that $\sigma^n(U) \cap V \neq \emptyset$ for all  $n \ge N$. We say that  $(\Sigma_A, \sigma)$ is \emph{transitive} if for any two open sets $U,V \subset \Sigma_A$ there exists $n \ge 1$ such that $\sigma^n(U) \cap V \neq \emptyset.$ We have that $(\Sigma_A, \sigma)$ is mixing if and only if $A$ aperiodic and $(\Sigma_A, \sigma)$ is transitive if and only if $A$ is irreducible. We will often suppress the dependence of $A$ in the notation for a subshift and will write $(\Sigma, \sigma)$.\\

Throughout the rest of the section $(\Sigma, \sigma)$ will be a mixing subshift of finite type, and consider a function $\psi:\Sigma \to \R$. We say that $\psi$ is \emph{constant on $2$ cylinders} if $\psi$ is constant on each set of the form $[x_0, x_1]$ where $x_0, x_1 \in \{1, \ldots, k \}$ and $A_{x_0, x_1} = 1$. For each $n\ge 1$ the $n$th \emph{Birkhoff sum} of $\psi$ is the function 
\[
\psi^n : \Sigma \to \R  \ \text{ such that } \ \psi^n(x) := \psi(x) + \psi(\sigma(x)) + \cdots + \psi(\sigma^{n-1}(x)).
\]
A point $x \in \Sigma$ is said to be \emph{periodic} if $\sigma^n(x) = x$ for some $n \ge 1$. Such an $n$ is called a \emph{period} of $x$ and it is often denoted by $|x|$. Note that a periodic point has infinitely many periods, and in general we are not writing $|x|$ for the least period of $x$.

Two functions $\psi, \varphi: \Sigma \to \R$, which we are assuming to be constant on $2$ cylinders, are said to be \emph{cohomologous} if there exists a continuous function $u: \Sigma \to \R$ such that
$\psi(x) = \varphi(x) + u(\sigma(x)) - u(x)$ for all $x \in \Sigma$. By Livsic's Theorem \cite{parry.pollicott}, $\psi$ and $\varphi$ are cohomologous if and only if $\psi^n(x) = \varphi^n(x)$ whenever $\sigma^n(x) = x.$

The variational principle states that there is a unique $\sigma$-invariant Borel probability measure on $(\Sigma, \sigma)$ that achieves the supremum
\[
\text{P}(\psi) := \sup_{\mu\in \calM_\sig}\left\{ h_\mu(\sigma) + \int \psi \ d\mu \right\},
\]
where $\calM_\sig$ is the collection of all $\sigma$-invariant Borel probability measures on $\Sig$ and $h_\mu(\sig)$ denotes the \emph{(metric) entropy} of $\sig$ with respect to the measure $\mu$. The quantity $\text{P}(\psi)$ is referred to as the \textit{pressure} of $\psi$ and the measure attaining the supremum is called the \textit{equilibrium state} of $\psi$. When $\psi$ is a constant function, the measure achieving the supremum for the pressure of $\psi$ is the \emph{measure of maximal entropy}. Furthermore the \emph{topological entropy} $h=h(\sig)$ of $(\Sigma,\sigma)$  is given by $h = \text{P}(0)$. 


Consider the quantities
\begin{equation*}
\al_{\min} := \inf_{\mu \in \calM_\sig}\int_\Sigma \p \ d\mu \ \ \text{ and } \ \ \al_{\max} := \sup_{\mu \in \calM_\sig} \int_\Sigma \p \ d\mu.   
\end{equation*}
The large deviation principle implies that there exists a real analytic, concave function $\calL(\psi, \cdot) : \R \to \R_{>0}\cup \{\infty\}$ such that, for any non-empty sets $U \subset V \subset \R$ with $U$ open and $V$ closed we have
\begin{align*}
    -\inf_{s\in U} \calL(\psi, s) &\le  \liminf_{n\to\infty}  \frac{1}{n} \log \mu \left( x \in \Sigma : \left| \frac{\psi^n(x)}{n} - \eta \right| < \epsilon \right)\\
    &\le \limsup_{n\to\infty}  \frac{1}{n} \log \mu \left( x \in \Sigma : \left| \frac{\psi^n(x)}{n} - \eta \right| < \epsilon \right) \le -\inf_{s\in V} \calL(\psi,s).
\end{align*}
The function $\calL(\psi,\cdot)$ is the Legendre transform of $t \mapsto  \text{P}(t\psi) - h$. That is,
\[
 - \calL(\psi,s) = \inf_{t\in \R} (\text{P}(t\psi) - h - ts ).
\]
Furthermore, $\calL(\psi,\cdot)$ is finite on $[\alpha_{\min}, \alpha_{\max}]$ and is infinite otherwise. An alternative characterisation for $\calL$ is the following:
\[
-\calL(\p,\eta)= \sup\left\{ h_\mu(\sigma) : \mu \in\mathcal{M}_\sigma \text{ and} \int \p \ d\mu = \eta\right\} - h.
\]

A function $\psi: \Sig \ra \R$ is \emph{lattice} if there are $a,b \in \R$ satisfying
\[
\left\{ \p^n(x) + an : x\in\Sigma \text{ and } \sigma^n(x)=x \text{ for some } n\ge 1 \right\} \subseteq b\Z.
\] If this is not the case then we say that $\p$ is \emph{non-lattice}. 
\begin{remark}\label{rem.lattice}
    Suppose that $\p$ is lattice. Then $\p$ is cohomologous to a function of the form $a + b \varphi$ where $a,b \in \R$ and $\varphi: \Sigma \to \Z$ \cite{parry.pollicott}. When this is the case, the large deviation behaviour of $\p$ and $\varphi$ over periodic orbits is the same, since  $\p^n(x) = an + b\varphi^n(x)$ when $\sigma^n(x) =x$.
\end{remark}


\subsection{Suspension flows}\label{sec.sf}
In this section we define suspension flows of subshifts of finite type. See \cite{parry.pollicott} for more details on the results stated in this section.
Let $\Sigma_A$ be a transitive subshift of finite type and $r: \Sigma_A \to \R_{>0}$ a function that is constant on $2$ cylinders. We note that in \cite{parry.pollicott} suspension flows are considered over mixing subshifts however the same proofs (with some minor modifications) work when the subshift is transitive.  We define the \emph{suspension flow} of $\Sigma_A^r$ to be the space
\[
\Sigma_A^r = \{(x,t) \in \Sigma_A \times \R_{\ge 0} : 0 \le t \le r(x) \} / \sim
\]
where $(x,t) \sim (r(x),0)$ equipped with the flow $\sigma^r=(\sigma^r_t)_{t>0}$ so that $\sig_t^r$ send $(x,s)$ to $(x,s+t)$ for $s \in \R$. There is a natural metric on $\Sigma_A^r$ which can be constructed as in \cite{parry.pollicott}. We will not present the construction of this metric here as it is a little technical.

For a H\"older continuous function $\Phi : \Sigma_A^r \to \R$ we can define its pressure as 
\[
\text{P}_{\sigma^r}(\Phi) = \sup_{m \in \mathcal{M}_{\sigma^r}}  \left\{ h_m(\sig^r) + \int_{\Sigma_A^r} \Phi \ dm \right\},
\]
 where $\mathcal{M}_{\sigma^r}$ is the space of $\sigma^r_t$- invariant Borel probability measures on $\Sigma_A^r$ and $h_m(\sig^r)$ is the entropy of the time-one map $\sig^r_1$ for the measure $m$. The map $s \mapsto \text{P}_{\sigma^r}(s\Phi)$ is real analytic. Let $\delta_r >0$ be the unique number such that $\text{P}(-\delta_r r) =0$ and write $\mu_{-\delta_r r}$ for the equilibrium state of $-\delta_r r$ on $\Sigma_A$. The measure of maximal entropy for $\Sigma_A^r$ is (locally) given by
 \[
 \frac{\mu_{-\delta_r r} \times \text{Leb}}{\int r \ d\mu_{-\delta_r r}},
 \] i.e. up to normalisation it is the measure that acts as Lebesgue along the fibers of the suspension and as $\mu_{-\delta_r r}$ on the base. If we write $m$ for the measure of maximal entropy then we have that
 \[
\left. \frac{d}{ds}\right|_{s=0} \text{P}_{\sigma^r}(s\Phi) = \int \Phi \ dm.
 \]

 For $T>0$ we will write $P(\Sigma_A^r, T)$ for the collection of periodic orbits of $\sigma^r$ of length less than $T$. Given $R >0$ we will write $P(\Sigma_A^r, R , T)$ for the collection of periodic orbits of length between $T -R$ and $T+R$.
 
Given a H\"older continuous function $\Phi: \Sigma_A^r \to \R$ it is a standard result that for any $R >0$ sufficiently large
\[
\lim_{T\to\infty} \frac{1}{T} \log \left(\sum_{\tau \in P(\Sigma_A^r, R, T)} e^{-s \int_\tau \Phi }\right) = \text{P}_{\sigma^r}(-s\Phi)
\]
for any $ s\in\R$.

Lastly we recall that two functions $\Phi, \Psi: \Sigma_A^r \to \R$ are \emph{cohomologous} if $\Phi - \Psi = u'$  where $u : \Sigma_A^r \to \R$ is continuously differentiable (along flow lines) and
\[
u'(x) = \lim_{t\to0} \frac{u(\sigma^r_t(x)) - u(x)}{t}.
\]
Further the function $s \mapsto P_{\sigma^r}(s\Phi)$ is a straight line if and only if $\Phi$ is cohomologous to a constant function.


\subsection{Finite-state automata}
Let $S$ be a finite set and let $S^\ast$ denote the set of finite words over the alphabet $S$. If $w=h_1\cdots h_n$ and $w'=h'_1\cdots h_m'$ are words in $S^\ast$, then its \emph{concatenation} is the word $ww':=h_1\cdots h_nh'_1\cdots h_m'$. The \emph{length} of a word is the number of letters in $S$ composing it. We let the empty set correspond to the unique word of length 0 in $S^\ast$. A \emph{language} over $S$ is any subset of words in $S^\ast$.

An \emph{automaton} over $S$ is a tuple
\[\mathcal{A}=(\mathcal{G},\pi,I,F),\]
where $\calG=(V,E)$ is a finite directed graph, $\pi:E\rightarrow S$ is a \emph{labeling function} and $I,F\subset V$ are non-empty sets of \emph{initial} and \emph{final} states. 

\begin{remark}\label{rmk.initialstates}
    This convention differs from the standard definition of automaton, where it is required for $I$ to consist of a single vertex. 
\end{remark} 

By a \emph{path} in $\calG$ we mean a sequence $\om$ of (always directed) edges $e_1,\dots,e_n$ in $E$ such that the final vertex $v_i$ of $e_i$ is the initial vertex of $e_{i+1}$ for $i=1,\dots, n-1$. If $v_0$ is the initial vertex of $e_1$, we denote this path $\om$ either by $\om=(\xrightarrow{ e_1} \cdots \xrightarrow{e_n})$ or $\om=(v_0\xrightarrow{ e_1} \cdots \xrightarrow{e_n} v_n)$ depending on the emphasis we want to give to the vertices. 
If there is no ambiguity on the edges, we can also denote this path by $(v_0 \ra \cdots \ra v_n)$. The \emph{length} of a path is the number of edges that determine it. Note that paths of length 1 correspond to the edges in $E$. We also allow the degenerate case of paths of length 0, which are the vertices in $V$.

If $\om, \om'$ are paths in $\calG$ such that the final vertex of $\om$ is the initial vertex of $\om'$, then the concatenation $\om \om'$ is the path in $\calG$ defined in the expected way. Similarly we define the concatenation of any finite number of paths. 

A word $w$ in $S^\ast$ is \emph{represented} by the path $\omega=(v_0\xrightarrow{ e_1} v_1\cdots \xrightarrow{e_n} v_n)$  in $\calG$ if $w=\pi(e_1)\cdots\pi(e_n)$. If in addition $v_0\in I$ and $v_n\in F$, we say that $w$
\emph{accepted} by $\calA$. It is clear that the word represented by a concatenation $\om \om'$ is the concatenation of the words represented by $\om$ and $\om'$.
Let $L=L_\calA$ be the language consisting of the words accepted by $\calA$. In this case we say that $L$ is \emph{parametrized} by $\calA$.

The automaton $\calA$ is \emph{deterministic} if any two edges in $\calG$ with the same initial vertex have different labels. In that case, for any $w\in L_\calA$ and any initial state $v\in I$ there exists at most one path in $\calG$ representing $w$ and starting at $v$. The automaton is \emph{pruned} if any vertex in $\calG$ is the final vertex of a path starting at an initial state.


\section{Large deviations} \label{sec.ld}

In this section we discuss large deviations with shrinking for potentials on mixing subshifts of finite type. The main result of the section is Theorem \ref{thm.ldtsi}, and it will be used in the proof of Theorems \ref{thm.wm} and \ref{thm.cubulation} in subsequent sections. 

Suppose that $(\Sigma,\sigma)$ is a mixing subshift of finite type with $k \times k$ transition matrix $A$, and
let $\mu$ denote its measure of maximal entropy. Also, let $M$ be the least number such that $A^M$ has strictly positive entries. The large deviation principle from Section \ref{sec.sd} implies that there is a real analytic, concave function $\calL(\p, \cdot): [\alpha_{\min}, \alpha_{\max}] \to \R_{>0}$ such that the following holds: for any $\eta \in (\alpha_{\min}, \alpha_{\max})$
\begin{equation*}\label{eq.ldp}
 \lim_{\epsilon \to 0^-} \limsup_{n\to\infty} \frac{1}{n} \log \mu \left( x \in \Sigma : \left| \frac{\p^n(x)}{n} - \eta \right| < \epsilon \right) = - \calL(\p,\eta).   
\end{equation*}

Instead of taking two limits as above, first with respect to $n$ and then with respect to $\epsilon$, it is natural to ask the following. How quickly can a sequence $\d_n$ decay to $0$ as $n\to\infty$ so that we have
\begin{equation}\label{eq.ldpsi}
 \lim_{n\to\infty} \frac{1}{n} \log \mu\left(  x \in \Sigma :  \left| \frac{\p^n(x)}{n} - \eta \right| < \d_n  \right) = - \calL(\p,\eta)   
\end{equation}
for each $\eta \in (\alpha_{\min}, \alpha_{\max})$? 
We refer to this problem as \textit{large deviations with shrinking intervals}. We can ask the same question when the limit in \eqref{eq.ldpsi} is replaced with the limit supremum.


Large deviations with shrinking intervals are best understood for potentials $\p$ that are non-lattice. For example, when $\p$ is non-lattice the local central limit theorem \cite{lclt} implies that \eqref{eq.ldpsi} holds when $\delta_n^{-1} = O(n)$. 
In \cite{pollicott.sharp} Pollicott and Sharp improved this result under an additional assumption. They showed that if $\p$ satisfies a non-Diophantine condition then there exist $\kappa >0$ such that $\eqref{eq.ldpsi}$ holds when $\d_n^{-1} = O(n^{1+\kappa})$. 

When $\p$ is lattice, large deviations with shrinking intervals are not as well understood. The aim of this section is to study 
\eqref{eq.ldpsi} for functions $\p$ that are constant on $2$ cylinders and are lattice. 

We now state our large deviation theorem with shrinking intervals. Write $\calL(\p, \cdot): [\alpha_{\min}, \alpha_{\max}] \to \R_{>0}$ be the function introduced above.
\begin{theorem} \label{thm.ldtsi}
   Suppose that $(\Sigma,\sigma)$ is a mixing subshift of finite type and that $\p : \Sigma \to \R$ is a function that is constant on $2$ cylinders. Then the following holds. There exists $C >0$ such that for any $\eta \in (\alpha_{\min}, \alpha_{\max})$ 
\[
  \limsup_{n\to\infty} \frac{1}{n} \log\left( \#\left\{ x \in \Sigma: \sigma^n(x) = x \text{ and } \left| \frac{\p^n(x)}{n} - \eta \right| < \frac{C}{n^2} \right\} \right) = h -\calL(\p,\eta)
\]
where $h$ is the topological entropy of $(\Sigma,\sigma)$.
    Furthermore we can take
    \[
     C = \frac{4M^2(1+k^2)^2(\alpha_{\max} - \alpha_{\min}) }{\sqrt{5}}.
    \]
    
In the case that $\p$ takes values in $\Z$ then there exists $\eta \in (\alpha_{\min}, \alpha_{\max})$ and $\epsilon >0$ such that
     \[
    \limsup_{n\to\infty} \frac{1}{n} \log\left( \#\left\{ x \in \Sigma: \sigma^n(x) = x \text{ and } \left| \frac{\p^n(x)}{n} - \eta \right| < \frac{\epsilon}{n^2} \right\} \right) = 0.
    \] 
\end{theorem}

\begin{remark}\label{rem.transitive}
i) This theorem also holds if we only assume that $(\Sigma,\sigma)$ is transitive (opposed to mixing). Indeed, if $(\Sigma,\sigma)$ is transitive we can find an integer $p \ge 1$ such that $(\Sigma,\sigma^p)$ decomposes as a disjoint union of $p$, $\sigma^p$-invariant sets. These $\sigma^p$-invariant sets are mixing subshifts of finite type when equipped with $\sigma^p$. We can then apply the mixing version of our theorem above to these subshifts to deduce the transitive version.\\
    ii) Our result significantly improves the decay rate implied by the local limit theorem under the non-lattice assumption. Furthermore, the case that $\p$ takes values in a lattice shows that the decay rates obtained in the first part of Theorem \ref{thm.ldtsi} are optimal.
\end{remark}

For the rest of the section we note the correspondence between periodic orbits of $(\Sig,\sig)$ and cycles (i.e. closed paths) in the adjacency graph $\calG_A$. We say that a cycle is \emph{simple} if it does not visit any state more than once.

To prove the above result we start with the following observation.
\begin{lemma} \label{lem.achieve}
    Suppose that $(\Sigma,\sigma)$ is a mixing subshift of finite type and that $\p : \Sigma \to \R$ is a function that is constant on $2$ cylinders. Then
    \[
    \alpha_{\min} = \inf_{\sigma^n(x) = x} \frac{\p^n(x)}{n} \ \text{ and } \  \alpha_{\max} = \sup_{\sigma^n(x) = x} \frac{\p^n(x)}{n},
    \]
    and furthermore there exist periodic orbits $\overline{x},\overline{y}$ that achieve these values, i.e. $\p^{|\ov x|}(\overline{x}) = |\ov x|\alpha_{\min}, \p^{|\ov y|}(\overline{y}) = |\ov y|\alpha_{\max}.$
    Here the infimum and supremum are over all the periodic orbits.
\end{lemma}
\begin{proof}
    The first statement follows from a result of Sigmund \cite{sigmund} which states that the set of probability measures supported on periodic orbits  
    is dense in $\mathcal{M}_\sig$ (equipped with the weak-$\ast$ topology). For the furthermore statement note that each periodic orbit can be written as a disjoint union of simple cycles. It follows easily that the above infimum and supremum are attained by simple cycles (and powers of them).
\end{proof}

We also need a more technical version of this lemma that applies on subsequences. We begin with the following observation.


\begin{lemma} \label{lem.lc}
    Take an interval $(s,t) \subset \R$ and a number $\eta \in (s,t)$. Then there are infinitely many $n \ge 1$ for which there exist integers $0 \le a,b \le n$ with $a + b = n$ and such that
    \[
\left|\frac{as + bt}{n} - \eta \right| \le \frac{t-s}{\sqrt{5} \, n^2}.
    \]
\end{lemma}

\begin{proof}
Note that it suffices to prove this result when $s = 0$, $t=1$. The general result then follows by shifting and rescaling the interval $(0,1)$ into $(s,t)$.    When $s = 0$, $t=1$ the result follows from the well-known Hurwitz's Theorem \cite{hurwitz} from Diophantine approximation.
\end{proof}
 We also need the following lemma. Suppose that the simple periodic orbit $\overline{x}$ realising $\alpha_{\max}$ has initial state $i$ and that the initial state for $\overline{y}$ realising $\alpha_{\min}$ is $j$ (from Lemma \ref{lem.achieve}). Further assume that we repeat $\overline{x}$ and $\overline{y}$ by each other's periods so that they both have period $l$ satisfying $1<l \leq k^2$.

\begin{lemma}\label{lem.construct}
    Suppose that $(\Sigma,\sigma)$ is a mixing subshift of finite type and that $\p : \Sigma \to \R$ is a function that is constant on $2$ cylinders such that $\al_{\min}<\al_{\max}$. Then we can find periodic orbits $x, y$ both with period at most $2M(1+k^2)$ such that the initial state for $x$ is $i$, the initial state for $y$ is $j$ and we have that
    \[
    \alpha_{\min} < \frac{\p^{|x|}(x)}{|x|} < \frac{\p^{|y|}(y)}{|y|} < \alpha_{\max}.
    \]
\end{lemma}
\begin{proof}
     Consider the start state $i$ and find a path $p$ from $i$ to $j$ of length $M$ and a path $q$ from $j$ to $i$ of length $M$. By composing the path $p$ with $r$ (to be chosen later) repeats of a single cycle of $\overline{y}$ and $q$ we obtain a periodic orbit $y$ of length $2M + lr$. Further we have that
    \[
    \frac{\p^{|y|}(y)}{|y|} \le \frac{2M\alpha_{\max} + lr\alpha_{\min}}{2M + lr}
    \]
    and $y$ starts at $i$. Similarly we find $x$ starting at $j$ with $|x| = 2M + lr$ and 
    \[
    \frac{\p^{|x|}(x)}{|x|}\ge \frac{2M\alpha_{\min} + lr\alpha_{\max}}{2M + lr}.
    \]
    Now, as long as $r$ is chosen so that
    \[
    2M\alpha_{\min} + lr\alpha_{\max} > 2M\alpha_{\max} + lr\alpha_{\min}
    \]
    $x,y$ will satisfy the final inequality in the lemma. Note that this inequality is satisfied for $ r =2M$ in which case $x$ and $y$ have periods $2M+2Ml \le 2M(1+k^2)$ as required. 
\end{proof}

We also require the following.
\begin{lemma} \label{lem.unifint}
      Suppose that $(\Sigma,\sigma)$ is a mixing subshift of finite type and that $\p : \Sigma \to \R$ is a function that is constant on $2$ cylinders. Suppose that there exist periodic orbits $x,y$ both with period $l$ and same initial state such that $lA = \p^l(x) < \p^l(y) = lB$.
      Then there exists $\overline{C} >0$ such that for any $\eta \in (A, B)$ we can find an infinite sequence of periodic orbits $x_n \in \Sigma$ such that
    \[
    \left| \frac{\p^{|x_n|}(x_n)}{|x_n|} - \eta \right| \le \frac{\overline{C}}{|x_n|^2}.
    \]
    Furthermore we can take
    \[
    \overline{C} = \frac{(\alpha_{\max} - \alpha_{\min})l^2}{\sqrt{5}}.
    \]
\end{lemma}
\begin{proof}
    By Lemma \ref{lem.lc} there exist infinitely many $n \ge 1$ such that the following holds. There are integers $n_1, n_2\geq 0$ with $n_1 + n_2 = n$ and such that
\[
\left| \frac{n_1 A + n_2 B}{n_1+n_2} - \eta  \right| \le \frac{B-A}{\sqrt{5} n^2} \le \frac{(\alpha_{\max} - \alpha_{\min}) l^2}{\sqrt{5} (l n)^2}.
\]
Since $x$ and $y$ have the same initial vertex, we can form a new periodic orbit by composing $n_1$ copies of $x$ followed by $n_2$ copies of $y$. This creates a periodic orbit $z$ of orbit length $nl$ with the property that
\[
\frac{\p^{nl}(z)}{nl} = \frac{n_1 A + n_2 B}{n_1+n_2} \ \text{ and so } \ \left| \frac{\p^{nl}(z)}{nl}  - \eta \right| \leq  \frac{(\alpha_{\max} - \alpha_{\min})l^2}{\sqrt{5} |z|^2}.
\]
Since we can run this construction for infinitely many $n$, the result follows.
\end{proof}

To obtain uniformity over $\eta$, that is, to show the existence of $C$ in Theorem \ref{thm.ldtsi}, we need to upgrade Lemma \ref{lem.unifint} using Lemma \ref{lem.construct}.

\begin{proposition} \label{prop.lciuniform}
        Suppose that $(\Sigma,\sigma)$ is a mixing subshift of finite type and that $\p : \Sigma \to \R$ is a function that is constant on $2$ cylinders. Then there exists $C >0$ such that for any $\eta \in (\alpha_{\min}, \alpha_{\max})$ there exists an infinite sequence of periodic orbits $x_n \in \Sigma$ such that
    \[
    \left| \frac{\p^{|x_n|}(x_n)}{|x_n|}  - \eta \right| \le \frac{C}{|x_n|^2}.
    \]
    Furthermore we can take
    \[
    C = \frac{4M^2 (1+k^2)^2 (\alpha_{\max} - \alpha_{\min})}{\sqrt{5}}.
    \]
\end{proposition}
\begin{proof}
    We can assume that $\p$ is not cohomologous to a constant function (otherwise the conclusion is clear), so that $\al_{\min}<\al_{\max}$.

            We split the interval $(\alpha_{\min}, \alpha_{\max}) = I_1 \cup I_2$ into the two (non-disjoint) intervals
            \[
            I_1 = \left(\alpha_{\min}, \frac{\p^{|y|}(y)}{|y|}\right) \ \text{ and } \ I_2 = \left(\frac{\p^{|x|}(x)}{|x|}, \alpha_{\max}\right)
            \]
            where $x,y$ are the orbits constructed in Lemma \ref{lem.construct}. In particular, both $|x|$ and $|y|$ are bounded above by $2M(1+k^2)$. We can now apply Lemma \ref{lem.unifint} to both of the intervals $I_1$ and $I_2$ to deduce the result.
\end{proof}                 

\begin{definition}
Let $\p,\Sigma$ be as above. Suppose  $w \in (\alpha_{\min}, \alpha_{\max})$ is chosen so that there exists $x$ with $\sigma^n(x) = x$ and $\p^n(x) = nw$. We define $d(\p,w)$ to be the greatest common divisor of all numbers $n \ge 1$ such that
    \[
\#\left\{ x\in\Sigma_{\mathcal{C}} : \sigma^n(x) = x, \p^n(x) = wn \right\} > 0.
\]
If $\#\left\{ x\in\Sigma_{\mathcal{C}} : \sigma^n(x) = x, \p^n(x) = wn \right\} =0$ for all $n$ then we set $d(\p,w) = 0$.
\end{definition}
Note that $d(\p, w) = 0$ for all but countably many values of $w.$ This is because the values for which $d(\p, w) > 0$ are contained in the rational span of the values that the averaged Birkhoff sum of $\p$ attains on simple cycles.
 By a result of Marcus and Tuncel \cite[Thm.~14]{marcus.tuncel} for any $\xi>0$ and any closed subset $W \subset ( \alpha_{\min}, \alpha_{\max}) $
there exist $r, N \in \Z_{\ge 0}$ and $\delta > 0$ such that, for any $n \ge N$ with $d(\p,w)|n$,
\begin{equation}\label{eq.mt}
    \#\left\{ x\in\Sigma : \sigma^n(x) = x, \p^n(x) = wn \right\} \ge \delta n^{-r} e^{n(h-(\calL(\p, w)) - \xi)}
\end{equation}
for each $w\in W$ satisfying $d(\psi,w)>0$. 
To make use of this result we need to control the values of $d(\p,w)$ as $w$ takes values in a shrinking interval.

\begin{lemma}
      There exists $C >0$ such that for any $\eta \in (\alpha_{\min}, \alpha_{\max})$ we can find a sequence $w_n \in (\alpha_{\min}, \alpha_{\max})$ and a sequence $x_n \in \Sigma$ of periodic orbits such that
    \[
    w_n \in \left[ \eta - \frac{C}{|x_n|^2}, \eta +\frac{C}{|x_n|^2} \right] \ \text{ with } \ \frac{\p^{|x_n|}(x_n)}{|x_n|} = w_n
    \]
    and $d(\p,w_n)||x_n|$.
\end{lemma}
\begin{proof}
    By Lemma \ref{prop.lciuniform} there exist $C >0, N\ge 1$ depending only on $\p, \Sigma$ such that for any $\eta \in (\alpha_{\min}, \alpha_{\max})$ and $n \ge N$ there is a sequence of periodic orbits $x_n \in \Sigma$ with periods $|x_n|$ (i.e. $\sigma^{|x_n|}(x_n) = x_n$) such that
    \[
    \left| \frac{\p^{|x_n|}(x_n)}{|x_n|} - \eta \right| \le \frac{C}{|x_n|^2}.
    \]
    We define $w_n = \frac{\p^n(x_n)}{n}$
    and note that  $d(\p, w_n)||x_n|$ if and only if
    \[
    \#\left\{ x \in \Sigma: \sigma^{|x_n|}(x) = x, \frac{\p^{|x_n|}(x_n)}{|x_n|} = w_n \right\} > 0.
    \]
    Hence we are done.
\end{proof}

For a sequence $(\d_n)_{n\geq 1}$ of positive numbers and $n \ge 1$ we define
\[
F_n(\eta, \d_n) = \left\{x\in \Sig : \sigma^n(x) = x, \left|\frac{\p^n(x)}{n} - \eta\right| < \d_n \right\}
\]
for $\epsilon >0, n\ge 1$ and $\eta \in (\alpha_{\min}, \alpha_{\max})$. 

\begin{proposition}\label{thm.main}
Suppose that $(\Sigma, \sigma)$ is a mixing subshift of finite type and that $\p: \Sigma \to \R$ is a function that is constant on $2$ cylinders. Then there exists $C>0$ such that for any  $\eta \in (\alpha_{\min}, \alpha_{\max})$ and $ \xi > 0$ there exist $\d, r, M, \epsilon >0$ and a sequence of integers $n_l$ with $n_l \to \infty$ as $l \to \infty$ such that
    \[
    \#F_{n_l}(\eta, C n_l^{-2}) \ge \d n_l^{-r} e^{n_l ((h-\calL(\p,\eta)) -\xi) }
    \]
    for all $l \ge 1.$  
\end{proposition}

\begin{remark}
For $\eta$ with $d(\p, \eta) > 0$ this result follows immediately from estimates due to Marcus and Tuncel \cite{marcus.tuncel}. The main strength of the above estimates are that they hold for all values of $\eta \in (\alpha_{\min}, \alpha_{\max})$.
\end{remark}

\begin{proof}[Proof of Proposition \ref{thm.main}]
    We use Lemma \ref{prop.lciuniform} to find a sequence of periodic orbits $(x_n)_n$ such that 
\[
|w_n - \eta| \le \frac{C}{|x_n|^2} \ \text{ where } \ w_n = \frac{\p^{|x_n|}(x_n)}{|x_n|}
\]
for each $n \ge 1$. By \eqref{eq.mt}, for any $\xi >0$ and for all $n$ sufficiently large we have that
\[
 \#\left\{ x \in \Sigma: \sigma^{|x_n|}(x) = x \text{ and } \left| \frac{\p^{|x_n|}(x_n)}{|x_n|} - \eta \right| < \frac{C}{|x_n|^2} \right\} \ge \delta |x_n|^{-r} e^{|x_n|((h-\calL(\p,w_n)) - \xi)}
\]
for some $\delta, r >0$.  By analyticity of $\calL$ there is $\widetilde{C} >0$ such that
\[
|\calL(\p,\eta) - \calL(\p,w_n)| \le \widetilde{C}|\eta - w_n| = O(|x_n|^{-1})
\]
(where the implied error constant is independent of $n$) and this implies our estimate. This concludes the proof.
\end{proof}

\begin{proof}[Proof of Theorem \ref{thm.ldtsi}]
Note that the Large Deviation Principle implies that
\[
\limsup_{n\to\infty} \frac{1}{n} \log \left(\#\left\{ x\in\Sigma_{\mathcal{C}} : \sigma^n(x) = x, \left| \frac{\p^n(x)}{n} -  \eta\right| < \frac{C'}{n^2} \right\} \right)\le h-\calL(\p, \eta)
\]
for any $C' >0$ and $\eta \in (\alpha_{\min}, \alpha_{\max})$.
Proposition \ref{thm.main} also implies that for $C > 0$ (as in the proposition)  
\begin{equation}\label{eq.egr}
    \limsup_{n\to\infty} \frac{1}{n} \log \left(\#\left\{ x\in\Sigma_{\mathcal{C}} : \sigma^n(x) = x, \left| \frac{\p^n(x)}{n} -  \eta\right| < \frac{C}{n^2} \right\} \right)\ge h-\calL(\p, \eta).
\end{equation}
This proves part of the theorem. The other part follows similarly, and the estimate for $C$ follows from Proposition \ref{prop.lciuniform}.

We now finish by proving the final statements. Without loss of generality we can, by Remark \ref{rem.lattice}, assume that $\p$ takes values in the integers. When this is the case, the set $
\left\{ \frac{\p^n(x)}{n} : \sigma^n(x) = x \right\}$
contains rational numbers all with denominator at most $n$. By \cite{hurwitz} there is $\eta \in (\alpha_{\min}, \alpha_{\max})$ such that if $\epsilon > 0$ is sufficiently small then for all but finitely many values of $n$,
\[
\left|\frac{\p^n(x)}{n} - \eta \right| > \frac{\epsilon}{n^2} \ \text{ for all $x\in \Sigma$ with $\sigma^n(x) = x$.}
\]
Hence $F_n(\eta, \epsilon n^{-2})$ is empty for all $n$ sufficiently large. 
\end{proof}

\section{Large deviation for pairs of word metrics}\label{sec.ldwordmetric}

In this section we prove Theorem \ref{thm.wm}. The proof uses the Cannon automatic structure for hyperbolic groups. To avoid having to define all of these objects here, we refer the reader to \cite{cantrell.new}. Throughout the proof, we will follow same terminology and notation as that used in \cite{cantrell.new}.

\begin{proof}[Proof of Theorem \ref{thm.wm}]
    Without loss of generality assume that $d_S, d_{S_{\ast}}$ are not roughly similar (i.e. there does not exist $\tau, C >0$ such that $|d_S(g,h) - \tau d_{S_\ast}(g,h)| < C$ for all $g,h \in \G$). By Lemma 3.8 and Example 3.9 in \cite{calegari-fujiwara} we can find a Cannon coding for $(\G, S)$ with corresponding shift space $(\Sigma,\sigma)$ and a H\"older continuous function $\psi: \Sigma \to \R$ satisfying the following: if $z = (z_0, z_1,\ldots, z_{n-1},z_1, \ldots)\in \Sigma$ is a periodic orbit of period $n$, then the $n$th Birkhoff sum of $\psi$ on $z$ outputs the value $\ell_{S_\ast}[g]$, where $g\in \G$ is the group element obtained by multiplying the labelings in the finite path $z_0,z_1, \ldots, z_{n-1}$. By construction, $\psi$ is constant on $2$ cylinders and takes values in $\Z$. 

    Fix a maximal component $\Sigma_\mathcal{C}$ in $\Sigma$.
    We know that the function $-\psi$ satisfies a large deviation principle over the periodic orbits on $\Sigma_\mathcal{C}$. Furthermore, the rate function $\calI$ is the Legendre transform of the Manhattan curve for $d_S, d_{S_\ast}$ by \cite[Lem.~3.3]{cantrell.new}. In particular,
      \[
    \al_{\min}=\inf_{\sigma^n(x) = x} \frac{\psi^n(x)}{n} = \Dil(S,S_\ast)^{-1}  \ \text{ and } \ \al_{\max}=\sup_{\sigma^n(x) = x} \frac{\psi^n(x)}{n} =\Dil(S_\ast,S) 
    \]
    where the infimum/supremum is taken over all periodic orbits in $\Sigma_\mathcal{C}$.
    Now, by Theorem \ref{thm.ldtsi} we obtain the same limits with shrinking intervals, i.e.  $\eta \in (\alpha_{\min}, \alpha_{\max})$
\[
  \limsup_{n\to\infty} \frac{1}{n} \log\left( \#\left\{ x \in \Sigma: \sigma^n(x) = x \text{ and } \left| \frac{\psi^n(x)}{n} - \eta \right| \le \frac{C}{n} \right\} \right) = \calI(\eta) 
\]
for some $C$ independent of $\eta$.
It is possible that the system $(\Sigma_{\mathcal{C}}, \sigma)$ is not mixing but is instead transitive. This is no issue as explained in Remark \ref{rem.transitive}.
We know that each periodic orbit has a corresponding conjugacy class, and \cite[Lem.~4.2]{cantrell.reyes.2} says that the periodic orbits overcount the number of conjugacy classes by at most a linear factor in $|x|_S$. Hence we must have 
  \[
    \limsup_{n\to\infty} \frac{1}{n} \log \#\left\{ [g] \in \conj: \ell_S[g] = n, | \ell_{S_\ast}[g]  - \eta \ell_{S}[g] | < \epsilon \right\} = \calI(\eta) \le v_S
    \]
  for any $\epsilon > 0$ and $\eta \in (\alpha_{\min}, \alpha_{\max})$. By \cite[Thm.~1.1]{cantrell.tanaka.1} we have that $\calI$ is maximised (and equals $v_S$) when 
  \[
    \eta = \lim_{T\to\infty} \frac{1}{\#\{[g]\in \conj \colon \ell_S[g]<T\}} \sum_{\ell_S[g] < T} \frac{\ell_{S_\ast}[g]}{T}.
    \]
  This concludes the proof.
\end{proof}

As a corollary we deduce Corollary \ref{coro.approx}.
\begin{proof}[Proof of Corollary \ref{coro.approx}]
    The first part of the theorem, when $\eta$ is irrational, follows as a direct corollary of Theorem \ref{thm.wm}. To deduce the final statement when $\eta$ is rational it suffices to show that, when $(\Sigma, \sigma)$ is a mixing subshift of finite type and $\p: \Sigma \to \Z$ is a function that is constant of $2$ cylinders, then the following holds:
 if $p/q \in (\alpha_{\min}, \alpha_{\max}) \cap \Q$ then there exists a periodic orbit $x \in \Sigma$ such that $\p^{|x|}(x) = |x|p/q.$ It is an easy exercise to verify this and so we leave it to the reader.
\end{proof}


\begin{example}\label{sec.example}
Let $F_2 = \langle a, b \rangle$ be the free group on two generators and consider the generating sets
\[
S = \{a,b,a^{-1}, b^{-1}\} \ \text{ and } \ S_\ast = \{a,b,ab,a^{-1}, b^{-1}, (ab)^{-1} \}.
\]
The corresponding word metrics have exponential growth rates $v_{S_\ast}= \log(4)$ and $v_S = \log(3)$.
The Manhattan curve $\theta_{S/S_\ast}$ was computed in \cite{cantrell.tanaka.1} and is given by
\[
\theta_{S/S_\ast}(t) = \log\left( \frac{1}{2} e^{-t} \left(e^{-t} + \sqrt{e^{-t}(e^{-t}+8)} + 4\right)\right).
\]
From the definition we see that $\calI(v_{S_\ast}/v_S) = \log(4) - \Lambda/\log(3)$ where $\Lambda$ is the constant
\[
\Lambda = \sup_{t\in\R} \left\{\log(4) - \frac{\log(4)}{\log(3)} t - \theta_{S/S_\ast}(t) \right\}.
\]
This can be computed (using, say Wolfram Alpha) to be
\small
\[
 \log(16) \log\left(\log\left(\frac{4}{3}\right)\right) + \log(9) \log\left(\frac{\log\left(\frac{3}{2}\right)}{\log(\frac{4}{3})}\right) + \log(4) \left(\log(2) - \log\left(\log\left(\frac{3}{2}\right) \log(2)\right)\right).
\]
\normalsize

Theorem \ref{thm.wm} then implies that there is $C >0$ such that
\begin{align*}
     \limsup_{T\to\infty} \frac{1}{T} \log \left(\#\left\{ [g] \in \conj: \ell_{S_\ast}[g] < T, |v_S\ell_S[g] - v_{S_\ast}\ell_{S_\ast}[g]| \le \frac{C}{T}  \right\}\right) &= \log(4) - \frac{\Lambda}{\log(3)} \\
     &\approx 1.3679878759
\end{align*}
and similarly we also have
\begin{align*}
  \limsup_{T\to\infty} \frac{1}{T} \log \left(\#\left\{ [g] \in \conj: \ell_{S}[g] < T, |v_S\ell_S[g] - v_{S_\ast}\ell_{S_\ast}[g]| \le \frac{C}{T}  \right\}\right)  &= \log(3) - \frac{\Lambda}{\log(4)}\\
  &\approx 1.0841047424.
\end{align*}
In this case we can set $k = 6, M=2, \alpha_{\max} = 2, \alpha_{\min} =1$ by \cite{cantrell.tanaka.1}. 
Also, it was computed in \cite{cantrell.tanaka.1} that $ -\theta'_{S/S_\ast}(0) = 4/3$. 
We therefore have that
\[
\frac{4M^2(1+k^2)^2(\alpha_{\max} - \alpha_{\min}) }{\sqrt{5}} = \frac{4(2^2)(1+6^2)(2-1)}{\sqrt{5}} = \frac{592}{\sqrt{5}} \approx 264.7 \le 300
\]
and
\[
\limsup_{T\to\infty} \frac{1}{T} \log \left(\#\left\{ [g] \in \conj: \ell_{S_\ast}[g] < T, |3\ell_S[g] - 4 \ell_{S_\ast}[g]| \le \frac{300}{T}  \right\}\right) = v_{S_\ast}.
\]
Therefore the length spectra of $3d_S$ and $4d_{S_\ast}$ are within distance $300$ on a set of full exponential growth rate for $d_{S_\ast}$. 
In particular, for any $\eta \in (1,2)$ we can find an infinite sequence of conjugacy classes $[g_n] \in \conj(\G)$ such that
\[
\left| \frac{\ell_{S}[g_n]}{\ell_{S_\ast}[g_n]} - \eta\right| \le \frac{300}{|g_n|_{S_\ast}^2}.
\]
\end{example}


\section{Encoding cubulations via finite-state automata}\label{sec.cubulation}

In this section we study further the class $\frakG$ defined in the introduction. We prove Proposition \ref{prop.groupsfornicecubulation}, which provides plenty of examples of groups in this class. Then we construct a finite-state automaton that encodes a pair of cubulations of a group in $\frakG$, our main result being Proposition \ref{prop.languageL_X}. This will be done in the greater generality of the class $\frakX$ of pairs of compatible cubulations. Proposition \ref{prop.languageL_X} is key to prove our main results in Section \ref{sec.thmscub}.


\subsection{The classes $\frakG$ and $\frakX$}
Throughout this and the next section we will work with the following notion of compatibility of pairs of group actions on $\CAT(0)$ cube complexes.

\begin{definition}\label{def.classX}
    Let $\frakX$ be the class of triples $(\G,\calX,\calX_\ast)$, where $\G$ is a non-virtually cyclic group acting cocompactly on the $\CAT(0)$ cube complexes $\calX,\calX_\ast$ and satisfying:
\begin{enumerate}
    \item the action on $\calX$ is proper and virtually co-special;
    \item every hyperplane stabilizer for the action on $\calX_\ast$ is convex-cocompact for the action on $\calX$; and,
    \item the action of $\G$ on $\calX$ has a contracting element.
\end{enumerate}
\end{definition}
Note that in the definition above we do not require the action on $\calX_\ast$ to be proper. 

Since hyperplane stabilizers are always convex-cocompact and being a contracting element does not depend on the cubulation \cite[Lem.~4.6]{genevois}, it follows that $(\G,\calX,\calX_\ast)\in \frakX$ whenever $\G\in \mathfrak{G}$ and $\calX,\calX_\ast$ are cubulations of $\G$. 

The main result of this section is the following proposition, which tells us that the class $\frakG$ is much bigger than the class of cubulable hyperbolic groups.

\begin{proposition}\label{prop.groupsfornicecubulation}
The following classes of groups are contained in $\mathfrak{G}$.
\begin{itemize}
    \item[i)] Hyperbolic cubulable groups that are non-elementary.
    \item[ii)] Right-angled Artin groups of the form $\G=A_G$, where $G$ is a finite graph with more than one vertex, that is not a join and such that $\st(v)$ is not contained in $\st(w)$ for every pair of distinct vertices $v,w$ of $G$. 
    \item[iii)] Right-angled Coxeter groups that are not virtually cyclic or direct products and are of the form $\G=W_G$, where $G$ is finite and does not have any \emph{loose squares}.
\end{itemize}
Moreover, if $\G$ is cubulable and hyperbolic relative to groups belonging to $\mathfrak{G}$, then $\G$ belongs to $\mathfrak{G}$.
\end{proposition}

For $\G=W_G$ a right-angled Coxeter group, a \emph{loose square} is a full subgraph $\Del\subset G$ that is a square, and such that for every maximal subgraph $\Lam \subset G$ with $W_\Lam$ virtually abelian, either $\Del \subset \Lam$ or
$\Del \cap \Lam$ generates a finite subgroup of $\G$. 

Many non-relatively hyperbolic right-angled Artin and Coxeter groups satisfy assumptions i) and ii) of Proposition \ref{prop.groupsfornicecubulation}. Indeed, a RAAG is non-trivially relatively hyperbolic if and only if it is a non-trivial free product, and hence RAAGs with underlying graphs $n$-agons for $n\geq 5$ are not relatively hyperbolic and belong to $\frakG$.

For the proof of Proposition \ref{prop.groupsfornicecubulation} we require two results, the first one being an observation about virtual specialness.

\begin{lemma}\label{lem.special}
    Let $\calX,\calX_\ast$ be cubulations of the group $\G$ and assume that the action on $\calX$ is virtually co-special. If every hyperplane stabilizer for the action on $\calX_\ast$ is convex-cocompact with respect to $\calX$, then the action of $\G$ on $\calX_\ast$ is virtually co-special. 
\end{lemma}

\begin{proof}
    If $\calX$ and $\calX_\ast$ satisfy the assumptions of the lemma, then all the double cosets of the hyperplane stabilizers for the action of $\G$ on $\calX_\ast$ are separable by \cite[Thm.~A.1]{reyes.cubrh}. Then the action of on $\calX_\ast$ is virtually co-special by the double-cosets criterion \cite[Thm.~9.19]{haglund-wise.special}.
\end{proof}

The second result we need is a criterion of convex-cocompactness for subgroups of cubulable relatively hyperbolic groups, which may be of independent interest and whose proof is postponed to the appendix.

\begin{proposition}\label{prop.criterioncvxcc}
    Let $\G$ be a relatively hyperbolic group acting properly and cocompactly on the $\CAT(0)$ cube complex $\calX$. Then the following are equivalent for a subgroup $H< \G$.
    \begin{enumerate}
        \item $H$ is convex-cocompact for the action of $\G$ on $\calX$.
        \item $H$ is relatively quasiconvex and $H\cap P$ is convex-cocompact for the action of $\G$ on $\calX$ for any maximal parabolic subgroup $P<\G$.    
    \end{enumerate}
\end{proposition}

\begin{proof}[Proof of Proposition \ref{prop.groupsfornicecubulation}]
    By Agol's Theorem \cite[Thm.~1.1]{agol} every cubulation of a hyperbolic group is virtually co-special. Moreover, the class of convex-cocompact subgroups for any such cubulation coincides with the class of quasiconvex subgroups \cite[Prop.~7.2]{haglund-wise.special}. Since any loxodromic element in a hyperbolic group is contracting, this solves the proposition for groups in class i). 

    Now, let $\G$ be a group belonging to the class ii) (respectively iii)). In \cite[Thm.~A]{FLS} (resp. \cite[Cor.~C]{FLS}) it was proven that $\G$ has a unique \emph{cubical coarse median structure}. By \cite[Thm.~2.15]{FLS} this uniqueness result is equivalent to the class of convex-cocompact subgroups being independent of the cubulation. Since finitely generated right-angled Artin (resp. Coxeter) groups are virtually special, $\G$ satisfies Items (1) and (2) of Definition \ref{def.frakG}. The existence of a contracting element for $\G$ for a geometric group action on a $\CAT(0)$ cube complex follows from \cite[Sec.~5.2]{yang} and the references therein. 

    To prove the moreover statement, let $\G$ be a group that is hyperbolic relative to groups belonging to $\mathfrak{G}$ and let $\calX,\calX_\ast$ be two cubulations of $\G$. If $P<\G$ is a maximal parabolic subgroup, then by \cite[Thm.~1.1]{sageev-wise} it has convex cores $Z_P\subset \calX$ and $(Z_P)_\ast\subset \calX_\ast$. Also, since $P$ belongs to $\mathfrak{G}$, by Lemma \ref{lem.special} the action of $P$ on $Z_P$  is virtually co-special. As this holds for every maximal parabolic subgroup, the action of $\G$ on $\calX$ is virtually co-special either by \cite[Thm.~A]{groves-manning.special} or \cite[Thm.~1.2]{reyes.cubrh}. 
    
    Consider a group $H<\G$ that is convex-cocompact for the action of $\G$ on $\calX$. By Proposition \ref{prop.criterioncvxcc} $H$ is relatively quasiconvex and $H\cap P$ is a convex-cocompact subgroup for any maximal parabolic subgroup $P<\G$. This implies that $H\cap P<P$ is convex-cocompact for the action of $P$ on $Z_P$, see. e.g. Lemmas 2.14 and 2.15 in \cite{reyes.cubrh}. But each such $P$ belongs to $\mathfrak{G}$, so $H\cap P$ is also convex-cocompact for the action of $P$ on $(Z_P)_\ast$, implying that $H\cap P$ is convex-cocompact for the action of $\G$ on $\calX_\ast$. By Proposition \ref{prop.criterioncvxcc} we deduce that $H$ is convex-cocompact for the action of $\G$ on $\calX_\ast$, so that the actions on $\calX$ and $\calX_\ast$ have the same sets of convex-cocompact subgroups. Since any contracting element in a maximal parabolic subgroup is contracting for $\G$, we have proven that $\G$ belongs to $\mathfrak{G}$.  
\end{proof}

\begin{example}
Let $M$ be a cusped hyperbolic 3-manifold with cusps $V_1,\dots,V_r\subset M$. We affirm that $\G=\pi_1(M)$ does not belong to $\frakG$.
For each $i=1,\dots,r$ choose distinct slopes $\al_i,\beta_i$ for the cusp $V_i$. Equivalently, each pair $\al_i,\beta_i$ represents a pair of cyclic subgroups (up to commensurability) such that their union generates a finite-index subgroup of $\pi_1(V_i)$. By \cite[Cor.~1.3]{cooper-futer}, there exists a cubulation $\calX$ of $\G$ such that each $\al_i$ and $\beta_i$ represents a convex subgroup. Since any cubulation of $\Z^2$ has a subgroup that is not convex-cocompact (this follows for instance from \cite[Thm.~3.6]{wise-woodhouse}), for $\G$ as above we can produce two cubulations $\calX,\calX_\ast$ not satisfying (2) in Definition \ref{def.frakG}. 

However, the data of the slopes that are convex-cocompact in cubulations of $\G$ are the only obstruction for them determining triplets in $\frakX$. That is, if $\calX,\calX_\ast$ are two cubulations of $\G$, then $(\G,\calX,\calX_\ast)\in \frakX$ if and only if they have the same pairs of convex-cocompact slopes for each cusp subgroup. The proof of this follows the same lines as the proof of Proposition \ref{prop.groupsfornicecubulation}.
\end{example}

\begin{example}
    There exist groups $\G$ not necessarily belonging to $\frakG$ for which we still can find essentially distinct cubulations $\calX,\calX_\ast$ such that $(\G,\calX,\calX_\ast)\in \frakX$. 

    As an example, let $\G$ be a finite graph product of finite groups and let $\calX$ be the graph-product complex with the standard action by $\G$. If $\phi\in \textnormal{Aut}(\G)$ is any automorphism, then we let $\calX_\ast$ be the cubulation obtained by precomposing by $\phi$ the action of $\G$ on $\calX$. 
    Then, as long as $\G$ has a contracting element, $(\G,\calX,\calX_\ast)\in \frakX$ by combining  Theorem D (1) and Theorem 2.15 in \cite{FLS}. We note that there are many right-angled Coxeter groups with \emph{large} outer automorphism groups \cite{sale-susse}.

    If instead we apply \cite[Thm.~D (2)]{FLS}, the same conclusion holds for $\G$ any Coxeter group with a contracting element, with $\calX$ being its Niblo-Reeves cubulation, and $\calX_\ast$ obtained from $\calX$ after twisting by an automorphism of $\G$.

\end{example}


\subsection{Constructing the appropriate automaton}

In this section we construct a finite-state automaton for a triple  $(\G,\calX,\calX_\ast)$ in $\mathfrak{X}$, our main result being Proposition \ref{prop.languageL_X}. Given such a triple, our first step is the construction of a cubulation for $\G$ that simultaneously encodes the actions on $\calX$ and $\calX_\ast$. This is the content of the next lemma. 

\begin{lemma}\label{lem.preparations}
       Given $(\G,\calX,\calX_\ast)\in \frakX$ there exists a proper, cocompact, and essential action of $\G$ on a $\CAT(0)$ cube complex $\calZ$, $\G$-invariant subsets $\bbW,\bbW_\ast\subset \bbH(\calZ)$, and a finite index subgroup $\ov\G<\G$ satisfying the following. 
       \begin{itemize}
           \item Let $\hat\calX=\calZ(\bbW)$ and $\hat\calX_\ast=\calZ(\bbW_\ast)$. Then $\hat\calX$ and $\hat\calX_\ast$ embed $\G$-equivariantly as convex subcomplexes of $\calX$ and $\calX_\ast$ respectively. In particular, the triple $(\G,\hat\calX,\hat\calX_\ast)$ belongs to $\frakX$ and we have the equalities $\ell_\calX= \ell_{\hat\calX}$ and $\ell_{\calX_\ast}=\ell_{\hat\calX_\ast}$.
           \item The action of $\ov\G$ on both $\calZ$ and $\hat\calX$ is co-special.
       \end{itemize}
\end{lemma}

\begin{proof}
    By \cite[Prop.~3.5]{caprace-sageev} let $\hat\calX$ and $\hat\calX_\ast$ be the $\G$-essential cores of $\calX$ and $\calX_\ast$ respectively, which $\G$-equivariantly embed in $\calX$ and $\calX_\ast$ as convex subcomplexes. 
    
    We claim that there exists a cubulation $\calZ'$ of $\G$ and $\G$-invariant subsets $\bbW,\bbW_\ast\subset \bbH(\calZ')$ so that $\hat\calX$ and $\hat\calX_\ast$ are $\G$-equivariantly isometric to $\calZ'(\bbW)$ and $\calZ'(\bbW_\ast)$. Under the additional assumption that the action on $\calX_\ast$ is proper, this is the content of the implication $(3) \Rightarrow (4)$ in \cite[Thm.~2.15]{FLS}, so we just sketch how to adapt the proof of this implication to the general case.
    
    The way $\calZ'$ is constructed in \cite[Thm.~2.15]{FLS} is by finding a $\G$-invariant and co-finite \emph{median subalgebra} $M\subset \hat\calX^0 \times \hat\calX_\ast^0$ containing the product $K\times K_\ast$ for $K,K_\ast$ finite sets such that $\G K=\hat\calX^0$ and $\G K_\ast=\hat\calX_\ast^0$. In the case the action on $\hat\calX_\ast$ is proper, this follows from local finiteness of $\hat\calX\times \hat\calX_\ast$ and the existence of some $\G$-invariant and cofinite median subalgebra $N\subset \hat\calX^0 \times \hat\calX_\ast^0$. On the other hand, the existence of such an algebra $N$ follows from the implication $(3) \Rightarrow (1)$ in \cite[Prop.~7.9]{fioravanti}. For this proposition, the assumption is that the action of $\G$ on $\hat\calX_\ast$ is essential and has only finitely many orbits of hyperplanes, and in that case $N$ is median algebra generated by the $\G$-orbit of any vertex in $\hat\calX^0\times \hat\calX_\ast^0$. However, the same argument in the proof of \cite[Prop.~7.9]{fioravanti} implies that the median algebra generated by the $\G$-translates of \emph{any} finite subset $\hat K$ of $\hat\calX^0\times \hat\calX_\ast^0$ is cofinite (this can be done by proving the appropriate generalization of \cite[Lem.~7.11]{fioravanti} so that the Claim in the proof of \cite[Prop.~7.9]{fioravanti} still holds). In particular, if the action of $\G$ on $\hat\calX_\ast$ is cocompact, we can choose $\hat K=K\times K_\ast$ as above, and then $\calZ'$ can be constructed from $N$ as in the rest of the proof of  \cite[Thm.~2.15]{FLS}. 
    
    To finish the proof of the first assertion, note that the action of $\G$ on $\calZ'$ may not be essential, so instead we consider the projection quotient $\calZ=\calZ'(\bbW\cup \bbW_\ast)$, which is essential and cocompact since the action of $\G$ on both $\hat\calX$ and $\hat\calX_\ast$ is essential and the action on $\calZ'$ is cocompact. The complex $\calZ$ still $\G$-equivariantly projects onto $\hat\calX$ and $\hat\calX_\ast$, so the action of $\G$ on $\calZ$ is proper because the action of $\G$ on $\hat\calX$ is proper.

    To prove the second assertion, note that by \cite[Thm.~2.15]{FLS} the actions of $\G$ on $\hat\calX$ and $\calZ$ have the same sets of convex-cocompact subgroups, so any hyperplane stabilizer for the action of $\G$ on $\calZ$ will be convex-cocompact with respect to $\hat\calX$. But the action of $\G$ on $\hat\calX$ is virtually co-special (it is a convex core for $\G$ acting on $\calX$), so $\calZ$ is virtually co-special by Lemma \ref{lem.special}. Finally, co-specialness is preserved under taking finite-index subgroups, so we can choose $\ov\G$ so that both quotients $\ov\G \bs \calZ$ and $\ov\G \bs \hat\calX$ are special.
\end{proof}

In virtue of the lemma above, throughout the rest of the section we will work under the following convention.

\begin{convention}\label{conv.restriction}
    $\G$ is a non-virtually cyclic group acting properly, cocompactly and essentially on the $\CAT(0)$ cube complex $\calZ$, and $\bbW\subset \bbH(\calZ)$ is a $\G$-invariant subset such that the action of $\G$ on $\calX=\calZ(\bbW)$ is proper, cocompact and essential. Let $\phi:\calZ \ra \calX$ be the restriction quotient. 
    
    Also, let $\ov\G<\G$ be a finite index subgroup such that the quotients $\ov\calZ=\ov\G \bs \calZ$ and $\ov\calX=\ov\G \bs \calX$ are special cube complexes. 
    We fix a base vertex $\wtilde o\in \calZ$ and set $o=\phi(\wtilde o)\in \calX$. 
\end{convention}

Let $S_\calZ$ and $S_\calX$ be the set of oriented hyperplanes in $\ov\calZ$ and $\ov\calX$ respectively. By specialness all the hyperplanes in $\ov\calZ$ and $\ov\calX$ are 2-sided, so there exist two orientations for each hyperplane. Since each hyperplane in $S_\calX$ corresponds to a $\ov\G$-orbit of oriented walls in $\bbW\subset \bbH(\calZ)$, there is a natural injection of $S_\calX$ into $S_\calZ$, so often we will consider $S_\calX$ as a subset of $S_\calZ$. The \emph{label} of an oriented hyperplane in $\calZ$ (resp. $\calX$) is its projection in $\ov\calZ$ (resp. $\ov\calX$).

We say that a word $w=\frakh_1\cdots \frakh_n$ in $(S_\calZ)^\ast$ is \emph{represented} by a (combinatorial) path $\gam=(\gam_0,\dots,\gam_n)$ in $\calZ$ if for each $i$ the oriented hyperplane $\frakh_i$ is the image in $\ov\calZ$ of the oriented hyperplane dual to the edge from $\gam_{i-1}$ to $\gam_i$. Similarly, we define when a word in $(S_\calX)^\ast$ is represented by a path in $\calX$. A consequence of specialness is that if a word is represented by two paths with the same initial vertex, then the paths must coincide. 

The next proposition gives us a finite-state automaton over $S_\calX$ that keeps track of the action of $\ov\G$ on both $\calX$ and $\calZ$.


\begin{proposition}\label{prop.languageL_X}
    There exists a language $L_\calX\subset (S_\calX)^\ast$ parametrized by the pruned finite-state automaton \[\calA_\calX=(\calG_\calX=(V_\calX,E_\calX),\pi_\calX,I_\calX,V_\calX)\] satisfying the following.
    \begin{enumerate}
        \item There exists $C\geq 1$ such that any $w\in L_\calX$ is represented by at most $C$ paths in $\calG_\calX$ starting at an initial state. 
        \item Every $w\in L_\calX$ is represented by a unique combinatorial geodesic $\gam_w\subset \calX$ starting at the vertex $o$. We let $\tau_\calX(w)$ denote the final vertex of $\gam_w$.
        \item The map $\tau_\calX:L_\calX \ra \calX^0$ is a bijection. 
        \end{enumerate}

        Moreover, there exist maps $\Psi:V_\calX \ra (S_\calZ \bs S_\calX)^\ast$ and  $\Xi:V_\calX \ra \ov\calX^0$ satisfying the following. 
        \begin{enumerate}
        \item[(4)] If $w\in L_\calX$ is represented by the path $\omega=(v_0\xrightarrow{e_1} v_1\cdots \xrightarrow{e_n} v_n)$ in $\calG_\calX$ starting at an initial state, then the concatenation
        \begin{equation}\label{eq.concatenation}            \al(\om):=\Psi(v_0)\pi_\calX(e_1)\Psi(v_1) \cdots \pi_\calX(e_n)\in (S_\calZ)^*
        \end{equation}
        can be represented by a unique geodesic path $\widetilde \gam _{\al(\om)}$ in $\calZ$ with initial vertex $\wtilde o$ and final vertex $\tau_\calZ(\al(\om))$, so that $\phi( \tau_\calZ(\al(\om)))=\tau_\calX(w)$.

         \item[(5)] If $w\in L_\calX$ is represented by the path $\omega=(v_0\ra v_1\cdots \ra v_n)$ in $\calG_\calX$ starting at an initial state, then the path $\gam_w=(\gam_0,\dots,\gam_n)$ in $\calX$ projects to $(\Xi(v_0),\dots,\Xi(v_n))$ in $\ov\calX$. 
         \end{enumerate}
\end{proposition}

We start the construction of $\calA_\calX$ by considering a regular language $L_\calZ\subset (S_\calZ)^\ast$ parametrizing geodesics in $\calZ$. For that we follow Sections 5.2 and 5.3 in \cite{li-wise}. 
More precisely, there exists a deterministic pruned finite-state automaton 
\[\calA_\calZ=(\calG_\calZ=(V_\calZ,E_\calZ),\pi_\calZ,\{\ast_\calZ\},V_\calZ)\]
over $S_\calZ$ parametrizing the language $L_\calZ$ and satisfying:
\begin{itemize}
    \item every $\widetilde w\in L_\calZ$ is represented by a unique geodesic path $\wtilde\gam_{\wtilde w}$ in $\calZ$ starting at $\wtilde o$ and ending at the vertex $\tau_\calZ(\wtilde w)$; and,  
    \item the map $\tau_\calZ: L_\calZ \ra \calZ^0$ is a bijection.
\end{itemize}

\begin{remark}\label{rmk.equivtoli-wise}
    The language constructed in \cite{li-wise} actually depends on an injection of $\ov\G$ into a right-angled Artin group $A_G$ inducing a $\ov\G$-equivariant isometric embedding of $\calZ$ into $R_G$ as a convex subcomplex, where $R_G$ is the universal cover of the Salvetti complex $\ov{R}_G$ associated to $G$. In that case, the language obtained is over the alphabet of oriented hyperplanes of $\ov{R}_G$. The language $L_\calZ$ described above is a particular case of this construction, when we consider the local isometric immersion  $\ov\calZ \ra \ov{R}_{G_{\ov\calZ}}$ for $G_{\ov\calZ}$ being the \emph{crossing graph} of $\ov\calZ$. For this immersion there is a natural bijection between $S_\calZ$ and the set of oriented hyperplanes in $\ov{R}_{G_{\ov\calZ}}$, see for instance \cite[Lem.~4.1]{haglund-wise.special}.
\end{remark}


We use the automaton $\calA_\calZ$ to produce a language $L_\calX$ over the alphabet $S_\calX$. We do this by first constructing an automaton $$\hat \calA_\calX=(\hat \calG_\calX=(\hat V_\calX,\hat E_\calX),\hat \pi_\calX,\hat I_\calX,\hat V_\calX)$$ as follows.

\begin{definition}
    Let $\hat V_\calX$ be the set of finite directed paths (possibly of length 0) in $\calG_\calZ$ of the form $\omega=(v_0\xrightarrow{e_1}\cdots \xrightarrow{e_n} v_n)$ and satisfying:
    \begin{itemize}
        \item $\pi_\calZ(e_i)\in S_\calZ \bs S_\calX$ for every $1\leq i\leq n$;
        \item either $v_0=\ast_\calZ$ or there exists an edge in $E_\calZ$ with label in $S_\calX$ and final vertex $v_0$; and, 
        \item either there exists an edge in $E_\calZ$ with label in $S_\calX$ and initial vertex $v_n$ or there are no edges in $E_\calZ$ with initial vertex $v_n$. 
    \end{itemize}
We consider an edge $\hat e$ from the vertex $\omega$ to the vertex $\omega'$ in $\hat V_\calX$ if there exists an edge $e\in E_\calZ$ with $\pi_\calZ(e)\in S_\calX$ and such that the concatenation $\omega e \omega'$ is a path in $\calG_\calZ$. We define $\hat \pi_\calX(\hat e):=\pi_\calZ(e)$ and let $\hat E_\calX$ be the set of all of the edges defined in this way. Finally, a vertex of $\hat V_\calX$ is an initial state if its initial vertex (as a path in $\calG_\calZ$) is $\ast_\calZ$. We let $\hat I_\calX$ be the set of all the initial states. 

\end{definition}

\begin{lemma}\label{lem.hatA_X}
    The set $\hat V_\calX$ is finite and non-empty.
    Therefore, $\hat \calA_\calX$ defines a pruned finite-state automaton over $S_\calX$.
\end{lemma}

\begin{proof}
The set $S_\calX$ is non-empty because $\G$ is non-elementary and $\calX$ is essential, and since $\calA_\calZ$ is pruned we can find a vertex in $\hat I_\calX$, so that $\hat V_\calX$ is non-empty. 

To show finiteness, let $M$ be the maximum cardinality of a preimage $\phi^{-1}(x)\cap \calZ^0$ among $x\in \calX^0$, which is finite since $\phi$ is $\G$-invariant and the action of $\G$ on $\calZ$ is cocompact. If $\omega=(v_0\xrightarrow{e_1}\cdots \xrightarrow{e_n} v_n)$ is a vertex in $\hat V_\calX$, the fact that $\calA_\calZ$ is pruned implies the existence of a geodesic path $\wtilde \gam$ in $\calZ$ representing the word $\pi_\calZ(e_1)\cdots \pi_\calZ(e_1)\in (S_\calZ \bs S_\calX)^\ast\subset (S_\calX)^\ast$. Since $\phi$ collapses the hyperplanes not belonging to $\bbW$, the image $\phi(\wtilde \gam)$ consists of a single point, implying that $n+1 \leq M$. We conclude that every vertex in $\hat V_\calX$ has uniformly bounded length as a path in $\calG_\calZ$, so $\hat V_\calX$ is finite because $\calG_\calZ$ is. 

Finally, $\calA_\calZ$ being pruned implies that $\hat \calA_\calX$ is pruned, and by construction this automaton is over the alphabet $S_\calX$.
\end{proof}

\begin{definition}\label{def.L_X} 
We let $L_\calX \subset (S_\calX)^*$ be the language parametrized by $\hat \calA_\calX$. 
\end{definition}

If 
$\hat \om =(\om_0\xrightarrow{\hat e_1} \cdots \xrightarrow{\hat e_n} \om_n)$ is a path in $\hat \calG_\calX$, then the concatenation
\begin{equation}\label{eq.conc}
    c(\hat \om):=\om_0e_1\cdots \om_{n-1}e_n
\end{equation}
is a path in $\calG_\calZ$. Let $\hat\al(\hat \om)\in (S_\calZ)^\ast$ be the word represented by $c(\hat \om)$. 




\begin{lemma}\label{lem.tauX} \begin{enumerate}
    \item If $\hat \om,\hat \om'$ are paths in $\hat \calG_\calX$ starting at an initial state and representing the same word $w\in L_\calX$, then $\phi(\tau_{\calZ}(\hat \al (\hat \om)))=\phi(\tau_{\calZ}(\hat \al (\hat \om')))\in \calX^0$. We denote this vertex by $\tau_\calX(w)$.
    \item Any $w$ in $L_\calX$ is represented by a unique geodesic path $\gam_w$ in $\calX$ starting at $o$ and ending at $\tau_\calX(w)$. 
    \item The map $\tau_\calX:L_\calX \ra \calX^0$ is a bijection.
\end{enumerate}
\end{lemma}

\begin{proof}
By induction on the length of $w$ we will prove simultaneously assertion (1) and that $d_\calX(o,\tau_\calX(w))$ equals the length of $w$. Suppose that $w$ has length $n$ and is represented by the paths 
\[\hat \om =(\om_0\xrightarrow{\hat e_1} \cdots \xrightarrow{\hat e_n} \om_n) \text{ and } \hat \om' =(\om_0'\xrightarrow{\hat e_1'} \cdots \xrightarrow{\hat e_n'} \om_n')\]
in $\hat \calG_\calX$. If $n=0$ then $\hat \al(\hat \om)=\om_0$ has no letters in $S_\calX$, and hence the projection of $\wtilde \gam_{\hat \al(\hat \om)}$ under $\phi$ consists of a single vertex, which must be $o$. As the same happens for $\om'$, this solves the base case. 

If $n\geq 1$ we consider $\hat \xi=(\om_0\xrightarrow{\hat e_1} \cdots \xrightarrow{\hat e_{n-1}} \om_{n-1})$ and $\hat \xi'=(\om_0'\xrightarrow{\hat e_1'} \cdots \xrightarrow{\hat e_{n-1}'} \om_{n-1}')$. Our inductive assumption implies that $\phi(\tau_{\calZ}(\hat \al (\hat \xi)))=\phi(\tau_{\calZ}(\hat \al (\hat \xi')))=:x$, so that no hyperplane with label in $S_\calX$ separates $\tau_{\calZ}(\hat \al (\hat \xi))$ and $\tau_{\calZ}(\hat \al (\hat \xi'))$.  Since $\ov\calX$ is special there exists at most one edge in $\calX$ with initial vertex $x$ and dual to a hyperplane with label $\frakh=\hat \pi_\calX(\hat e_{n})= \hat \pi_\calX(\hat e_{n}') \in S_\calX$. But $\phi$ is injective on $\bbW$, so that the hyperplane labeled $\frakh$ that separates $\tau_{\calZ}(\hat \al (\hat \xi))$ and $\tau_{\calZ}(\hat \al (\hat \om))$ in $\calZ$ is the same as the hyperplane labeled $\frakh$ that separates $\tau_{\calZ}(\hat \al (\hat \xi'))$ and $\tau_{\calZ}(\hat \al (\hat \om'))$. Since this is the only hyperplane with a label in $S_\calX$ that separates these pairs, we conclude that every hyperplane separating $\tau_{\calZ}(\hat \al (\hat \om))$ and $\tau_{\calZ}(\hat \al (\hat \om'))$ has a label in $S_\calZ \bs S_\calX$, which gives us $\phi(\tau_{\calZ}(\hat \al (\hat \om)))=\phi(\tau_{\calZ}(\hat \al (\hat \om')))$. Moreover, if $w'\in L_\calX$ is the word represented by $\hat \xi$, then by induction we have $d_\calX(o,\tau_\calX(w'))=n-1$ and $d_\calX(\tau_\calX(w'),\tau_\calX(\xi))=1$. But all these points belong to the projection under $\phi$ of the geodesic $\wtilde \gam_{\hat \al(\hat \om)}$, so Remark \ref{rmk.projquot} implies that $d_\calX(o,\tau_\calX(w))=n$ and concludes the proof by induction, proving (1). 

It is not hard to see that if $w=\frakh_1\cdots \frakh_n\in L_\calX$ then $$\gam_w:=(o,\tau_\calX(\frakh_1),\tau_\calX(\frakh_1\frakh_2),\dots,\tau_\calX(\frakh_1\cdots \frakh_n))$$ is the unique geodesic representing $w$ in $\calX$ and starting at $o$, which settles (2). 

To prove (3), injectivity can be deduced by induction on the length of words in $L_\calX$ combined with the fact that no two distinct edges with the same initial vertex in $\calX$ are dual to hyperplanes with the same label in $S_\calX$. This last statement is true by specialness of $\ov\calX$. 

To prove surjectivity, let $x\in \calX^0$ and consider $\wtilde w\in L_\calZ$ such that $\phi(\tau_\calZ(\wtilde w))=x$. Such an $\wtilde w$ exists because both $\tau_\calZ$ and $\phi$ are surjective. We write $\wtilde w=w_0e_1\cdots e_nw_n,$
where each $e_i$ is a letter in $S_\calX$ and each $w_i$ is a (possibly empty) word in $(S_\calZ \bs S_\calX)^\ast$. 
Then $\wtilde w':=w_0e_1\cdots w_{n-1}e_n$ equals $\hat \al (\hat \om)$ for some path $\hat \om$ in $\hat \calG_\calX$ representing the word $w\in L_\calX$, for which $\tau_\calX(w)=\phi(\tau_\calZ(\wtilde w'))=\phi(\tau_\calZ(\wtilde w))=x$.     
\end{proof}

\begin{lemma}\label{lem.almostdeterministic}
    There exists $C\geq 1$ such that every $w\in L_\calX$ is represented by at most $C$ paths in $\hat \calG_\calX$ starting at an initial vertex.
\end{lemma}
\begin{proof}
Since $\tau_\calZ$ is injective and $\phi$ is uniformly finite-to-one when restricted to vertices, 
it is enough to prove that the assignment $\hat \om \mapsto \hat \al(\hat \om)$ from the paths in $\hat \calG_\calX$ starting at an initial vertex into $L_\calZ$ is uniformly finite-to one. To show this, note that such $\hat \om$ is completely determined by its concatenation $c(\hat \om)$ in $\calG_\calZ$ (defined in \eqref{eq.conc}) and its final vertex in $\hat V_\calX$, and that $c(\hat \om)$ is completely determined by $\hat \al (\hat \om)$ since $\calA_\calZ$ is deterministic. The lemma then follows from Lemma \ref{lem.hatA_X}.
\end{proof}

\begin{proof}[Proof of Proposition \ref{prop.languageL_X}]

First we note that $\ov\calX^1$ can be seen as the finite-state automaton 
$$\ov\calX^1=((\ov\calX^0,S_\calX),\textnormal{Id}_{S_\calX},\{\ov o\}, \ov\calX^0)$$
over $S_\calX$, where $\ov o$ is the image of $o$ under the quotient $\calX \ra \ov\calX$. This automaton is deterministic since $\ov\calX$ is special. 


Let $\calA_\calX=(\calG_\calX=(V_\calX,E_\calX),\pi_\calX,I_\calX,V_\calX)$ be the fiber product of $\hat \calA_\calX$ and $\ov\calX^{1}$. That is, in $\hat V_\calX \times \ov\calX^0$ consider a directed edge from $(\hat \omega, \ov x)$ to $(\hat \omega',\ov x ')$ if there exists an edge $\hat e$ from $\omega$ to $\omega'$ such that $\hat \calG_\calX$ and $\hat \pi_\calX( \hat e)$ is the directed hyperplane dual to an edge from $\ov x$ to $\ov x'$ in $\ov \calX$ (so that $\ov x, \ov x'$ must be adjacent). By abuse of notation we will call this edge $e$ and define $\pi_\calX(e):=\hat\pi_\calX(\hat e)$. Let $I_\calX=\hat I_\calX \times \{\ov o\}$ be the set of initial states of $\calG_\calX$ and let $V_\calX \subset \hat V_\calX \times \ov \calX^0$ be the set of all the vertices in some directed path in $\hat V_\calX \times \ov \calX^0$ starting at an initial state. Let $E_\calX$ be the set of all the directed edges between vertices in $V_\calX$ as defined above. Clearly $\calA_\calX$ is pruned and finite.

There exists a label-preserving map $\mathfrak{p}$ from the set of paths in $\calG_\calX$ starting at an initial state into the set of paths in $\hat \calG_\calX$ starting at an initial state, which sends the path $((\om_0,\ov x_0)\xrightarrow{e_1} \cdots \xrightarrow{e_n} (\om_n,\ov x_n ))$ to the path $\hat \om=(\om_0\xrightarrow{\hat e_1} \cdots \xrightarrow{\hat e_n} \om_n)$. The map $\mathfrak{p}$ is a bijection since the sequence $\ov x_0,\dots,\ov x_n$ is the image of $\gam_w$ under $\calX \ra \ov\calX$, where $w$ is the word represented by $\hat \om$. In particular, the language parametrized by $\calA_\calX$ is precisely $L_\calX$, so Lemma \ref{lem.almostdeterministic} implies Item (1). Also, Items (2) and (3) follow from Lemma \ref{lem.tauX}. 

For Item (4) we consider the map $\Psi:V_\calX \ra (S_\calZ \bs S_\calX)^\ast$ that sends the vertex $(\om,\ov x)$ to the word represented by $\om$, seen as a path in $\calG_\calZ$. From the definition of $\hat \calG_\calX$ it is clear that for the path $\om =(v_0\xrightarrow{e_1} \cdots \xrightarrow{e_n}v_n)$ in $\calG_\calX$ starting at an initial state, the concatenation
$$\al(\om):=\hat \al(\mathfrak{p}(\om))=\Psi(v_0)\pi_\calX(e_1)\Psi(v_1) \cdots \pi_\calX(e_n)$$ belongs to $L_\calZ$, so it is represented by the geodesic path $\wtilde \gam _{\al(\om)}$ in $\calZ$ starting at $\wtilde o$. The word $w$ represented by $\om$ is also represented by $\mathfrak{p}(\om)$, so Lemma \ref{lem.tauX} implies that $\tau_\calX(w)=\phi(\tau_\calZ(\al(\om)))$. $\calA_\calZ$ being deterministic implies that $\wtilde \gam_{\al(\om)}$ is the unique geodesic in $\calZ$ starting at $\wtilde o$ and representing $w$.

Finally, we define $\Xi:V_\calX \ra \ov\calX^0$ as the coordinate projection $\Xi(\om, \ov x)=\ov x$, and the same argument for the proof that $\mathfrak{p}$ is a bijection implies Item (5). 
\end{proof}



\section{Proof of the main theorems}\label{sec.thmscub}

In this section we prove our main results about pairs of actions on $\CAT(0)$ cube complexes from the introduction. They are consequences of more general statements, given by Theorems  \ref{thm.manhattanfactsX} and \ref{thm.cubulation}. The strategy is to use the automaton from Proposition \ref{prop.languageL_X} to define an appropriate suspension flow, and then use Proposition \ref{prop.manattanpressure} to relate the Manhattan curves with pressure functions for potentials on this suspension. This relation will allow us to use the tools from symbolic dynamics and thermodynamic formalism discussed in Section \ref{sec.ld} to deduce our main results.

The following are the main theorems of the section, and they are proven in Section \ref{subsec.analytic+LDcub}. 
For their statements, we interpret the quantities $\Dil(\calX,\calX_\ast)^{-1}$ and $v_{\calX^\frakw}/v_{\calX^{\frakw_\ast}_\ast}$ as zero if the action of $\G$ on $\calX_\ast$ is not proper.

\begin{theorem}\label{thm.manhattanfactsX}
Let $(\G,\calX,\calX_\ast) \in \frakX$ and let $\frakw,\frakw_\ast$ be $\G$-invariant orthotope structures on $\calX$ and $\calX_\ast$ respectively.  
Then the Manhattan curve $\thet_{\calX^{\frakw_\ast}_\ast/\calX^\frakw}:\R \ra \R$ is convex, decreasing, and analytic. In addition, the following limit exists and equals $-\thet'_{\calX^{\frakw_\ast}_\ast/\calX^\frakw}(0)$:
\[\tau(\calX^{\frakw_\ast}_\ast/\calX^\frakw):= \lim_{T\to\infty} \frac{1}{\# \frakC_{\calX^{\frakw}}(T)} \sum_{[g]\in \frakC_{\calX^\frakw}(T)} \frac{\ell^{\frakw_\ast}_{\calX_\ast}[g]}{\ell^\frakw_{\calX}[g]}.
    \]
Moreover, we always have
\[
\tau(\calX^{\frakw_\ast}_\ast/\calX^\frakw)\geq v_{\calX^\frakw}/v_{\calX^{\frakw_\ast}_\ast}.
\]
If the action of $\G$ on $\calX_\ast$ is proper then the following are equivalent:
\begin{enumerate}
    \item $\theta_{\calX^{\frakw_\ast}_\ast/\calX^\frakw}$ is a straight line;
    \item there exists $\Lam >0$ such that $\ell^\frakw_\calX[g] = \Lam  \ell^{\frakw_\ast}_{\calX_\ast}[g]$ for all $[g] \in \conj(\G)$; and,
    \item $\tau(\calX^{\frakw_\ast}_\ast/\calX^\frakw) = v_{\calX^\frakw}/v_{\calX^{\frakw_\ast}_\ast}$.
\end{enumerate}
\end{theorem}

\begin{theorem}\label{thm.cubulation}
Let $(\G,\calX,\calX_\ast)\in \frakX$. Then there exists an analytic function $$\mathcal{I}:[\Dil(\calX,\calX_\ast)^{-1}, \Dil(\calX_\ast,\calX)]\ra \R$$ and $C>0$ such that for any $\eta \in (\Dil(\calX,\calX_\ast)^{-1}, \Dil(\calX_\ast,\calX))$ we have
    \begin{equation}\label{eq.LDcubX}
   0 <  \limsup_{T\to\infty} \frac{1}{T} \log\left( \#\left\{ [g] \in \mathfrak{C}_\calX(T): | \ell_{\calX_\ast}[g]  - \eta \ell_{\calX}[g] | < \frac{C}{T} \right\}\right) = \calI(\eta) \le v_\calX.
   \end{equation}
 Furthermore, we have equality in the above inequality if and only if $\eta =\tau(\calX_*/\calX)$.
\end{theorem}
As we noted in the previous section, $(\G,\calX,\calX_\ast)\in \frakX$ whenever $\G\in \frakG$ and $\calX,\calX_\ast$ are cubulations of $\G$. From this we see that Theorems \ref{thm.manhattanfactsX} and  \ref{thm.cubulation} imply Theorems \ref{thm.manhattanfactsG} and \ref{thm.cubulation.Gversion} from the introduction.


\subsection{Manhattan curves for pairs of cubulations}

In this section we use the finite-state automaton given by Proposition \ref{prop.languageL_X} to describe the Manhattan geodesics for a pair of cubulations in terms of pressure functions. This will allow us to use thermodynamic formalism to prove our main results in the next section. For this section we keep the notation from the previous section and assume the following.

\begin{convention}\label{conv.X_ast}
Let $(\G,\calX,\calZ)$ be a triple satisfying Convention \ref{conv.restriction}. Consider a non-empty $\G$-invariant subset $\bbW_\ast\subset \bbH(\calZ)$ and set $\calX_\ast=\calZ(\bbW_\ast)$. Then the action of $\G$ on $\calX_\ast$ is cocompact and essential, but not necessarily proper. Let $\frakw$ and $\frakw_\ast$ be $\G$-invariant orthotope structures on $\calX$ and $\calX_\ast$ respectively. We also require the action of $\G$ on $\calX$ to have a contracting element. 

Let $S_{\calX_\ast}\subset S_\calZ$ be the set of all the directed hyperplanes in $\ov\calZ$ whose lifts to $\calZ$ correspond to hyperplanes in $\bbW_\ast$, and let $\phi_\ast:\calZ \ra \calX_\ast$ be the projection quotient with $o_\ast=\phi_\ast(\wtilde o)$. 
\end{convention}



Since the structures $\frakw,\frakw_\ast$ are $\G$-invariant, there exist natural weighting maps $\ov\frakw, \ov\frakw_\ast:S_{\calZ} \ra \R$
defined as follows. If $\frakh\in S_\calX\subset S_\calZ$ then $\ov \frakw(\frakh)=\frakw( \wtilde \frakh)$ for any lift of $\wtilde \frakh$ in $\calX$, and $\ov\frakw(\frakh)=0$ if $\frakh\in S_\calZ \bs S_\calX$. Analogously, if $\frakh\in S_{\calX\ast}\subset S_\calZ$ then $\ov \frakw_\ast(\frakh)=\frakw_\ast( \wtilde \frakh)$ for any lift of $\wtilde \frakh$ in $\calX_\ast$, and $\ov\frakw_\ast(\frakh)=0$ if $\frakh\in S_\calZ \bs S_{\calX_\ast}$. By abuse of notation we extend these weightings to $\ov \frakw, \ov \frakw_\ast: (S_\calZ)^\ast\ra \R$
by declaring the empty word to have weights 0 and assigning 
\[ \ov\frakw (\frakh_1\cdots \frakh_n)= \ov\frakw(\frakh_1)+\cdots + \ov\frakw(\frakh_n) \ \text{ and } \ \ov\frakw_\ast(\frakh_1\cdots \frakh_n)= \ov\frakw_\ast(\frakh_1)+\cdots + \ov\frakw_\ast(\frakh_n) \]
for a word $\frakh_1\cdots \frakh_n\in (S_\calZ)^\ast$ of positive length.

Let $\Sig^\ast$ be the set of all the finite directed paths in $\calG_\calX$. We can see $\Sig^\ast$ as a set of finite sequences of edges in $E_\calX$. Similarly, let $\Sig\subset (E_\calX)^{\N}$ be the set of all the semi-infinite directed paths $(e_i)_{i\geq 1}$ in $\calG_\calX$. We also let $\sig:\Sig \ra \Sig$ denote the shift map $\sig((e_i)_{i\geq 1})=(e_{i+1})_{i\geq 1}$. For each $n$ we let $P_n(\Sig^\ast)\subset \Sig^\ast$ be the subset of all the closed paths of length $n$, and set $P_{\leq n}(\Sig^\ast)=\bigcup_{j\leq n}{P_j(\Sig^\ast)}$ and $P(\Sig^\ast)=\bigcup_{j}{P_j(\Sig^\ast)}$. We will often identify the set $\Fix_n(\Sig)$ of sequences $\om\in \Sig$ satisfying $\sig^n(\om)=\om$ with $P_n(\Sig^\ast)$ via the truncation $\om=(e_1,\dots,e_n,e_1,\dots) \mapsto t_n(\om):=(e_1,\dots,e_n)$.

The next definition will be useful for the rest of the section.

\begin{definition}
    A combinatorial path $\gam$ in $\calX$ is a \emph{good representative} of $\om \in \Sig^\ast$ if there exists a path $\om_0\in \Sig^\ast$ starting at an initial state and ending at the initial vertex of $\om$ and satisfying the following. If $w_0, w_0w\in L_\calX$ are the words corresponding to $\om_0$ and $\om_0 \om$ respectively, then $\gam$ is the portion of the path $\gam_{w_0w}$ representing $w_0w$ from $\gam^-=\tau_\calX(w_0)$ to $\gam^+=\tau_\calX(w_0w)$. 
\end{definition}

Note that good representatives are geodesic. Also, by Proposition \ref{prop.languageL_X} (5) any two good representatives of the same path in $\Sig^\ast$ differ by a translation by an element in $\ov\G$. In consequence, if $\om\in P(\Sig^\ast)$ then there exists a well-defined conjugacy class $\beta(\om)\in \conj(\ov\G)$ represented by any $g\in \ov\G$ such that $\gam^+=g\gam^-$ for $\gam$ a good representative of $\om$. Clearly $\om\in P_n(\Sig^\ast)$ implies $\ell_\calX[\beta(\om)]=n$. 



We can consider lifts of paths in $\Sig^\ast$ to $\calZ$. First, we extend the equation \eqref{eq.concatenation} to define a map $\al:\Sig^\ast \ra (S_\calZ)^\ast$. If $\gam$ is a good representative of $\om \in \Sig^\ast$ defined using the path $\om_0$ as above, then $\wtilde \gam$ is the portion of $\wtilde \gam_{\al(\om_0\om)}$ starting at $\wtilde \gam_{\al(\om_0)}^+$ and ending at $\wtilde \gam_{\al(\om_0\om)}^+$, where $\wtilde \gam_{\al(\om_0)}$ and $\wtilde \gam_{\al(\om_0\om)}$ are given by Proposition \ref{prop.languageL_X} (4). In this way the path $\wtilde \gam$ represents the word $\al(\om)$. Different choices of $\om_0,\om_0'$ may give different lifts $\wtilde \gam, \wtilde \gam'$ even if $\tau_\calX(w_0)=\tau_\calX(w_0')$, but under this assumption we have $\phi(\wtilde \gam)=\phi(\wtilde \gam') =\gam$. 


A key feature of the automaton $\calA_\calX$ is that it keeps track of translation lengths for the actions of $\G$ on the cuboid complexes $\calX^\frakw$ and $\calX_\ast^{\frakw_\ast}$, via the potential on $(\Sig,\sig)$ defined below. 

\begin{definition}\label{def.psicalX}
    Let $r=r^\frakw_\calX,\psi=\psi^{\frakw_\ast}_{\calX_\ast}: E_\calX \ra \R$ be the functions such that 
    \begin{equation}\label{eq.defpsi_X}
r(e)=\ov\frakw(\pi_\calX(e)) \ \text{ and } \ \psi(e)=\psi( v_0\xrightarrow{e} v_1)=
\ov\frakw_\ast(\Psi(v_0))+\ov\frakw_\ast(\pi_\calX(e)),
\end{equation}
where $\Psi$ is the function from Proposition \ref{prop.languageL_X} (4). By abuse of notation we extend these functions to potentials $r,\psi:\Sig \ra \R$ via
\[r(e_1,e_2,\dots)=r(e_1) \ \text{ and } \ \psi(e_1,e_2,\dots)=\psi(e_1).\]
\end{definition}


Clearly $r$ and $\p$ are constant on 2 cylinders, and $r$ is positive. The next lemma can be seen as a weak analog of \cite[Lem.~3.8]{calegari-fujiwara} for pairs of word metrics on hyperbolic groups.

\begin{lemma}\label{lem.psiell_ast}
    Let $r,\p:\Sig \ra \R$ be the potentials defined above. If $n\geq 1$ and $\omega\in \Fix_n(\Sig)$ has truncation $t_n(\omega)\in P_n(\Sig^\ast)$, then $\ell_\calX[\beta(t_n(\omega))]=n$ and the $n$th Birkhoff sums at $\om$ satisfy
    \begin{equation}\label{eq.ell_X}
r^n(\omega)=r(\omega)+r(\sig(\omega))+\cdots +r(\sig^{n-1}(\omega))=\ell^{\frakw}_{\calX}[\beta(t_n(\omega))]
    \end{equation}   
    and 
     \begin{equation}\label{eq.ell_Xast}
\psi^n(\omega)=\psi(\omega)+\psi(\sig(\omega))+\cdots +\psi(\sig^{n-1}(\omega))=\ell^{\frakw_\ast}_{\calX_\ast}[\beta(t_n(\omega))].
    \end{equation}   
\end{lemma}

We will need the next lemma, which follows immediately from Remarks \ref{rmk.geodorthotope} and \ref{rmk.projquot}.

\begin{lemma}\label{lem.geodesictogeodesic}
    Let $\gam$ be a $g$-invariant geodesic in $\calZ$ for some $g\in \G$. Then the image of $\gam$ under $\phi$ (resp. $\phi_\ast$) in $\calX$ (resp.  $\calX_\ast$) is a (possibly non-parametrized) geodesic. For $\calX_\ast$ we allow the degenerate case that $\phi_\ast (\gam)$ is a point (which happens if and only if $\ell_{\calX_\ast}[g]=0$).
\end{lemma}

\begin{proof}[Proof of Lemma \ref{lem.psiell_ast}]
Let $\om\in \Fix_n(\Sig)$ be as in the statement of the lemma and let $\om_n=t_n(\om)=(v_0\xrightarrow{e_1} \cdots \xrightarrow{e_n}v_n)\in P_n(\Sig^\ast)$. As we noted previously we have $\ell_\calX[\beta(\om_n)]=n$, so the main content of the lemma are the identities \eqref{eq.ell_X} and \eqref{eq.ell_Xast}, which we now prove. 

For any $k\geq 1$, let $\om_n^{(k)}\in P(\Sig^\ast)$ be the concatenation of $k$ copies of $\om_n$ and let $(\gam^{(k)})_k=(\gam_{({\om_n}^{(k)})})_k\subset \calX$ be a sequence of good representatives of $\om_n^{(k)}$ with a common starting vertex $\gam^-=(\gam^{(k)})^-$. Let $q\in \G$ satisfy $(\gam^{(1)})^+=q\gam^-$, so that $[q]=\beta(\om_n)$ and $(\gam^{(k)})^+=q^k\gam^-$ for each $k$. Also, since $\om$ is a closed path, by Lemma \ref{lem.geodesictogeodesic} we know that $\gam^{(1)}$ is a fundamental domain for a $q$-invariant geodesic in $\calX^\frakw$, and hence 
\begin{align*}
r^n(\om) &=\ov\frakw(\pi_\calX(e_1))+\cdots +\ov\frakw(\pi_\calX(e_n))\\
& =d_\calX^{\frakw}((\gam^{(1)})^-,(\gam^{(1)})^+)=d_\calX^{\frakw}(\gam^-,q\gam^-)=\ell_{\calX}^\frakw[q]=\ell_\calX^\frakw[\beta(t_n(\om))].
\end{align*}
This proves \eqref{eq.ell_X}.

To prove \eqref{eq.ell_Xast}, let $L$ be such that $\phi_\ast:\calZ \ra \calX_\ast^{\frakw_\ast}$ is $L$-Lipschitz and consider the lifts $\wtilde \gam^{(k)}\subset \calZ$ of the geodesic $\gam^{(k)}$ with a common starting vertex $\wtilde\gam^-=(\wtilde\gam^{(k)})^-$. We project these paths to $\calX_\ast$ by defining $\gam^-_\ast:=\phi_\ast(\wtilde \gam^-)$ and $(\gam_\ast^{(k)})^+:=\phi_\ast((\wtilde \gam^{(k)})^+)$, and for all $k$ we get
\begin{equation}\label{eq.gam^(k)}
    |d^{\frakw_\ast}_{\calX_\ast}(\gam_\ast^-,(\gam_\ast^{(k)})^+)-d^{\frakw_\ast}_{\calX_\ast}(\gam^-_\ast,q^k\gam_\ast^-)|\leq d^{\frakw_\ast}_{\calX_\ast}((\gam^{(k)}_\ast)^+,q^k \gam^-_\ast)\leq Ld_\calZ((\wtilde \gam^{(k)})^+,q^k\wtilde \gam^-).
\end{equation}
The last term in the inequality above is bounded by a number independent of $k$ since $\phi((\wtilde \gam^{(k)})^+)=\phi(q^k\wtilde \gam^-)=q^k \gam^-$. Also, we note that 
$d_{\calX_\ast}^{\frakw_\ast}(\gam_\ast^-,(\gam_\ast^{(k)})^+)=\ov\frakw_\ast(\al(\om_n^{(k)}))$, which equals $k\cdot \ov\frakw_\ast(\al(\om_n))$ since $\al(\om_n^{(k)})$ is the concatenation of $k$ copies of $\al(\om_n)$. 
By Proposition \ref{prop.languageL_X} (4) it follows that
\begin{align*}
    d^{\frakw_\ast}_{\calX_\ast}(\gam_\ast^-,(\gam_\ast^{(k)})^+) & =k\cdot \ov\frakw_\ast(\al(\om_n)) \\
    & =k \cdot (\ov\frakw_\ast(\Psi(v_0)\pi_\calX(e_1))+\ov\frakw_\ast(\Psi(v_1)\pi_\calX(e_2))+\cdots +\ov\frakw_\ast(\Psi(v_{n-1})\pi_\calX(e_n))) \\
    & =k \cdot (\psi(v_0\xrightarrow{e_1} v_1)+\psi(v_1\xrightarrow{e_2} v_2)+\cdots +\psi(v_{n-1}\xrightarrow{e_n} v_n)) \\
    & = k \cdot (\psi(\om)+\psi(\sig(\om))+\cdots +\psi(\sig^{n-1}(\om)))=k\psi^n(\om). 
\end{align*}
Therefore, combining this with \eqref{eq.gam^(k)} and after dividing by $k$ and letting $k$ tend to infinity we obtain $\ell_{\calX_\ast}^{\frakw_\ast}[\beta(t_n(\om))]=\ell^{\frakw_\ast}_{\calX_\ast}[q]=\psi^n(\om)$, as desired.
\end{proof}

To apply the results from Section \ref{sec.ld} we require a mixing (or at least transitive) dynamical system. To obtain such a system we consider a maximal recurrent component $\calC$ of the graph $\calG_\calX$. As before we let ${\Sig_\calC}^\ast\subset \Sig^\ast$ and $\Sig_\calC\subset \Sig$ be the subsets corresponding to paths in $\calC$, and note that $\Sig_\calC$ is $\sig$-invariant. Similarly we define $P(\Sig^\ast_\calC)$, $P_n(\Sig^\ast_\calC)$ and $P_{\leq n}(\Sig_\calC^\ast)$, and we identify $P_n(\Sig_\calC^\ast)$ with $\Fix_n(\Sig_\calC)$. 

Let $r_{\calX}^\frakw : \Sigma_\Cc \to \R_{>0}$ be the (constant on $2$ cylinder) restriction to $\Sig_\calC$ of the potential introduced in Definition \ref{eq.defpsi_X} and consider the suspension flow:
\[
\Sigma_\Cc^{r_{\calX}^\frakw} := \{(\om,t) \in \Sigma_\Cc \times \R : 0 \le t \le r_{\calX}^\frakw(\om) \}/\sim,
 \]
where each $(\om,r_\calX^\frakw(\om))$ is identified with $(\sigma(\om), 0)$ and the flow $\sig^{r_\calX^\frakw}=(\sig^{r_\calX^\frakw}_s)_{s\in \R_{>0}}$ acts as $\sig^{r_\calX^\frakw}_s(\om,t) = (\om,t+s)$. 

Note that any closed $\sig^{r_\calX^\frakw}$-orbit $\tau$ in $\Sigma_\Cc^{r_{\calX}^\frakw}$ corresponds to a closed $\sigma$-orbit in $\Sig_\calC$. More precisely, such an orbit $\tau$ must be of the form 
$\tau=\{(\om,t) : 0\leq t \leq  (r_\calX^{\frakw})^n(\om) \}$ for some $\om \in \Sig_\calC$ such that $\sig^n(\om)=\om$. In this case the period of $\tau$ equals $l_\tau=(r_\calX^{\frakw})^n(\om)$.


We fix a smooth function $\Delta : [0,1] \to \R_{\ge 0}$ such that $\Delta(0) = \Delta(1) = 1$ and $\int_0^1 \Delta(t) \ dt =1$ and define $\Phi:\Sig_\calC^{r_\calX^\frakw}\ra \R_{\ge 0}$ according to
\begin{equation}\label{eq.defPhi}
\Phi(\om,t) = \Delta\left(\frac{t}{r_{\calX}^\frakw(\om)} \right) \frac{\p_{\calX_\ast}^{\frakw_\ast}(\om)}{r_{\calX}^\frakw(\om)} \ \text{ for each } (\om,t) \in \Sigma_\Cc^{r_{\calX}^\frakw} \text{ with }0\leq t\leq r_\calX^\frakw(\om).
\end{equation}

This function has the property that, for any closed $\sig^{r_\calX^\frakw}$-orbit $\tau$ in $\Sig_\calC^{r_\calX^\frakw}$ with period $l_\tau$ and corresponding periodic $\sig$-orbit $\om,\sig(\om),\dots,\sig^n(\om)=\om$ in $\Sigma_\Cc$ we have
\[
\int_\tau \Phi := \int_0^{l_\tau} \Phi(\sig^{r_\calX^\frakw}_t(\om,0)) \ dt =  \ell_{\calX_\ast}^{\frakw_\ast}[\beta(t_n(\om))],
\]
where we have used Lemma \ref{lem.psiell_ast}. For $\tau, \om$ and $n$ as above we adopt the notation 
\[
\beta(\tau) = \beta(t_n(\omega)),
\]
which defines a map $\beta:  P(\Sigma_\Cc^{r_{\calX}^\frakw}) \to \conj(\ov{\G})$ 
from the set $P(\Sigma_\Cc^{r_{\calX}^\frakw})$ of periodic orbits of $\Sigma_\Cc^{r_{\calX}^\frakw}$ into $\conj(\ov\G)$. By Lemma \ref{lem.psiell_ast}, the period of any $\tau \in P(\Sigma_\Cc^{r_{\calX}^\frakw})$ equals $\ell_\calX^\frakw[\beta(\tau)]$.


For such a suspension flow $(\Sig^{r_\calX^\frakw}_\calC,\sig^{r_\calX^\frakw})$, the Manhattan curve for $(\calX^\frakw,\calX^{\frakw_\ast}_\ast)$ can be described in terms of the pressures related to $\Phi$, as stated in the next proposition.


\begin{proposition}\label{prop.manattanpressure}
Let $\G,\calX,\calX_\ast,\calZ,\frakw,\frakw_\ast$ satisfy Convention \ref{conv.X_ast} and let $\calA_\calX$ and $r_\calX^\frakw,\psi_{\calX_\ast}^{\frakw_\ast}:\Sig \ra \R$ be given by Proposition \ref{prop.languageL_X} and Definition \ref{def.psicalX} respectively. If $\calC$ is a maximal recurrent component of $\calG_\calX$ and $\Phi:\Sig^{r_\calX^\frakw}_\calC \ra \R$ is given by \eqref{eq.defPhi}, then for any $s\in \R$ we have
$$
\theta_{\calX^{\frakw_\ast}_\ast/\calX^\frakw}(s) = \textnormal{P}_\Cc(-s\Phi),
$$
where $\textnormal{P}_\calC(-s\Phi)$ is the pressure of the potential $-s\Phi$ on the suspension $(\Sig_\calC^{r_\calX^\frakw},\sig^{r_\calX^\frakw})$.
\end{proposition}
In particular, this result implies that the pressure $s \mapsto \textnormal{P}_\calC(-s\Phi)$ is independent of the choice of maximal recurrent component $\calC$.

\begin{remark}\label{rmk.pressure}
    As it will be clear from the proof of Proposition \ref{prop.manattanpressure}, if $\calC$ is any maximal recurrent component of $\calG_\calX$ then for any $s\in \R$ we also have
    \[\theta_{\calX_\ast/\calX}(s)=\text{P}_\calC(-s\psi),\]
    where $\psi=\psi_{\calX_\ast}^{\frakw_\ast}$ for $\frakw_\ast \equiv 1$ the constant orthotope structure  and $\text{P}_\calC(-s\psi)$ is the pressure of the potential $-s\psi$ on $(\Sig_\calC,\sig)$.
\end{remark}


The rest of the section is devoted to proving this proposition, and for the sequel we fix a compact set $K\subset\calX$ such that $\calX=\G K$ and assume that $o\in K$. Since $\calG_\calX$ is finite and pruned we also fix $N>0$ such that any $\om\in \Sig^\ast$ has a good representative $\gam_\om$ satisfying $d_\calX(\gam^-,o)\leq N$. Otherwise explicit, for any $\om\in \Sig^\ast$ we fix a good representative $\gam_\om$ that minimizes $d_\calX(\gam^-,o)$ and define
     \[\G_\calC:=\{g\in \G \colon \text{ there exists }\om\in \Sig_\calC^\ast \text{ such that }\gam_\om^+\in gK\}.\]
We also write $B_n=\{g\in \G \colon d_\calX(go,o)\leq n\}$ for each $n\geq 0$. 

\begin{lemma}\label{lem.gam_omC}
    Let $C$ be the constant from Proposition \ref{prop.languageL_X} (1). Then
    \[ \sup_{x\in \calX^0}{\#\{\om\in \Sig^\ast \colon \gam^+_\om=x\}}\leq C(N+1).\]
\end{lemma}
\begin{proof}
    Let $x\in \calX^0$ and $\om \in \Sig^\ast$ be such that $\gam_\om^+=x$, so that $\gam_\om$ is constructed from a path $\om_0\in \Sig^\ast$ such that $\om'=\om_0\om$ also belongs to $\Sig^\ast$. Then $\om_0$ is a prefix of $\om'$ of length at most $N$ and $\om'$ represents the unique word $w\in L_\calX$ satisfying $\tau_\calX(w)=x$. From this we deduce
\begin{align*}
 \#\{\om\in \Sig^\ast \colon \gam^+_\om=x\} & \leq  (N+1)\cdot \# \{\om' \in \Sig^\ast \colon \om' \text{ starts at an initial state and }\gam^+_{\om'}=x\} \\
& \leq (N+1)\cdot \#\{\om' \in \Sig^\ast \colon \om' \text{ represents }w \text{ and starts at an initial state}\},
\end{align*}
and the lemma follows from Proposition \ref{prop.languageL_X} (1).
\end{proof}



\begin{lemma}\label{lem.upperdensity}
    The set $\G_\calC$ has positive lower density for the action on $\calX$. That is
    \[\liminf_{n\to \infty}{\frac{\#(\G_\calC \cap B_n)}{\# B_n}}>0.\]
\end{lemma}

\begin{proof}
For each $n$ we let $(\Sig_\calC^\ast)_{\leq n}$ denote the set of paths in $\calC$ of length at most $n$. First we claim that there exist $B>1$ and $0\leq \lam < e^{v_\calX}$ such that 
\begin{equation}\label{eq.claimsigma}
    \# (\Sig_\calC^\ast)_{\leq n} \geq B^{-1}e^{n v_\calX}-B\lam^{n }
\end{equation}
for all $n$ large enough. To show this, let $A$ be the adjacency matrix of $\calC$, which is irreducible since $\calC$ is recurrent. Moreover, $\calC$ being maximal implies that the spectral radius of $A$ equals $e^{v_\calX}$. Suppose that $A$ has $p\geq 1$ eigenvalues of absolute value $e^{v_\calX}$ and let $0\leq \lam <e^{v_\calX}$ be any number greater than the absolute value of all the other eigenvalues of $A$. By \cite[Thm.~3.1]{DFW}, for each $k\geq 1$ the matrix $A^{kp}$ has $e^{kpv_\calX}$ as eigenvalue of multiplicity $p$ and all its other eigenvalues have absolute value less than $\lam^{kp}$. In particular, its trace satisfies
\[\tr (A^{kp}) \geq pe^{kpv_\calX}-(\dim A -p)\lam^{kp}.\]
But $\tr (A^{kp})$ equals the number of closed paths of length $kp$ in $\calC$, so if $n=kp+r$ with $0\leq r<p$ an integer and $k\geq 1$, then
\begin{align*}
    \#(\Sig_\calC^\ast)_{\leq n}\geq \#(\Sig_\calC^\ast)_{\leq kp} 
 & \geq \tr(A^{kp})\geq  he^{kpv_\calX}-(\dim A -p)\lam^{kp}\\
  &  \geq  (pe^{-pv_\calX})e^{nv_\calX}-(\dim A -p)\lam^{n}.
\end{align*}

This concludes the proof of the claim. Now consider $n$ large enough and $\om\in (\Sig_\calC^\ast)_{\leq n}$. Since the $\G$-translates of $K$ cover $\calX$ we have $\gam_\om^+\in gK$ for some $g\in \G$, so that $d_\calX(\gam^+_\om,go)\leq D$ with $D$ being the diameter of $K$. In addition, by definition we have $d_\calX(\gam_\om^-,o)\leq N$ and hence $d_\calX(o,go)\leq n+D+N$. By Lemma \ref{lem.gam_omC} this implies that 
\begin{equation*}
\#(\Sig_\calC^\ast)_{\leq n} \leq \#(\G_\calC \cap B_{n+D+N})\cdot \sup_{g\in \G}\# \{\om \in \Sig_\calC^\ast \colon \gam^+_\om \in gK\} \leq (N+1)C\cdot \#K \cdot \#(\G_\calC \cap B_{n+D+N}),
\end{equation*}
where $C$ is the constant from Proposition \ref{prop.languageL_X} (1).

Combining this with \eqref{eq.claimsigma} we get
\begin{align*}
    \#(\G_\calC \cap B_n) &\geq ((N+1)C\#K)^{-1}B^{-1}e^{(n-D-N)v_\calX}-((N+1)C\#K)^{-1}B\lam^{n-D-N} \\
    &= [ ((N+1)(C\#K)B)^{-1}e^{-(D+N)v_\calX}]e^{nv_\calX}-[((N+1)C\#K)^{-1}B\lam^{-D-N}]\lam^{n}.
\end{align*}
Finally, the since the action on $\calX$ has a contracting element, \cite[Thm.~1.8~(2)]{yang} implies that there exists $C'>0$ such that $\#B_n \leq C'e^{nv_{\calX} }$ for all $n$ large enough, and the conclusion follows.
\end{proof}

The next two results are used to give a uniform comparison of the number of conjugacy classes in $\G$ with a bound on their translation lengths (with respect to $\calX^\frakw$) and the number of periodic orbits in the suspension $(\Sig^{r_\calX^\frakw}_\calC,\sig^{r_\calX^\frakw})$ with bounded period. The assumption of having a contracting element is essential in the proof of the Lemma \ref{lem.goodpath}. For $T\geq R \geq 0$, recall that $P(\Sigma_\Cc^{r_{\calX}^\frakw}, R, T)$ denotes the set of periodic orbits in $(\Sig_\calC^{r_\calX^\frakw},\sig^{r_\calX^\frakw})$ with period on the interval $[T-R,T+R]$.

\begin{lemma}\label{lem.polyn}
    For any $R >0$ there exists a polynomial $Q$ (depending on $R$) that is non-decreasing on $\R_{>0}$ and such that for any $[g]\in \conj(\ov\G)$ we have 
    \begin{equation}\label{eq.polyQ}
        \# \{\tau\in P(\Sigma_\Cc^{r_{\calX}^\frakw}, R, T) : \beta(\tau) = [g] \}\leq Q(T).
    \end{equation}
\end{lemma}
\begin{proof}
    To solve the lemma it is enough to show that there exists a polynomial $\widetilde{Q}$ that is increasing on $\R_{>0}$ and such that for any $[g]\in \conj(\ov\G)$ we have 
    \begin{equation}\label{eq.polyninpaths}
    \# \{\omega\in P(\Sig^\ast) \colon \beta(\omega)=[g]\}\leq \widetilde{Q}(\ell_\calX[g]).
    \end{equation}
    Indeed, since $\frakw$ is non-vanishing, the identity map $\textnormal{Id}:\calX \ra \calX^\frakw$ is a quasi-isometry so there exists $L>0$ such that $\ell_\calX[g]\leq L\ell_\calX^{\frakw}[g]$ for any $[g]\in \conj(\G)$. Also, if $\tau\in P(\Sigma_\Cc^{r_{\calX}^\frakw}, R, T)$ satisfies $\beta(\tau)=[g]$, then $l_\tau=\ell_\calX^\frakw[g]\in [R-T,R+T]$. Since any periodic orbit in $P(\Sigma_\Cc^{r_{\calX}^\frakw})$ corresponds to an (orbit determined by an) element in $P(\Sig_\calC^\ast)\subset P(\Sig^\ast)$, for any $[g]$ such that the left-hand side in \eqref{eq.polyQ} is non-zero we have 
    \begin{align*}
        \# \{\tau\in P(\Sigma_\Cc^{r_{\calX}^\frakw}, R, T) : \beta(\tau) = [g] \} & \leq \# \{\om\in P(\Sigma) : \beta(\om) = [g] \} \\
        & \leq \widetilde{Q}(\ell_\calX[g]) \leq \widetilde{Q}(LT+LR)=:Q(T).
    \end{align*}
    
    To prove \eqref{eq.polyninpaths} we claim that for any group $\ov\G$ acting properly, cocompactly and co-specially on a $\CAT(0)$ cube complex $\calX$ and for any $x\in \calX^{0}$ and $R\geq 0$ there exists a polynomial $\hat Q$ that is increasing on $\R_{>0}$ and such that
    \begin{equation}\label{eq.polyn}
        \#\{g\in [g] \colon d_\calX(gx,x)\leq \ell_\calX[g]+R\}\leq \hat Q(\ell_\calX[g]) 
    \end{equation}
    for any $[g]\in \conj(\ov{\G})$.

To see how this claim proves the lemma, fix $[g]\in \conj(\ov\G)$ and let $\om\in P(\Sig^\ast)$ be such that $\beta(\om)=[g]$. If $\gam=\gam_{\om}$ is a good representative such that $d_\calX(\gam^-,o)\leq N$, we let $q\in [g]$ be such that $\gam^+=q\gam^-$. Then $d_\calX(qo,\gam^+)=d_\calX(o,\gam^-)\leq N$ and we have 
\[|d_\calX(o,qo)-\ell_\calX[g]|=|d_\calX(o,qo)=d_\calX(\gam^-,\gam^+)|\leq 2N.\]
In addition, there exists a constant $\hat C>0$ such that for any $[g]\in \conj(\ov\G)$ and any $q\in [g]$ satisfying $|d_\calX(o,qo)-\ell_\calX[g]|\leq 2N$, the set
\[\{\om \in P(\Sig^\ast) \colon \beta(\om)=[g] \text{ and } d_\calX(\gam_\om^+,qo)\leq N\}\]
has cardinality at most $\hat C$. 
Indeed, if $B$ is the set of vertices at distance at most $N$ from $o$, then this cardinality is bounded above by
\begin{align*}
    \sum_{x\in B}{\# \{ \om \in P(\Sig^\ast) \colon \gam^+_\om=qx\} } & \leq \# B \cdot \sup_{x\in \calX^0}{\# \{ \om \in P(\Sig^\ast) \colon \gam^+_\om = x\}} \leq \# B \cdot C(N+1),
\end{align*}
where for the last inequality we used Lemma \ref{lem.gam_omC}.

Applying this to our case of interest, we deduce
\begin{align*}
    \# \{\omega\in P(\Sig^\ast) \colon \beta(\omega)=[g]\}\leq \hat C \cdot \# \{q\in [g] \colon |d_\calX(qo,o)-\ell_\calX[g]|\leq 2N\}\leq \hat C \cdot \hat Q(\ell_\calX[g]),
\end{align*}
where $\hat Q$ is the polynomial given by the claim for $x=o$ and $R=2N$.

To prove the claim \eqref{eq.polyn}, by equivariantly embedding $\calX$ as a convex subcomplex of the universal cover of a Salvetti complex we can assume that $\ov\G$ is a right-angled Artin group with standard (symmetric) generating set $S$ and $\calX$ is the universal cover of its Salvetti complex. Then $\calX^{1}$ is the Cayley graph for $\ov\G$ with respect to $S$, and since the expected conclusion is independent of the base point we can assume that $x=o$ is the identity element of $\G$, so that $d_\calX(gx,x)=|g|_S$ is the word length of $g$ for any $g\in \ov \G$. 
    
Now we fix $[g]\in \conj(\ov\G)$, set $\ell=\ell_\calX[g]=\ell_S[g]$ and consider the sets $E_n[g]=\{g\in [g]\colon |g|_S\leq \ell+n\}.$
Note that $\# E_0[g]\leq \ell$ since any two conjugate elements that minimize the word length are actually cyclically conjugated with respect to some minimal word representations. 
Also, if $g\in E_n[g]$, and $n>0$, then indeed $n\geq 2$ and $g$ is represented by a word of the form $x_1a^{\pm}x_2a^{\mp}x_3$, where $a\in S$ is a standard generator and all the letters in the words $x_1$ and $x_3$ commute with $a$. Then the element $g'$ represented by the word $x_1x_2x_3$ belongs to $E_{n-2}[g]$, and there are at most $\#{S}(\ell+n)(\ell+n-1)/2$ ways to reconstruct $g$ from $g'$. Therefore, we have $$\#E_n[g]\leq \#{S}(\ell+n)(\ell+n-1)/2\cdot \#E_{n-2}[g]$$
for each $n$, and hence 
$ \#\{g\in [g] \colon |g|_S\leq \ell+R\}\leq \#E_{2R}[g]\leq \hat Q(\ell)$
for $$\hat Q(t)=(\#S)^R(t+2R)(t+2R-1)\cdots (t+1)t/2^R.$$
This concludes the proof of the claim and the lemma.    
\end{proof}

\begin{lemma}\label{lem.goodpath}
    There exists $C'>0$ such that for any non-torsion conjugacy class $[g]\in \conj(\G)$ we can find a representative $\hat{g}\in [g]$ and a closed path $\omega_{[g]}\in P(\Sig_\calC^\ast)$ satisfying   
\[\min\{d_\calX(\hat g o,\gam^+_{\om_{[g]}}),d_\calX(\hat g^{-1}o,\gam^+_{\om_{[g]}} )\} \leq C',\] and additionally
    \[\max\{|\ell^\frakw_{\calX}[g]-\ell^\frakw_\calX[\beta(\om_{[g]})]|,  |\ell^{\frakw_\ast}_{\calX_\ast}[g]-\ell^{\frakw_\ast}_{\calX_\ast}[\beta(\om_{[g]})]|\}\leq C'.\]
\end{lemma}

\begin{proof}
For the proof we consider the following constants. Let $M_1$ be such that any two vertices in $\calC$ can be joined by a path in $\calC$ of length at most $M_1$ (in both directions). This number exists since $\calC$ is recurrent. Also, the projection $\phi:\calZ \ra \calX$ is a quasi-isometry since it is $\G$-equivariant and the action of $\G$ on both $\calZ$ and $\calX$ is proper and cocompact, so let $M_2>0$ be such that 
    \[d_\calZ(x,y)\leq M_2d_\calX(\phi(x),\phi(y))+M_2\]
for all $x,y\in \calZ$. In addition, let $M_3$ be the diameter of $\phi^{-1}(K)\subset \calZ$ and fix a constant $M_4$ larger than $N$ and the diameter of $K$. Finally, let $L$ be the maximum of all the weights $\frakw(\frakh)$ or $\frakw_\ast(\frakh_\ast)$ among hyperplanes $\frakh \in \calX$ and $\frakh_\ast \in\calX_\ast$, so that the four functions 
\[\phi: \calZ \ra \calX^\frakw, \ \ \phi_\ast: \calZ \ra \calX_\ast^{\frakw_\ast}, \ \ \textnormal{Id}:\calX \ra \calX^\frakw, \ \ \textnormal{Id}:\calX_\ast \ra \calX_\ast^{\frakw_\ast}\]
are $L$-Lipschitz.

 We now start the proof, so we let $g\in \G$ represent the non-torsion conjugacy class $[g]\in\conj(\G)$. Then $g$ fixes a bi-infinite combinatorial axis $\wtilde\lam$ in the cubical barycentric subdivision $\dot{\calZ}$  \cite[Thm.~1.4]{haglund}. After conjugating by an element of $\G$ we can assume that $d_\calZ(\wtilde o,\wtilde\lam)\leq M_3$, so in particular we have $|d_\calZ(\wtilde o,g \wtilde o)-\ell_\calZ[g]|\leq 2M_3.$

By Lemma \ref{lem.geodesictogeodesic} and the fact that $\phi,\phi_\ast$ are Lipschitz, the images $\lam=\phi(\wtilde \lam)$ and $\lam_\ast=\phi_\ast(\wtilde \lam)$ are also $g$-invariant (unparametrized) geodesics satisfying $d^\frakw_\calX( o,\lam)\leq LM_3$ and $d^{\frakw_\ast}_{\calX_\ast}(o_\ast,\lam_\ast)\leq LM_3$, which gives us
\begin{equation}\label{ineq.first}
|d^\frakw_\calX(o,g o)-\ell^\frakw_\calX[g]|\leq 2LM_3 \ \  \text{ and } \ \  |d^{\frakw_\ast}_{\calX_\ast}(o_\ast,g o_\ast)-\ell^{\frakw_\ast}_{\calX_\ast}[g]|\leq 2LM_3. 
\end{equation}

The action of $\G$ on $\calX$ is proper, cocompact and has a contracting element, and hence by \cite[Thm.~C]{yang} there exists a constant $\ep>0$ satisfying the following for any $h\in \G$. Let $\calV_h$ denote the set of all the group elements $k\in \G$ such that if $\gam \subset \calX$ is a combinatorial geodesic path with endpoints $\gam^{\pm}$ verifying $d_\calX(\gam^-,o)\leq M_4$ and $d_\calX(\gam^+,ko)\leq M_4$, then there exists no $s\in \G$ such that $d_\calX(so,\gam)\leq \ep$ and $d_\calX(sho,\gam)\leq \ep$. Then 
\[\lim_{n\to \infty}{\frac{\#(\calV_h\cap B_n)}{\#B_n}}=0\]
(the freedom in our choice for $M_4$ comes from the Remark after Theorem~C in \cite{yang}). In virtue of Lemma \ref{lem.upperdensity} we conclude that the set $\G_\calC \bs \calV_h$ is non-empty for every $h\in \G$. Applying this to $h=g$ we deduce the existence of a path $\om'\in \Sig_\calC^\ast$ and $s\in \G$ such that 
$d_\calX(so,\gam_{\om'})\leq \ep$ and $d_\calX(sgo,\gam_{\om'})\leq \ep$. 

Let $u,v\in \gam_{\om'}$ be such that $d_\calX(so,u)\leq \ep$ and $d_\calX(sgo,v)\leq \ep$, and without loss of generality assume that $u$ belongs to the portion of $\gam_{\om'}$ from $\gam^-_{\om'}$ to $v$. Let $\om=\om_{[g]}\in \Sig_\calC^\ast$ be a closed path composed by the concatenation of the subpath $\ov\om'$ of $\om'$ that determines the portion of $\gam_{\om'}$ from $u$ to $v$ and a path in $\Sig_\calC^\ast$ of length at most $M_1$ from the final vertex of $\ov\om'$ to its initial vertex. 
Let $\gam=\gam_{\om}\subset \calX$ be the good representative of $\om$ with 
\begin{equation}\label{ineq.N2}
  L^{-1}d^\frakw_\calX(\gam^-,o)\leq d_\calX(\gam^-,o) \leq  N,  
\end{equation} and let $s'\in \ov\G$ be such that $\gam^-=s'u$. This implies \begin{equation}\label{ineq.dx2}
L^{-1}d^\frakw_\calX(s'so,\gam^-)\leq d_\calX(s'so,\gam^-)\leq \ep \ \text{ and }\  L^{-1}d^\frakw_\calX(s'sgo,\gam^+)\leq d_\calX(s'sgo,\gam^+)\leq \ep+M_1.\end{equation}
Since $\om$ is a loop we have $\gam^+=q\gam^-$ for $[q]=\beta(\om)\in \conj(\ov\G)$, and in particular from \eqref{ineq.first} we get 
\[|\ell^\frakw_\calX[g]-\ell^\frakw_\calX[\beta(\om)]|\leq L(2\ep+M_1+2M_3).\]
Also, for $k\geq 1$ let $\om^{(k)}\in \Sig_\calC^\ast$ be the concatenation of $k$ copies of $\om$, and let $\gam^{(k)}\subset \calX$ be a good representative of $\om^{(k)}$ so that $(\gam^{(k)})^{-}=\gam^{-}$ and $(\gam^{(k)})^{+}=q^k\gam^{-}$. Note that $\gam^{(k)}$ is always a subpath of $\gam^{(k+1)}$.

Now we lift each $\gam^{(k)}$ to $\calZ$ to get a sequence $\wtilde \gam^{(k)}$ of geodesic paths in $\calZ$. Then $\phi(\wtilde \gam^{(k)})=\gam^{(k)}$ (up to parametrization), so that $(\gam^{(k)})^\pm=\phi((\wtilde \gam ^{(k)})^\pm)$ for all $k$. We denote $\wtilde \gam^\pm=(\wtilde \gam^{(1)})^\pm$ and we assume $(\wtilde \gam^{(k)})^-= \wtilde \gam^{-}$ for all $k$. 
By $\G$-equivariance of $\phi$ we have $\phi((\wtilde \gam^{(k)})^+)=(\gam^{(k)})^+=q^k\gam^-=\phi(q^k\wtilde\gam^{-} )$, and hence  
\begin{equation}\label{ineq.sss}
    d_\calZ((\wtilde \gam^{(k)})^+,q^k\wtilde\gam^{-})\leq M_3
\end{equation} for all $k$. Also, the inequalities \eqref{ineq.dx2} imply
\begin{equation}\label{ineq.dZ}
    d_\calZ(s's\wtilde o,\wtilde\gam^-)\leq M_2\ep+M_2 \ \ \text{ and }\ \ d_\calZ(s'sg\wtilde o,\wtilde\gam^+)\leq M_2(\ep+M_1)+M_2.
\end{equation}
We project the geodesics $\wtilde \gam^{(k)}$ to $\calX_\ast^{\frakw_\ast}$ via $\phi_\ast$, so we consider $\gam_\ast^{(k)}:=\phi_\ast(\wtilde \gam^{(k)})$, which are (unparametrized) geodesics by Lemma \ref{lem.geodesictogeodesic}, and as before we denote $ \gam_\ast^\pm=\phi_\ast(\gam_\ast^\pm)$.

The length of $\gam^{(k)}_\ast$ in $\calX_\ast^{\frakw_\ast}$ equals $\ov\frakw_\ast(\al(\om^{(k)}))$, and since the word $\al(\om^{(k)})$ is the concatenation of $k$ copies of $\al(\om)$
we have
\begin{equation}\label{ineq.equality}
    d^{\frakw_\ast}_{\calX_\ast}(\gam_\ast^{-},(\gam_\ast^{(k)})^{+})=kd^{\frakw_\ast}_{\calX_\ast}(\gam_\ast^{-},\gam_\ast^{+})
\end{equation}
for all $k$. In addition, by \eqref{ineq.sss} and \eqref{ineq.dZ} we obtain
\begin{equation}\label{ineq.sss*}
        d^{\frakw_\ast}_{\calX_\ast}((\gam_\ast^{(k)})^+,q^k\gam^{-}_\ast)\leq L M_3,
\end{equation}
and
\begin{equation}\label{ineq.dX*}
    d^{\frakw_\ast}_{\calX_\ast}(s's o_\ast,\gam_\ast^-)\leq LM_2(\ep+1) \  \ \text{ and } \ \ d^{\frakw_\ast}_{\calX_\ast}(s'sg o_\ast,\gam^+_\ast)\leq LM_2(\ep+M_1+1).
\end{equation}
From these inequalities and \eqref{ineq.first} we get 
\begin{align*}
    \ell^{\frakw_\ast}_{\calX_\ast}[q]  \leq d^{\frakw_\ast}_{\calX_\ast}(\gam_\ast^-,q\gam_\ast^-) & \leq d^{\frakw_\ast}_{\calX_\ast}(\gam_\ast^-,\gam_\ast^+)+LM_3 \\
    & \leq d^{\frakw_\ast}_{\calX_\ast}(o_\ast, g o_\ast)+L(M_3+M_2(2\ep+M_1+2))\\
    & \leq \ell^{\frakw_\ast}_{\calX_\ast}[g]+L(3M_3+M_2(2\ep+M_1+2)).
\end{align*}
On the other hand, \eqref{ineq.equality} and \eqref{ineq.sss*} imply
\begin{equation*}
    kd^{\frakw_\ast}_{\calX_\ast}(\gam_\ast^-,\gam_\ast^+)=d^{\frakw_\ast}_{\calX_\ast}(\gam_\ast^-,(\gam_\ast^{(k)})^+) \leq d^{\frakw_\ast}_{\calX_\ast}(\gam_\ast^-,q^k\gam_\ast^-)+LM_3,
\end{equation*}
and after dividing by $k$ and letting $k$ tend to infinity we get
\begin{equation*}
    d^{\frakw_\ast}_{\calX_\ast}(\gam_\ast^-,\gam_\ast^+)\leq \ell^{\frakw_\ast}_{\calX_\ast}[q].
\end{equation*}
Combining this inequality with \eqref{ineq.first} and \eqref{ineq.dX*} gives us
\begin{align*}
    \ell^{\frakw_\ast}_{\calX_\ast}[g] & \leq d^{\frakw_\ast}_{\calX_\ast}(s's o_\ast, s's g o_\ast)+2LM_3 \\
    & \leq d^{\frakw_\ast}_{\calX_\ast}(\gam_\ast^-,\gam_\ast^+)+L(2M_3+M_2(2\ep+M_1+2)) \\
    & \leq \ell^{\frakw_\ast}_{\calX_\ast}[q]+L(2M_3+M_2(2\ep+M_1+2)),
\end{align*}
and we deduce
\begin{equation*}
    |\ell^{\frakw_\ast}_{\calX_\ast}[g]-\ell^{\frakw_\ast}_{\calX_\ast}[\beta(\om)]|=|\ell^{\frakw_\ast}_{\calX_\ast}[g]-\ell^{\frakw_\ast}_{\calX_\ast}[q]|\leq L(3M_3+M_2(2\ep+M_1+2)).
\end{equation*}
Finally, if we define $\hat{g}=s'sg(s's)^{-1}\in [g]$, then by \eqref{ineq.N2} and 
\eqref{ineq.dx2} we get $$d_\calX(\hat g o, \gam^+)\leq d_\calX(\hat g o, gs'so)+\ep+M_1 \leq d_\calX((s's)^{-1}o,o)+\ep+M_1\leq 2\ep+M_1+N.$$
In conclusion, the lemma follows with $C'=(L+1)(N+3M_3+M_2(2\ep+M_1+2))$.
\end{proof}




\begin{proof}[Proof of Proposition \ref{prop.manattanpressure}]
    For each $s\in \R$ and $R >0$ we consider the sums
    \[
\calP(R, T,s) = \sum_{|\ell^\frakw_\calX[g] - T| \le R} e^{-s\ell^{\frakw_\ast}_{\calX_\ast}[g]} \ \text{ and } \ \calP_\Cc(R, T,s) = \sum_{\tau \in P(\Sigma_\Cc^{r_\calX^\frakw},R,T)} e^{-s \int_\tau{\Phi}}.
    \]
    Since $l_\tau=\ell_\calX^\frakw[\beta(\tau)]$ for any closed orbit $\tau$, by Lemmas \ref{lem.psiell_ast} and \ref{lem.polyn} there exists a polynomial $Q$ depending only on $R$ such that $\calP_\Cc(R, T,s) \le  Q(T)\calP(R, T,s)$ for each $s\in \R$ and $T>0$.  
    
   For an inequality in the other direction, for any $[g]$ we use Lemma \ref{lem.goodpath} to find a path $\o_{[g]} \in P(\Sigma_\Cc^\ast)$ and a representative $\hat g$ of $[g]$ satisfying 
   \begin{equation}\label{eq.om_gused}
       \min \{d_{\calX}(\hat g o, \gam^+_{\o_{[g]}}),d_{\calX}(\hat g^{-1} o, \gam^+_{\o_{[g]}})\} \leq C'
   \end{equation}
   and 
     \begin{equation}\label{eq.om_used2}
           \max\{|\ell^\frakw_{\calX}[g]-\ell^\frakw_\calX[\beta(\om_{[g]})]|,  |\ell^{\frakw_\ast}_{\calX_\ast}[g]-\ell^{\frakw_\ast}_{\calX_\ast}[\beta(\om_{[g]})]|\}\leq C'.   
     \end{equation} 
    From \eqref{eq.om_gused} we get that the association $[g] \mapsto \o_{[g]}$ is uniformly finite-to-$1$. We extend this association to $[g]\mapsto \om_{[g]} \mapsto \tau_{[g]}$, where $\tau_{[g]}\in P(\Sig_\calC^{r_\calX^\frakw})$ is the periodic orbit corresponding to the path $\om_{[g]}$. Since changing the initial vertex of a closed path in $P(\Sig_\calC^\ast)$ does not change the periodic orbit in $P(\Sig_\calC^{r_\calX^\frakw})$, the association $\om_{[g]} \mapsto \tau_{[g]}$ is at most (linear in $\ell_\calX[g]$)-to-1. But $\ell_\calX[g]$ is comparable to $\ell^{\frakw}_\calX[g]=\ell^\frakw_\calX[\beta(\tau_{[g]})]=l_{\tau_{[g]}}$ (recall that $\calX$ and $\calX^\frakw$ are quasi-isometric), and so from \eqref{eq.om_used2} we deduce that for each $s \in \R$ there is $C_s >0$ such that   $\calP(R,T,s) \le C'_s F(T)\calP_\Cc(R+C',T,s)$ for each $T > 0$, where $F$ is a degree 1 polynomial depending only on $R$. 
    
    It follows that for any fixed $R$ sufficiently large and for any  $s \in \R$
    \begin{equation}\label{eq.press}
        \lim_{T\to\infty} \frac{1}{T} \log \calP(R,T,s) =\lim_{T \to \infty} \frac{1}{T} \log \left( \sum_{\tau \in P(\Sigma^{r_\calX^\frakw}_\Cc,R,T)}e^{-s \int_\tau{\Phi}}\right)=\textnormal{P}_\Cc(-s\Phi),
    \end{equation}
where $\text{P}_\Cc(-s\Phi)$ is the pressure of the potential $-s\Phi$ on the suspension $(\Sig^{r_\calX^\frakw}_\calC,\sig^{r_\calX^\frakw})$. 

Also, note that
\[
\sum_{[g] \in \conj(\G)} e^{- t\ell^\frakw_\calX[g] - s\ell_{\calX^{\frakw_\ast}_\ast}[g]} \le e^{R|t|} \sum_{T = 1}^\infty \calP(R,T,s) e^{-tT}
\]
assuming the right-hand side of the above converges.
Similarly we have 
\[
 \sum_{T = 1}^\infty \calP(R,T,s) e^{-tT} \le 
 2R e^{R|t|}\sum_{[g] \in \conj(\G)} e^{- t\ell^\frakw_\calX[g] - s\ell^{\frakw_\ast}_{\calX_\ast}[g]} 
\]
when the right-hand side converges. We deduce that for each $s \in \R$ the series
\[
 \sum_{T = 1}^\infty \calP(R,T,s) e^{-tT} \ \text{ and } \ \sum_{[g] \in \conj(\G)}e^{- t\ell^\frakw_\calX[g] - s\ell^{\frakw_\ast}_{\calX_\ast}[g]} 
\]
have the same abscissa of convergence as $t$ varies.
Hence by $\eqref{eq.press}$ we deduce
\[
\theta_{\calX^{\frakw_\ast}_\ast/\calX^{\frakw}}(s) = \text{P}_\Cc(-s\Phi),
\]
as desired.
\end{proof}


\subsection{Analyticity and Large deviation for pairs of cubulations}\label{subsec.analytic+LDcub}

In this section we prove Theorems \ref{thm.manhattanfactsX} and \ref{thm.cubulation}. For a triple $(\G,\calX,\calX_\ast)\in \frakX$ we always assume that it satisfies Convention \ref{conv.X_ast}, which is possible by Lemma \ref{lem.preparations}. In particular, all the results and notations from this and the previous section are valid for this triple. We first prove a large deviation principle that follows from Proposition \ref{prop.manattanpressure}.

\begin{corollary}\label{coro.ldmanhattan}
    Let $(\G,\calX,\calX_\ast)\in \frakX$ and $\frakw,\frakw_\ast$ be  $\G$-invariant orthotope structures on $\calX,\calX_\ast$, and let $\calL : [\Dil(\calX^\frakw, \calX^{\frakw_\ast}_\ast)^{-1}, \Dil(\calX^{\frakw_\ast}_\ast, \calX^{\frakw}) ] \to \R$ be the Legendre transform of $\theta_{\calX^{\frakw_\ast}_\ast/\calX^\frakw}$. Then for any non-empty open set $U \subset \R$ and closed set $V \subset \R$ with $U \subset V$ we have that
    \begin{align*}
        -\inf_{s \in U} \calL(s) &\le \liminf_{T\to\infty} \frac{1}{T} \log\left( \frac{1}{\#\mathfrak{C}_{\calX^\frakw}(T)} \#\left\{[g] \in \mathfrak{C}_{\calX^\frakw}(T): \frac{\ell^{\frakw_\ast}_{\calX_\ast}[g]}{ \ell^\frakw_\calX[g]}\in U \right\}\right)\\
        &\le \limsup_{T\to\infty} \frac{1}{T} \log\left( \frac{1}{\#\mathfrak{C}_{\calX^\frakw}(T)} \#\left\{[g] \in \mathfrak{C}_{\calX^\frakw}(T): \frac{\ell^{\frakw_\ast}_{\calX_\ast}[g]}{ \ell^\frakw_\calX[g]}\in V \right\}\right) \le -\inf_{s \in V} \calL(s).
    \end{align*}
In consequence, the limit
\[
\tau(\calX^{\frakw_\ast}_\ast/\calX^\frakw) :=
\lim_{T\to\infty} \frac{1}{\#\mathfrak{C}_{\calX^\frakw}(T)} \sum_{[g] \in \frakC_{\calX^\frakw}(T)} \frac{\ell^{\frakw_\ast}_{\calX_\ast}[g]}{\ell^\frakw_\calX[g]}
\]
exists and equals $-\thet'_{\calX^{\frakw_\ast}_\ast/\calX^\frakw}(0)$.
\end{corollary}

\begin{proof}
Recall that for $T>R >0$ and $s \in \R$ we defined
    \[
\calP(R, T,s) = \sum_{|\ell^\frakw_\calX[g] - T| \le R} e^{-s\ell^{\frakw_\ast}_{\calX_\ast}[g]} 
\]
in the proof of Proposition \ref{prop.manattanpressure} above. We saw during that proof that
\[
     \lim_{T\to\infty} \frac{1}{T} \log \calP(R,T,s) = \theta_{\calX^{\frakw_\ast}_\ast/\calX^\frakw}(s).
\]
It follows from the G\"artner-Ellis Theorem \cite[Thm.~2.3.6]{G.E} that the large deviation principle stated in this corollary holds but with $\mathfrak{C}_{\calX^\frakw}(T)$ replaced by
\[
\mathfrak{C}_{\calX^\frakw}(T, R) = \{ [g]\in \conj \colon |\ell^\frakw_\calX[g] - T| < R\}
\]
for any fixed $R >0$ sufficiently large. It is then easy to check that this large deviation principle implies the one stated in the corollary.

By this large deviation principle we know that for any $\epsilon >0$ the cardinality of the set
\[
E_\epsilon(T): = \left\{[g] \in \mathfrak{C}_{\calX^\frakw}(T): \left|\frac{\ell^{\frakw_\ast}_{\calX_\ast}[g]}{ \ell^\frakw_\calX[g]} + \theta'_{\calX^{\frakw_\ast}_\ast/\calX^\frakw}(0) \right| > \epsilon \right\}
\]
grows strictly exponentially slower than $\#\mathfrak{C}_{\calX^\frakw}(T)$ as $T\to\infty$, i.e. the quotient $\#E_\epsilon(T)/\#\mathfrak{C}_{\calX^\frakw}(T)$ decays to $0$ exponentially as $T\to\infty$. 
It is then standard to deduce that
\[
\tau(\calX_\ast^{\frakw_\ast}/\calX^\frakw) :=
\lim_{T\to\infty} \frac{1}{\#\mathfrak{C}_{\calX^\frakw}(T)} \sum_{[g] \in \frakC_{\calX^\frakw}(T)} \frac{\ell^{\frakw_\ast}_{\calX_\ast}[g]}{\ell^\frakw_\calX[g]}
\]
exists and is equal to $ -\theta_{\calX^{\frakw_\ast}_\ast/\calX^\frakw}'(0)$ as required.
\end{proof}

\begin{proof}[Proof of Theorem \ref{thm.manhattanfactsX}]
We showed in Proposition \ref{prop.manattanpressure} that $\theta_{\calX^{\frakw_\ast}_\ast/\calX^\frakw}(s)$ is equal to the pressure $\textnormal{P}_\Cc(-s\Phi)$ for any $s$. It follows that $\theta_{\calX^{\frakw_\ast}_\ast/\calX^\frakw}$ is analytic, convex and decreasing.
Also, by Corollary \ref{coro.ldmanhattan} we know that the limit labeled $\tau(\calX^{\frakw_\ast}_\ast/\calX^\frakw)$ in the theorem exists. 
Further by comparing the exponential growth rates of both sides of the inequality
\small
\[
\#\left\{[g] \in \mathfrak{C}_{\calX^\frakw}(T): \left|\frac{\ell^{\frakw_\ast}_{\calX_\ast}[g]}{ \ell^\frakw_\calX[g]} - \tau(\calX^{\frakw_\ast}_\ast/\calX^\frakw) \right| \le \epsilon 
\right\} \le \#\left\{ [g] \in \mathfrak{C}_{\calX^\frakw}(T) : \ell^{\frakw_\ast}_{\calX_\ast}[g] \le (\tau(\calX^{\frakw_\ast}_\ast/\calX^\frakw) + \epsilon) T\right\}
\]
\normalsize
we see that
\[
\tau(\calX^{\frakw_\ast}_\ast/\calX^{\frakw}) \ge \frac{v_{\calX^\frakw}}{v_{\calX^{\frakw_\ast}_\ast}}.
\]
Therefore to conclude the proof we need to check the equivalence of the statements $(1)$, $(2)$ and $(3)$ when the action of $\G$ on $\calX_\ast$ (and hence on $\calX_\ast^{\frakw_\ast}$) is proper. When this is the case we have $0<v_{\calX^{\frakw_\ast}_\ast} < \infty$ and the Manhattan curve $ \theta_{\calX^{\frakw_\ast}_\ast/\calX^\frakw}(s)$ is $0$ at $s=v_{\calX^{\frakw_\ast}_\ast}.$

We will prove the implications $(1) \Ra (3) \Ra (2) \Ra (1)$. 
Note that the implication $(1) \Ra (3)$ follows easily from the facts that
$\tau(\calX^{\frakw_\ast}_\ast/\calX^\frakw) = - \theta'_{\calX^{\frakw_\ast}_\ast/\calX^\frakw}(0)$, $\theta_{\calX^{\frakw_\ast}_\ast/\calX^\frakw}(0) = v_{\calX^\frakw}, \theta_{\calX^{\frakw_\ast}_\ast/\calX^\frakw}(v_{\calX^{\frakw_\ast}_\ast}) = 0$ and $\theta_{\calX^{\frakw_\ast}_\ast/\calX^\frakw}$ is convex so has non-increasing derivative. Also, the implication $(2) \Ra (1)$ follows from the definition of the Manhattan curve. Hence we just need to prove the implication $(3) \Ra (2).$

To do so we note that
\[
\tau(\calX_\ast^{\frakw_\ast}/\calX^\frakw) = \int_{\Sigma_\Cc^{r_\calX^\frakw}} \Phi \ dm 
\]
where $m$ is the measure of maximal entropy for $(\Sigma_\Cc^{r_\calX^\frakw},\sig^{r_{\calX}^{\frakw}})$. However we saw in Section \ref{sec.sf} that
\[
\int_{\Sigma_\Cc^{r_\calX^\frakw}} \Phi \ dm  = \frac{\int_{\Sigma_\Cc} \psi_{\calX_\ast}^{\frakw_\ast}\ d\mu_1}{\int_{\Sigma_\Cc} r_\calX^\frakw \ d\mu_1}
\]
where $\mu_1$ is the equilibrium state of $- \delta_{r_\calX^\frakw} r_\calX^\frakw$ on $\Sigma_\Cc$. To simplify notation going forward we will also write $r =r_\calX^\frakw$, $\psi =\psi_{\calX_\ast}^{\frakw_\ast} $ and $\mu_2$ for the equilibrium state of $- \delta_{\psi_{\calX_\ast}^{\frakw_\ast}} \psi_{\calX_\ast}^{\frakw_\ast}$ on $\Sigma_\Cc$.
We now note that by Proposition \ref{prop.manattanpressure} we have that
\[
\text{P}_\Cc(-v_{\calX^\frakw} r_\calX^\frakw) = \text{P}_\Cc(-v_{{\calX^{\frakw_\ast}_\ast}} \psi_{\calX_\ast}^{\frakw_\ast}) = 0.
\]
Here the pressures are the pressures of the potentials over the subshift (not suspension).
Hence the inequality $\tau(\calX^{\frakw_\ast}_\ast/\calX^{\frakw})  \ge v_{\calX^\frakw}/v_{\calX^{\frakw_\ast}_\ast}$ can be rewritten as
\[
\frac{h_{\mu_2}(\sigma)}{\int_{\Sigma_\Cc} \psi \ d\mu_2} \ge \frac{h_{\mu_1}(\sigma)}{\int_{\Sigma_\Cc} \psi \ d\mu_1}
\]
where $h_{\mu_1}(\sigma), h_{\mu_2}(\sigma)$ are the entropies of $\mu_1, \mu_2$ over the component $\Cc$. This inequality is true by the variational principle. Furthermore this inequality is a strict equality unless $r$ and $\psi$ are cohomologous. This implies by Lemmas \ref{lem.psiell_ast} and \ref{lem.goodpath} that there exist $\Lambda, C >0$ such that 
\[
|\ell^{\frakw}_{\calX^\frakw}[g] - \Lambda \ell^{\frakw_\ast}_{\calX_\ast^{\frakw_\ast}}[g]| < C
\]
for all $[g] \in \conj(\G)$.
This can only happen if $(2)$ holds.
\end{proof}

\begin{proof}[Proof of Theorem \ref{thm.cubulation}]
Let $\psi=\psi_{\calX_\ast}^{\frakw_\ast}:\Sig\ra \Z$ be the potential associated to the constant orthotope structure $\frakw_\ast \equiv 1$. Let $\calL:[\Dil(\calX, \calX_\ast)^{-1}, \Dil(\calX_\ast, \calX) ]\ra \R$ be the Legendre transform of $\thet_{\calX_\ast/\calX}$, which by Remark \ref{rmk.pressure} equals the Legendre transform of $s \mapsto  \textnormal{P}_\Cc(-s\psi)$ for $\calC$ any maximal recurrent component of $\calG_\calX$. Hence $\calL$ is analytic. 

From our large devation principle in Corollary \ref{coro.ldmanhattan} we have that
\[
\limsup_{T\to\infty} \frac{1}{T} \log \left(\#\left\{ [g] \in \conj: \ell_\calX[g] < T, | \ell_{\calX_\ast}[g]  - \eta \ell_{\calX}[g] | < \frac{C}{T} \right\}\right) \le v_\calX - \calL(\eta)
\]
for all $\eta \in (\Dil(\calX, \calX_\ast)^{-1}, \Dil(\calX_\ast, \calX)).$ 

We now prove the lower bound. Fix a maximal component $\Cc$. By Lemma \ref{lem.psiell_ast} and Lemma \ref{lem.polyn} there exists a polynomial $Q$ such that for any $C >0$ and $\eta \in (\Dil(\calX, \calX_\ast)^{-1}, \Dil(\calX_\ast, \calX))$ 
\[
\# \left\{\o \in P_n(\Sigma_\Cc): \left|\frac{\p^n(\o)}{n} - \eta \right| < \frac{C}{n} \right\} \le Q(n) \#\left\{ [g] \in \conj: \ell_\calX[g] \le n, \left| \frac{\ell_{\calX_\ast}[g]}{\ell_\calX[g]}  - \eta \right| < \frac{C}{n} \right\}
\]
where $\p$ is the potential from Definition \ref{def.psicalX}. However, by Theorem \ref{thm.ldtsi} (and Remark \ref{rem.transitive} as $\Cc$ may only be transitive) we have that
\[
\limsup_{n\to\infty} \frac{1}{n} \log \left(\# \left\{\o \in P_n(\Sigma_\Cc): \left|\frac{\p^n(\o)}{n} - \eta \right| < \frac{C}{n} \right\} \right)= h - \calI(\eta)
\]
where $\calI$ is the Legendre transform of the map $s \mapsto  \textnormal{P}_\Cc(-s\psi)$ and $h$ is the topological entropy of the subshift $(\Sigma_\Cc,\sig)$. However, as we saw above, $\calI$ is precisely $\calL$ and further by Lemma \ref{lem.upperdensity} we have $h = v_\calX$. Hence we deduce that 
\[
\limsup_{T\to\infty} \frac{1}{T} \log \#\left\{ [g] \in \conj: \ell_\calX[g] < T, | \ell_{\calX_\ast}[g]  - \eta \ell_{\calX}[g] | < \frac{C}{T} \right\} \ge v_\calX - \calL(\eta)
\]
for each $\eta \in (\Dil(\calX, \calX_\ast)^{-1}, \Dil(\calX_\ast, \calX))$. We have shown that the limit supremum in the statement of the theorem is equal to $v_\calX - \calL$ as required.

To conclude the proof we need to explain the additional conditions mentioned in the theorem. In particular we need to show that
\[
 0 < v_\calX - \calL(\eta) \le v_\calX \text{ for all } \eta \in (\Dil(\calX, \calX_\ast)^{-1}, \Dil(\calX_\ast, \calX))
\]
and that the upper bound inequality is an equality if and only if $\eta = \tau(\calX_\ast/\calX).$ All of these properties follow from the
definition of $\calL$ and the fact that $s\mapsto \textnormal{P}_\Cc(-s\psi)$ is strictly convex. 
\end{proof}

\appendix
 
\section{convex-cocompact subgroups of cubulable relatively hyperbolic groups}\label{sec.appendix}

In this appendix we prove Proposition \ref{prop.criterioncvxcc}. First, we recall the statement.
\begin{proposition}\label{prop.appcriterioncvxcc}
    Let $\G$ be a relatively hyperbolic group acting properly and cocompactly on the $\CAT(0)$ cube complex $\calX$. Then the following are equivalent for a subgroup $H< \G$.
    \begin{enumerate}
        \item $H$ is convex-cocompact for the action on $\calX$.
        \item $H$ is relatively quasiconvex and $H\cap P$ is convex-cocompact for the action of $\G$ on $\calX$ for any maximal parabolic subgroup $P<\G$.    
    \end{enumerate}
\end{proposition}

\begin{proof}
Under these assumptions $\G$ is finitely generated, so fix $S\subset \G$ a finite symmetric generating set and a $\G$-equivariant quasi-isometry $\phi: \calX \ra \Cay(\G,S)$. We also fix a vertex $x_0\in \calX$ such that $\phi(x_0)=o$ is the identity element in $\G$. Let $\bbP$ be a complete collection of representatives of conjugacy classes of maximal parabolic subgroups in $\G$, and let $\calP=(\bigcup \bbP) \bs \{o\}$. We let $d_S$ denote the (graph) word metric on $\Cay(\G,S)$ and let $H<\G$ be any subgroup. 

If $H$ is convex-cocompact, then it is undistorted, hence finitely generated and relatively quasiconvex by \cite[Thm.~1.5]{hruska}. Also, any maximal parabolic subgroup $P<\G$ is convex-cocompact by \cite[Thm.~1.1]{sageev-wise}, and hence $H\cap P$ is also convex-cocompact by \cite[Lem.~2.14 \& Lem.~2.15]{reyes.cubrh}. This proves the implication $(1) \Rightarrow (2)$.

The implication $(2) \Rightarrow (1)$ is more involved, and for its proof we adopt the following convention. 
If $\gam'$ is a parameterized curve and $x=\gam'_{t^-},y=\gam'_{t^+}$ belong to $\gam'$ with $t^-\leq t^+$, then $\gam'|_{[x,y]}=\gam'|_{[y,x]}$ is a set of points of form $\gam'_t$, with $t^-\leq t \leq t^+$ (if there is more than one option for $t^\pm$, we consider any of them).

By \cite[Lem.~4.3]{genevois} it is enough to prove the following: there exists $K>0$ such that if $\ov\gam \subset \calX$ is a (continuous) combinatorial geodesic with endpoints in $H x_0$, then $\ov\gam \subset N_K(H x_0)$. 

To find such $K$, consider constants $L,C$ such that the image under $\phi$ of any combinatorial geodesic $\ov\gam$ in $\calX$ is at Hausdorff distance at most $L$ from an $L$-Lipschitz $(L,C)$-quasigeodesic $\gam:[0,\ell]\ra \Cay(\G,S)$ with same endpoints as $\phi(\ov\gam)$ (see e.g.~\cite[Prop.~8.3.4]{BBI}). 

Let $\ov\gam \subset \calX$ be a geodesic with endpoints in $Hx_0$ and let $\gam=\gam([0,\ell])\subset \Cay(\G,S)$ be as above, so that the endpoints of $\gam$ belong to $H$. Also, let $\hat c$ be a geodesic in $\Cay(\G,S\cup \calP)$ with same endpoints as $\gam$ and let
$\hat c_0,\dots, \hat c_n$ be the (ordered) vertex set of $\hat c$. We define $I=\{j_0<j_i<\cdots <j_k\}$ to be the set of all $0\leq i \leq n-1$ such that $\hat c_{i+1}^{-1}\hat c_i\in \calP$. 
    
    By quasiconvexity of $H$, there exists $\k$ (independent of $\gam$) and $h_i\in H$ such that $d_S(\hat c_i,h_i)\leq \k$ for all $0\leq i \leq n$, see e.g. \cite[Def.~6.10]{hruska}. Also, by \cite[Lem.~8.8]{hruska} there exists $A_0$ depending only on $L,C$ such that for any $0\leq i \leq n$ there exists $c_i=\gam_{t_i}\in \gam$ satisfying $d_S(c_i,\hat c_i)\leq A_0$, for which we assume $c_0=\hat c_0$ and $c_n=\hat c_n$. Up to increasing $A_0$ (only in terms of $L,C$), we can always assume that $c_i$ is a group element.  
    
    Since $\G$ is finitely generated, by \cite[Prop.~9.4]{hruska} there exists $B_0$ depending only on $A_0$ and $\k$ (so only on $L,C$) such that if $g_1,g_2\in \G$ satisfy $|g_1|_S,|g_2|_S\leq A_0+\k$, then 
    \begin{equation}\label{eq.interqc}
        N_{A_0+\k}(g_1H) \cap N_{A_0+\k}(g_2P)\subset N_{B_0}(g_1Hg_1^{-1}\cap g_2Pg_2^{-1})
    \end{equation}
    for any $P\in \bbP$, where the neighborhoods are considered in $\Cay(\G,S)$.

    Let $\wtilde p$ be a geodesic lift of $\hat c$ to $\Cay(\G,S)$. That is, $\wtilde p$ is obtained from $\hat c$ by replacing each edge corresponding to an element of $\calP$ by a geodesic in $\Cay(\G,S)$ with the same endpoints. For a point $x\in \gam$ we distinguish two cases.

    \textbf{Case 1:} $x\in \gam|_{[c_{{j_i}+1},c_{j_{i+1}}]}$ for some $j_i\in I$ (with the convention that $j_{-1}=-1$ and $j_{k+1}=n$). Consider geodesic paths $[c_{{j_i}+1},\hat c_{{j_i}+1}]$ and $[c_{j_{i+1}}, \hat c_{j_{i+1}}]$ in $\Cay(\G,S)$, and the quasigeodesic triangle with sides
    \[\ell_1=[\hat c_{{j_i}+1},c_{{j_i}+1}] \cup \gam|_{[c_{{j_i}+1},x]} , \ \ \ell_2= \gam|_{[x,c_{j_{i+1}}]}  \cup [c_{j_{i+1}}, \hat c_{j_{i+1}}],  \ \ \ell_3= \wtilde p|_{[c_{j_{i+1}}, c_{{j_i}+1}]}.\] 
    We also set 
    $$\ell_1^-=\hat c_{j_i+1}, \ell_1^+=x, \ \ \ell_2^-=\hat x,
      \ell_2^+=\hat c_{j_{i+1}}, \ \ \text{ and } \ \ \ell_3^-=\hat c_{j_{i+1}},  \ell_3^+=\hat c_{j_i+1}.$$
    Note that $\ell_1,\ell_2,\ell_3$ are Lipschitz quasigeodesics with constants depending only on $L,C$ and $A_0$ (hence only on $L,C$). 
    Then by \cite[Lem.~8.19]{drutu-sapir} there exists $R$ depending on $L,C$ such that either:
    \begin{itemize}
        \item there exists $z\in \Cay(\G,S)$ with $d_S(z,\ell_i)\leq R$ for $i=1,2,3$; or,
        \item there exist $g\in \G$ and $P\in \bbP$ such that $d_S(gP,\ell_i)\leq R$ for $i=1,2,3$.
    \end{itemize}
    In the first subcase, let $u_i\in \ell_i$ be such that $d_S(z,u_i)\leq R$. Then $d_S(u_a,u_b)\leq 2R$ for all $1\leq a,b\leq 3$, and since $\gam$ is $(L,C)$-quasigeodesic and $x\in \gam|_{[u_1,u_2]}$, we have that
    $d_S(x,\ell_3)\leq d_S(x,u_3)\leq d_S(x,u_1)+d_S(u_1,u_3)$ is bounded above in terms of $L,C$ and $R$ (thus only in terms of $L,C$). 

   In the second subcase, by \cite[Lem.~8.15]{drutu-sapir} we can find $M$ and $\mathfrak{d}$ depending only on $L,C$ and $R$ (so only on $L,C$) and points $u_i^-,u_i^+\in \ell_i$ for $i=1,2,3$ that satisfy:
   \begin{itemize}
       \item $d_S(u_i^\pm,gP) \leq M$; and,
       \item $\diam(\ell_i|_{[\ell_i^\pm,u_i^\pm]}\cap \ov N_M(gP))\leq \mathfrak{d}$.
   \end{itemize}
    Take $v_i^\pm\in gP$ such that $d_S(u_i^\pm,v_i^\pm)\leq M$. Then by the definition of $I$ 
    and after considering vertices in $\hat c$ that are closest to $u_3^\pm$ in $\Cay(\G,S)$ we get \[d_S(u_3^+,u_3^-)\leq d_{S\cup \calP}(v_3^-,v_3^+)+2(1+M)+2\leq 5+2M.  \]
    Also, \cite[Lem.~8.14]{drutu-sapir} implies the existence of $D_1$ depending only on $L,C,M$ and $\mathfrak{d}$ (so only on $L,C$) with 
    $d_S(u_i^+,u_{i+1}^-)\leq D_1$ (mod 3) for all $i$. In particular, $$d_S(u_1^-,u_2^+)\leq d_S(u_1^-,u_3^+)+d_S(u_3^+,u_3^-)+d_S(u_3^-,u_2^+),$$ and as in the first subcase we conclude that $x$ belongs to a neighborhood of $\ell_3$ depending only on $L$ and $C$.

    In both subcases, we deduce that $d_S(x,\ell_3)$ is bounded in terms of $L,C$, and since $\ell_3$ is contained in a neighborhood of $H$ depending only on $\k$, we have that 
    \begin{equation}\label{eq.case1}
        d_S(x,H)\leq K_0
    \end{equation} for some $K_0$ depending only on $L,C$ and $\k$ (hence only in terms of $L,C$). 

    \textbf{Case 2:} $x\in \gam|_{[c_{j},c_{j+1}]}$ for some $j\in I$. Suppose $\hat c_j^{-1}\hat c_{j+1}=p\in P$ for $P\in \bbP$. 
    Then 
    \begin{align*}
        d_S(c_j^{-1}c_{j+1},(c_j^{-1}h_j)H)  \leq d_S(c_j^{-1}c_{j+1},c_j^{-1}h_{j+1}) 
         = d_S(c_{j+1},h_{j+1})  \leq A_0+\k,
    \end{align*}
    and 
    \begin{align*}
    d_S(c_j^{-1}c_{j+1}, (c_j^{-1}\hat c_j)P)   \leq d_S(c_j^{-1}c_{j+1},c_j^{-1}\hat c_jp) 
     = d_S(c_{j+1},\hat c_{j+1})\leq A_0.
    \end{align*}
Since $\max\{|c_j^{-1}h_j|_S,|c_j^{-1}\hat c_j^{-1}|_S\}\leq A_0+\k$, by \eqref{eq.interqc} we conclude
\begin{equation}\label{eq.case2}
    d_S(c_j^{-1}c_{j+1}, (c_j^{-1}h_j)H(c_j^{-1}h_j)^{-1}\cap (c_j^{-1}\hat c_j)P(c_j^{-1}\hat c_j)^{-1})\leq B_0.
\end{equation}

Note that any point $x\in \gam$ satisfies the assumptions of one of the two cases above. Indeed, for $x=\gam_t\in \gam$, let $I_-$ be the set of all the $j\in I$ such that $c_j$ is not of the form $\gam_{t'}$ with $t'>t$. Suppose first that $I_-$ is non-empty and let $j$ be its maximal element. If $x$ does not satisfy Case 2, then $c_{{j}+1}$ does not belong to $\gam|_{[x,\gam_\ell]}$. But if $j=j_i<j_k$, then $j_{i+1}\notin I_-$, so that $x\in \gam|_{[c_{{j_i}+1},c_{j_{i+1}}]}$ and $x$ satisfies Case 1. 

Also, if $j=j_k$, then $x\in \gam|_{[c_{{j_i}+1},c_n]}$ and $x$ also satisfies Case 1. Therefore, we can assume that $I_-$ is empty. But if $I$ is non-empty then $x\in \gam|_{[c_0,c_{j_1}]}$ and $x$ satisfies Case 1, and if $I$ is empty then 
$x\in \gam|_{[c_0,c_{n}]}$ and $x$ also satisfies Case 1.

Now, take $\ov x\in \ov\gam$ and let $x\in \gam$ within $r$ from $\phi(\ov x)$ in $\Cay(\G,S)$, where $r$ is independent of $\ov x$ and $\ov\gam$. If $x$ satisfies Case 1, by \eqref{eq.case1} we conclude that $d_\calX(\ov x,H x_0)\leq K_1$ for $K_1$ a constant independent of $\ov x$ and $\ov \gam$. 

If $x$ satisfies Case 2, suppose that $x\in \gam|_{[c_j,c_{j+1}]}$ for $j\in I$. Then by \eqref{eq.case2} there exist vertices $\ov x^-,\ov x^+\subset \ov \gam$ satisfying $d_\calX(c_j x_0,\ov x^-)\leq \hat r$ and $d_\calX(c_{j+1} x_0,\ov x^+)\leq \hat r$, where $\hat r$ depends only on $\phi$ and $L,C$. 

Let $F$ be the set of pairs $\al,\beta\in \G$ satisfying $|\al|_S,|\beta|_S\leq A_0+\k$. By our assumption and \cite[Thm.~1.1]{sageev-wise} we can find a convex core $Z_{\al,\beta}\subset \calX$ for the group $\al H\al^{-1}\cap \beta P \beta^{-1}$ that contains the $\hat{r}$-neighborhood of $x_0$. By cocompactness, we can find $K_2>0$ such that \begin{equation*}
    Z_{\al,\beta}\subset N_{K_2}((\al H\al^{-1}\cap \beta P \beta^{-1})x_0)\subset N_{K_2}((\al H\al^{-1})x_0)
\end{equation*} for all $(\al,\beta)\in F$. Note that $K_2$ is independent of $\ov\gam$. In particular we have \begin{align*}
   \ov x \in c_j Z_{c_j^{-1}h_j,c_j^{-1}\hat c_j}\subset c_jN_{K_3}(c_j^{-1} H(h_j^{-1}c_j)x_0)\subset N_{K_2}(Hx_0),
\end{align*}
where $K_3:=K_2+\max\{d_\calX(\al x_0,x_0)\colon |\al|_S\leq A_0+\k\}$ is independent of $\ov x$ and $\ov \gam$. In conclusion, $\ov\gam\subset N_K(H x_0)$ for $K:=\max\{K_1,K_3\}$, and the implication $(2) \Rightarrow (1)$ follows.
\end{proof}


\subsection*{Open access statement}
For the purpose of open access, the authors have applied a Creative Commons Attribution (CC BY) licence to any Author Accepted Manuscript version arising from this submission.

\noindent\small{Department of Mathematics, University of Warwick, Coventry, CV4 7AL, UK}\\
\small{\textit{Email address}: \texttt{stephen.cantrell@warwick.ac.uk}\\
\\
\small{Max Planck Institute for Mathematics, Bonn, Germany, 53111}\\
\small{\textit{Email address}: \texttt{eoregon@mpim-bonn.mpg.de}}\\

\end{document}